\documentclass[a4paper, 10pt, notitlepage]{article}

\usepackage{amsthm, amsmath, amssymb, latexsym}
\usepackage{mathrsfs,cite}
\usepackage{graphicx}
\usepackage[ruled,vlined]{algorithm2e}

\theoremstyle{plain}
\newtheorem{theorem}{Theorem}
\newtheorem{lemma}{Lemma}
\newtheorem{corollary}{Corollary}
\newtheorem{proposition}{Proposition}

\theoremstyle{definition}
\newtheorem{definition}{Definition}
\newtheorem{example}{Example}

\theoremstyle{remark}
\newtheorem{remark}{Remark}

\DeclareMathOperator{\trace}{Tr}
\DeclareMathOperator{\co}{co}
\DeclareMathOperator{\interior}{int}
\DeclareMathOperator{\dist}{dist}
\DeclareMathOperator{\diag}{diag}
\DeclareMathOperator*{\minimize}{minimize}
\DeclareMathOperator*{\maximize}{maximize}

\author{M.V. Dolgopolik\footnote{Institute for Problems in Mechanical Engineering of the Russian Academy of Sciences,
Saint Petersburg, Russia}
\footnote{This work was performed in IPME RAS and supported by the Russian Science Foundation (Grant No. 20-71-10032).}}
\title{DC Semidefinite Programming and Cone Constrained DC Optimization: Theory and Local Search Methods}

\begin{document}

\maketitle

\begin{abstract}
In this paper, we study possible extensions of the main ideas and methods of constrained DC optimization to
the case of nonlinear semidefinite programming problems and more general nonlinear and nonsmooth cone constrained
optimization problems. 

In the first part of the paper, we analyse two different approaches to the definition of DC matrix-valued functions
(namely, order-theoretic and componentwise), study some properties of convex and DC matrix-valued mappings and
demonstrate how to compute DC decompositions of some nonlinear semidefinite constraints appearing in applications. We
also compute a DC decomposition of the maximal eigenvalue of a DC matrix-valued function. This DC decomposition can be
used to reformulate DC semidefinite constraints as DC inequality constrains. Finally, we study local optimality
conditions for general cone constrained DC optimization problems.

The second part of the paper is devoted to a detailed convergence analysis of two extensions of the well-known DCA
method for solving DC (Difference of Convex functions) optimization problems to the case of general cone constrained DC
optimization problems. We study the global convergence of the DCA for cone constrained problems and present a 
comprehensive analysis of a version of the DCA utilizing exact penalty functions. In particular, we study the exactness
property of the penalized convex subproblems and provide two types of sufficient conditions for the convergence of the
exact penalty method to a feasible and critical point of a cone constrained DC optimization problem from an infeasible
starting point. In the numerical section of this work, the exact penalty DCA is applied to the problem of computing
compressed modes for variational problems and the sphere packing problem on Grassmannian. 
\end{abstract}

\section{Introduction}

Starting with the pioneering works of Hiriart-Urruty \cite{HiriartUrruty85,HiriartUrruty89}, Pham Dinh and Souad
\cite{PhamDinh1986}, Strekalovsky \cite{Strekalovsky87}, Tuy \cite{Tuy86}, and many others in the 1980s, 
DC (Difference of Convex functions) programming has been an active area of research in nonlinear nonconvex
optimization. One of the main features of DC optimization problems is the fact that one can derive constructive global
optimality conditions \cite{Tuy2003,HiriartUrruty98,DurHorstLocatelli,Strekalovsky98,Zhang2013} and develop
deterministic global optimization methods
\cite{Tuy_book,HorstThoai,LeThiPhamDinh2002,Ferrer,Strekalovsky2021,Strekalovsky_Conf}
for this class of problems. Local search methods for minimizing DC functions have also attracted a considerable
attention of researchers (see \cite{TorBagKar,GaudiosoBagirov,JokiBagirov2020,Strekalovsky_Collect} and the references
therein). 

Perhaps, the most efficient and well-known numerical method for DC optimization problems is the so-called DCA, 
originally presented by Pham Dinh and Souad in \cite{PhamDinh1986} and later on thoroughly investigated in the works of
Le Thi and Pham Dinh et al. 
\cite{PhamDinhLeThi96,PhamDinhLeThi98,LeThiPhamDinh2014,LeThiPhamDinh2014b,LeThiPhamDinh2018} (a particular version of
the DCA is sometimes called \textit{the concave-convex/convex-concave procedure} \cite{Yuille,Lanckreit}). Some closely
related local search methods were studied in the works of de Oliveira et al.
\cite{AckooijDeOliveira2019,deOliveiraTcheou,deOliveira2019,AckooijDeOliveira2019b,deOliveira2021}.
For a detailed survey on DC programming, DCA, and their applications see
\cite{PhamDinhLeThi97,LeThiPhamDinh2005,LeThiDinh2018}. A comprehensive literature review of the DCA, the convex-concave
procedure, and other related optimization methods can be found in \cite{LippBoyd}.

Cone constrained optimization is one the central areas of constrained optimization, since it provides a unified setting
for many different problems appearing in applications. Standard equality and inequality constrained problems,
semidefinite programming problems \cite{KocvaraStingl,Stingl,Todd2001}, second order cone programming problems
\cite{AlizadehGoldfarb}, semi-infinite programming problems \cite{ReemtsenRuckmann,GobernaLopez}, and many other
particular problems (see, e.g. \cite{BoydVandenberghe,NesterovNemirovski,BenTalNemirovski}) can be formulated as general
cone constrained optimization problems. 

A detailed theoretical analysis of smooth and nonsmooth cone constrained optimization problems was presented in
\cite{BonnansShapiro,MordukhovichNghia,Tung,ZhenYang2007,Kanzi2011,Gadhi}. Optimization methods for solving various
\textit{convex} cone constrained optimization problems can be found in
\cite{BoydVandenberghe,NesterovNemirovski,BenTalNemirovski}, while algorithms for solving various classes of smooth
nonconvex cone constrained optimization problems were developed, e.g. in
\cite{KocvaraStingl,Stingl,YamashitaYabe,KatoFukushima,YamashitaYabe2015,CanelasCarrasco} (see also the references 
therein).

Despite the abundance of publications on cone constrained optimization and (usually inequality) constrained DC
optimization problems, very little attention has been paid to extensions of the main results and methods of DC
optimization to the case of problems with cone constraints. Even in the comprehensive survey paper \cite{LeThiDinh2018},
only unconstrained and inequality constrained DC optimization problems are discussed.

The convex-concave procedure and the penalty convex-concave procedure for solving cone constrained DC optimization
problems were proposed by Lipp and Boyd in \cite{LippBoyd}, where an application of these methods to multi-matrix
principal component analysis was presented. However, to the best of the author's knowledge, a convergence analysis of
these methods remains an open problem. An application of the DCA to bilinear and quadratic matrix inequality feasibility
problems was considered by Niu and Dinh \cite{NiuDinh2014}. Finally, optimality conditions for DC semi-infinite
programming problems were studied in the recent paper \cite{CorreaLopezPerezAros}.

The main goal of this paper is to fill in the gap and extend some of the main results and methods of inequality
constrained DC optimization (such as the DCA) to the case of DC optimization problems with DC cone constraints,
particularly, DC semidefinite programming problems. The motivation behind this extension is connected to the fact that
the DC optimization approach allows one to develop general methods for solving \textit{nonsmooth} cone constrained
optimization problems, as well as extend global DC optimization methods to the case of such problems. Furthermore, 
the nonlocal nature of the DCA (the method uses global majorants of the objective function and constraints) in some
cases allows this algorithm to find better local solutions than traditional optimization methods. This peculiarity makes
the DCA a potentially appealing alternative to existing methods for solving cone constrained optimization problems 
(see \cite{LippBoyd} and Section~\ref{sect:NumerExperiments} for some promising results of numerical experiments).

In the first part of the paper, we present a detailed discussion of two different approaches to the definition of DC
matrix-valued mappings: order-theoretic and componentwise. We obtain several useful properties of convex and DC
matrix-valued functions, prove that any DC (in the order-theoretic sense) matrix-valued map is necessarily componentwise
DC, and demonstrate how one can compute DC decompositions of several nonlinear matrix-valued functions appearing in
applications. We also construct a DC decomposition of the maximal eigenvalue of a componentwise DC matrix-valued
mapping. This result allows one to easily extend all ideas and methods of inequality constrained DC optimization to
the case of DC optimization problems with componentwise DC semidefinite constraints. Finally, we also derive local
optimality conditions for general cone constrained DC optimization problems in several different forms.

In the second part of the paper, we study convergence of the DCA for solving cone constrained DC optimization problems
and its exact penalty version from \cite{LippBoyd}. Namely, we provide sufficient conditions for the convergence of both
methods to a critical point of the problem under consideration. We also study the exact penalty property of the
penalized convex subproblems of the second method and obtain two types of sufficient conditions for the convergence of
this method to a feasible critical point of a cone constrained DC optimization problem from an infeasible starting
point. In the end of the paper, we apply the exact penalty DCA to the problem of computing compressed modes for
variational problems \cite{OzolinsLaiCaflischOsher} and the sphere packing problem on Grassmannian
\cite{AbsilHosseini,DirrHelmkeLageman}, that has applications to multi-antenna channel communications
\cite{GoharyDavidson,ZhengTse}. We present some interesting results of numerical experiments that support our
theoretical observations on penalty updating rules and the overall performance of the method. For an interesting
application of the algorithms discussed in this paper to multi-matrix principal component analysis see \cite{LippBoyd}.

The paper is organized as follows. Order-theoretic and componentwise approaches to DC matrix valued functions are 
studied in Section~\ref{sect:DC_MatrixValued}, while a DC structure of the maximal eigenvalue of a nonlinear
matrix-valued mapping is discussed in Section~\ref{sect:MaximalEigenvalue}. Section~\ref{sect:ConeConstrainedDC} is
devoted to the derivation of local optimality conditions for general cone constrained DC optimization problems.
A convergence analysis of the DCA for cone constrained DC optimization problems is presented in Section~\ref{sect:DCA},
while a detailed analysis of the penalty convex-concave procedure (exact penal\-ty DCA) from \cite{LippBoyd} is given in
Section~\ref{sect:PenalizedCCP}. Some results of numerical experiments are contained in
Section~\ref{sect:NumerExperiments}. Finally, a primal-dual version of the exact penalty DCA is briefly discussed in the
appendix.

\section{Two Approaches to DC Matrix-Valued Functions}
\label{sect:DC_MatrixValued}

Denote by $\mathbb{S}^{\ell}$ the space of all real symmetric matrices of order $\ell \in \mathbb{N}$, and let
$\preceq$ be the L\"{o}ewner partial order on $\mathbb{S}^{\ell}$, i.e. $A \preceq B$ for some matrices 
$A, B \in \mathbb{S}^{\ell}$ if and only if the matrix $B - A$ is positive semidefinite. Nonlinear semidefinite
optimization is concerned with problems of minimizing functions subject to constraints of the form $F(x) \preceq 0$,
where $F \colon \mathbb{R}^d \to \mathbb{S}^{\ell}$ is a given nonlinear mapping. To extend the main ideas and results
of DC optimization to the case of nonlinear semidefinite programming problems, first one must introduce a suitable
definition of a DC matrix-valued mapping $F$. There are two possible approaches to this definition: order-theoretic and
componentwise. Let us discuss and compare these approaches.

Recall that the matrix-valued function $F$ is called \textit{convex} (see, e.g. \cite[Sect.~5.3.2]{BonnansShapiro}
and \cite[Sect.~3.6.2]{BoydVandenberghe}), if 
\[
  F(\alpha x_1 + (1 - \alpha) x_2) \preceq \alpha F(x_1) + (1 - \alpha) F(x_2)
  \quad \forall x_1, x_2 \in \mathbb{R}^d, \: \alpha \in [0, 1].
\]
Therefore it is natural to call the function $F$ DC (\textit{Difference-of-Convex}), if there exist convex mappings
$G, H \colon \mathbb{R}^d \to \mathbb{S}^{\ell}$ such that $F = G - H$. Any such representation of the function $F$
(or, equivalently, any such pair of functions $(G, H)$) is called a \textit{DC decomposition} of $F$. 

The definition of matrix-valued DC mapping given above has several disadvantages. Firstly, the convexity of
matrix-valued functions is much harder to verify than the convexity of real-valued functions. Many matrix-valued
mappings that might seem to be convex judging by the experience with the real-valued case are, in actuality, nonconvex.
In particular, the convexity of each component $F_{ij}(\cdot)$ of $F$ is not sufficient to ensure the matrix convexity
of $F$.

\begin{example} \label{ex:NonconvexMatrixValuedFunc}
Let $d = 1$, $\ell = 2$, and $F(x) = \left( \begin{smallmatrix} 1 & x^2 \\ x^2 & 1 \end{smallmatrix} \right)$. Then for
$x_1 = 1$ and $x_2 = -1$ one has
\[
  \alpha F(x_1) + (1 - \alpha) F(x_2) - F(\alpha x_1 + (1 - \alpha) x_2)
  = \left( \begin{smallmatrix} 0 & 1 - (2\alpha - 1)^2 \\ 1 - (2\alpha - 1)^2 & 0 \end{smallmatrix} \right).
\]
This matrix is \textit{not} positive semidefinite for any $\alpha \in (0, 1)$, which implies that the map $F$ is
nonconvex.
\end{example}

Secondly, recall that the set $\mathbb{S}^{\ell}$ equipped with the L\"{o}ewner partial order is \textit{not} a vector
lattice, since by Kadison's theorem \cite{Kadison} the least upper bound (the supremum) of two matrices in 
the L\"{o}ewner order exists if and only if these matrices are comparable. Therefore, many standard results and
techniques from convex analysis do not admit a direct extension to the case of matrix convexity (cf. the general theory
of convex vector-valued maps \cite{Papageorgiou,Thera,KusraevKutateladze}, in which the assumption on the completeness
of partial order is often indispensable). For example, in most cases the supremum of two convex matrix-valued functions
is not correctly defined.

Nevertheless, there are some similarities between matrix-valued DC mappings and their real-valued counterparts. In
particular, one can construct a DC decomposition of a twice continuously differentiable matrix-valued map with bounded
Hessian in the same way one can construct DC decomposition of a twice continuously differentiable real-valued function.

Let $I_{\ell}$ be the identity matrix of order $\ell$. Denote by $| \cdot |$ the Euclidean norm, by 
$\langle \cdot, \cdot \rangle$ the inner product in $\mathbb{R}^k$, and by $\| A \|_F = \sqrt{\trace(A^T A)}$ the
Frobenius norm of a real matrix $A$, where $\trace(\cdot)$ is the trace of a square matrix.

\begin{theorem} \label{thrm:DCDecompMatrixValued}
Let a map $F \colon \mathbb{R}^d \to \mathbb{S}^{\ell}$ be twice continuously differentiable and suppose that
there exists $M > 0$ such that $\| \nabla^2 F_{ij}(x) \|_F \le M$ for all $i, j \in \{ 1, \ldots, \ell \}$. Then 
the mapping $F$ is DC and for any $\mu \ge \ell M$ both pairs $(G_k, H_k)$, $k \in \{ 1, 2 \}$, with
\[
  G_1(x) = F(x) + \frac{\mu}{2} |x|^2 I_{\ell}, \quad 
  H_1(x) = \frac{\mu}{2} |x|^2 I_{\ell}, \quad \forall x \in \mathbb{R}^d,
\]
and
\[
  G_2(x) = \frac{\mu}{2} |x|^2 I_{\ell}, \quad
  H_2(x) = \frac{\mu}{2} |x|^2 I_{\ell} - F(x) \quad \forall x \in \mathbb{R}^d,
\]
are DC decompositions of $F$.
\end{theorem}

\begin{proof}
Observe that by the definitions of matrix convexity and the L\"{o}ewner partial order, a mapping 
$G \colon \mathbb{R}^d \to \mathbb{S}^{\ell}$ is convex if and only if for any $z \in \mathbb{R}^{\ell}$ one has
\[
  \langle z, \Big( \alpha G(x_1) + (1 - \alpha) G(x_2) - G(\alpha x_1 + (1 - \alpha) x_2) \Big) z \rangle \ge 0
\]
for all $x_1, x_2 \in \mathbb{R}^d$ and $\alpha \in [0, 1]$ or, equivalently, 
\[
  \langle z, G(\alpha x_1 + (1 - \alpha) x_2) z \rangle \le
  \alpha \langle z, G(x_1) z \rangle + (1 - \alpha) \langle z, G(x_2) z \rangle.
\]
Therefore, a map $G \colon \mathbb{R}^d \to \mathbb{S}^{\ell}$ is convex if and only if for any 
$z \in \mathbb{R}^{\ell}$ the real-valued function $G_z(\cdot) = \langle z, G(\cdot) z \rangle$ is convex. Consequently,
in the case when $G$ is twice continuously differentiable, this function is convex if and only if for any $z$ the
Hessian of the function $G_z$ is positive semidefinite, i.e. for all $x \in \mathbb{R}^d$ and $z \in \mathbb{R}^{\ell}$
the matrix
\[
  \nabla^2 G_z(x) = \sum_{i, j = 1}^{\ell} z_i z_j \nabla^2 G_{ij}(x)
\]
is positive semidefinite. 

Let us now turn to the proof of the theorem. Denote $G(x) = F(x) + \frac{\mu}{2} |x|^2 I_{\ell}$ and 
$H(x) = \frac{\mu}{2} |x|^2 I_{\ell}$. Let us check that the mappings $G$ and $H$ are convex, provided 
$\mu \ge \ell M$. Then one can conclude that $F$ is a DC function and the pair $(G_1, H_1)$ from the formulation of 
the theorem is a DC decomposition of $F$. The fact that the pair $(G_2, H_2)$ is also a DC decomposition of $F$ can be
proved in a similar way.

Indeed, for any $x \in \mathbb{R}^d$ and $v, z \in \mathbb{R}^{\ell}$ one has
\begin{equation} \label{eq:HessianFplusSquaredNorm}
  \langle v, \nabla^2 G_z(x) v \rangle = \sum_{i, j = 1}^{\ell} z_i z_j \langle v, \nabla^2 F_{ij}(x) v \rangle
  + \mu \sum_{i = 1}^{\ell} z_i^2 |v|^2
\end{equation}
Let us estimate the first term on the right-hand side of this equality. Indeed, applying the obvious inequality 
$2|z_i z_j| \le z_i^2 + z_j^2$, one gets 
\begin{align*}
  \sum_{i, j = 1}^{\ell} z_i z_j \langle v, \nabla^2 F_{ij}(x) v \rangle &\ge
  - \sum_{i, j = 1}^{\ell} |z_i z_j| \big| \langle v, \nabla^2 F_{ij}(x) v \rangle \big|
  \\
  &\ge - \frac{1}{2} \sum_{i, j = 1}^{\ell} (z_i^2 + z_j^2) \big| \langle v, \nabla^2 F_{ij}(x) v \rangle \big|.
\end{align*}
By the Cauchy-Bunyakovsky-Schwarz inequality and the fact that the Frobenius norm is compatible with the Euclidean norm
one has
\[
  \big| \langle v, \nabla^2 F_{ij}(x) v \rangle \big| \le |v| \Big| \nabla^2 F_{ij}(x) v \Big|
  \le \| \nabla^2 F_{ij}(x) \|_F |v|^2,
\]
which implies that
\begin{align*}
  \sum_{i, j = 1}^{\ell} z_i z_j \langle v, \nabla^2 F_{ij}(x) v \rangle
  &\ge - \frac{|v|^2}{2} \sum_{i, j = 1}^{\ell} (z_i^2 + z_j^2) \| \nabla^2 F_{ij}(x) \|_F
  \\
  &= - \frac{|v|^2}{2} \sum_{i = 1}^{\ell} 2 z_i^2 \Big( \sum_{j = 1}^{\ell} \| \nabla^2 F_{ij}(x) \|_F \Big),
\end{align*}
where the last equality follows from the fact that the matrix $F(x)$ is by definition symmetric. Combining this
inequality with \eqref{eq:HessianFplusSquaredNorm}, one finally obtains that
\[
  \langle v, \nabla^2 G_z(x) v \rangle \ge 
 |v|^2 \sum_{i = 1}^{\ell} \left( \mu - \sum_{j = 1}^{\ell} \| \nabla^2 F_{ij}(x) \| \right) z_i^2.
\]
Hence for any $x \in \mathbb{R}^d$, $z \in \mathbb{R}^{\ell}$, and $\mu \ge \ell M$ one has
\[
  \langle v, \nabla^2 G_z(x) v \rangle \ge 0 \quad \forall v \in \mathbb{R}^{\ell},
\]
that is, the Hessian $\nabla^2 G_z(x)$ is positive semidefinite. Thus, one can conclude that the matrix-valued mapping
$G(x) = F(x) + \frac{\mu}{2} |x|^2 I_{\ell}$ is convex. The convexity of $H$ can be readily verified directly.
\end{proof}

The difficulties connected with the use of matrix convexity motivate us to consider a different approach to the
definition of DC matrix-valued mappings.

\begin{definition}
A function $F \colon \mathbb{R}^d \to \mathbb{S}^{\ell}$ is called \textit{componentwise convex}, if each component
$F_{ij}(\cdot)$, $i, j \in \{ 1, \ldots, \ell \}$, is convex. The function $F$ is called \textit{componentwise DC}, if
there exist componentwise convex functions $G, H \colon \mathbb{R}^d \to \mathbb{S}^{\ell}$ such that $F = G - H$. Any
such representation of $F$ (or, equivalently, any such pair of functions $(G, H)$) is called 
\textit{a componentwise DC decomposition} of $F$.
\end{definition}

Many properties of real-valued DC functions can be easily extended to the case of componentwise DC matrix-valued 
mappings. For example, a linear combination of componentwise DC mappings is obviously componentwise DC. With the use
of the well-known results of Hartman \cite{Hartman}, one can easily see that the Hadamard and the Kronecker products of
componentwise DC matrix-valued mappings are componentwise DC, etc.

Let us point out some connections between convex/DC and componentwise convex/DC matrix-valued mappings. As
Example~\ref{ex:NonconvexMatrixValuedFunc} demonstrates, componentwise convex matrix-valued functions need not be
convex. On the other hand, from the fact that for any convex matrix-valued map $F$, the real-valued function 
$\langle z, F(\cdot) z \rangle$ is convex for all $z \in \mathbb{R}^{\ell}$ it follows that all diagonal components
$F_{ii}(\cdot)$ of a convex matrix-valued map $F$ must be convex (put $z = e_i$ for every vector $e_i$ from the
canonical basis of $\mathbb{R}^{\ell}$). However, non-diagonal components of $F$ need not be convex.

\begin{example}
Let $d = 1$, $\ell = 2$, and 
$F(x) = \left(\begin{smallmatrix} 0.5 x^2 & \sin x \\ \sin x & 0.5 x^2 \end{smallmatrix}\right)$. Then for all
$z \in \mathbb{R}^2$ and $x \in \mathbb{R}$ one has
\begin{align*}
  \sum_{i, j = 1}^2 z_i z_j \nabla^2 F_{ij}(x) = z_1^2 - 2 (\sin x) z_1 z_2 + z_2^2 
  &\ge z_1^2 - 2 |z_1| |z_2| + z_2^2 
  \\
  &= (|z_1| - |z_2|)^2 \ge 0.
\end{align*}
Consequently, the function $F$ is convex by \cite[Proposition~5.72, part~(ii)]{BonnansShapiro}, despite the fact that
non-diagonal elements of $F$ are nonconvex.
\end{example}

Although non-diagonal elements of a convex matrix-valued mapping $F$ might be nonconvex, they cannot be too `wild',
e.g. discontinuous. Namely, the following result holds true.

\begin{theorem} \label{thrm:NonDiagElements}
Let a map $F \colon \mathbb{R}^d \to \mathbb{S}^{\ell}$ be convex. Then for all $i, j \in \{ 1, \ldots, \ell \}$, 
$i \ne j$, the function $F_{ij}$ is DC and, therefore, Lipschitz continuous on any bounded set and twice differentiable
almost everywhere.
\end{theorem}

\begin{proof}
We prove the theorem by induction in $\ell$. The case $\ell = 1$ is trivial. Let us prove the case $\ell = 2$ in order
to highlight the main idea of the proof.

As was noted above, the function $\langle z, F(\cdot) z \rangle$ is convex for all $z \in \mathbb{R}^{\ell}$, which, in
particular, implies that the functions $F_{11}(\cdot)$ and $F_{22}(\cdot)$ are convex. For the vector $z = (1, 1)^T$ one
obtains that the function
\[
  F_z(x) = \langle z, F(x) z \rangle = F_{11}(x) + 2 F_{12}(x) + F_{22}(x), \quad x \in \mathbb{R}^d
\]
is convex as well. Therefore the function
\[
  F_{12}(x) = F_{21}(x) = \frac{1}{2} F_z(x) - \frac{1}{2}(F_{11}(x) + F_{22}(x))
\]
is DC, which completes the proof of the case $\ell = 2$.

\textit{Inductive step.} Suppose that the theorem is valid for some $\ell \in \mathbb{N}$. Let us prove it for 
$\ell + 1$. The function $F_z(\cdot) = \langle z, F(\cdot) z \rangle$ is convex for all $z \in \mathbb{R}^{\ell + 1}$.
Putting $z = (z_1, \ldots, z_{\ell}, 0)^T$ and $z = (0, z_2, \ldots, z_{\ell + 1})^T$ for any $z_i \in \mathbb{R}$, 
$i \in \{ 1, \ldots, \ell + 1 \}$ one obtains that the matrix-valued mappings
\[
  G(x) = \begin{pmatrix} F_{11}(x) & \dots & F_{1 \ell}(x) \\ \vdots & \vdots & \vdots \\ F_{\ell 1}(x) & \dots &
  F_{\ell \ell}(x) \end{pmatrix}, \quad
  H(x) = \begin{pmatrix} F_{22}(x) & \dots & F_{2 (\ell + 1)}(x) \\ \vdots & \vdots & \vdots \\ 
  F_{(\ell + 1) 2}(x) & \dots & F_{(\ell + 1) (\ell + 1)}(x) \end{pmatrix}
\]
are convex. Therefore, by the induction hypothesis all functions $F_{ij}$, $i, j \in \{ 1, \ldots, \ell + 1 \}$ are DC,
except for $F_{1 (\ell + 1)}$ (or, equivalently, $F_{(\ell + 1) 1}$, since $F(x)$ is by definition a symmetric matrix).

For $z = (1, \ldots, 1)^T$ one gets that the function
\[
  F_z(x) = \sum_{i, j = 1}^{\ell + 1} F_{ij}(x), \quad x \in \mathbb{R}^d
\]
is convex, which obviously implies that the function $F_{1 (\ell + 1)}$ is DC.

Finally, taking into account the fact that finite-valued convex functions are Lipschitz continuous on bounded sets
\cite[Thm.~10.4]{Rockafellar} and twice differentiable almost everywhere by the Busemann-Feller-Aleksandrov theorem
(see, e.g. \cite{BianchiColesantiPucci}), one can conclude that for all $i, j \in \{ 1, \ldots, \ell \}$ 
the functions $F_{ij}$ are Lipschitz continuous on bounded sets and twice differentiable almost everywhere.
\end{proof}

As simple corollaries to the previous theorem we obtain straightforward extensions of some well-known results for
real-valued convex function to the matrix-valued case.

\begin{corollary}
Let a map $F \colon \mathbb{R}^d \to \mathbb{S}^{\ell}$ be convex. Then $F$ is Lipschitz continuous on bounded
sets, i.e. for any bounded set $K \subset \mathbb{R}^d$ there exists $L > 0$ such that
$\| F(x_1) - F(x_2) \|_F \le L |x_1 - x_2|$ for all $x_1, x_2 \in K$.
\end{corollary}

\begin{corollary}[Busemann-Feller-Aleksandrov theorem for matrix-va\-lued functions]
Let a map $F \colon \mathbb{R}^d \to \mathbb{S}^{\ell}$ be convex. Then $F$ is twice differentiable almost everywhere.
\end{corollary}

\begin{remark}
Note that the statement of Theorem~\ref{thrm:NonDiagElements} is obviously true for locally convex (i.e. convex in a
neighbourhood of every point) matrix-valued mappings defined on not necessarily convex sets. Therefore, the previous
corollary remains true in this case as well. Namely, every locally convex map $F \colon U \to \mathbb{S}^{\ell}$ defined
on an open set $U \subset \mathbb{R}^d$ is twice differentiable almost everywhere on $U$.
\end{remark}

Since the difference of two real-valued DC functions is a DC function, Theorem~\ref{thrm:NonDiagElements} also allows
one to point out a direct connection between DC and componentwise DC matrix-valued mappings.

\begin{corollary}
Any DC map $F \colon \mathbb{R}^d \to \mathbb{S}^{\ell}$ is componentwise DC.
\end{corollary}

Since the definition of DC function provides a lot of flexibility (namely, there are infinitely many DC decompositions
of a given function), it seems reasonable to assume that despite some drawbacks of matrix convexity the class of
matrix-valued DC mappings is sufficiently rich. In particular, one might ask whether the class of matrix valued DC
functions coincides with the class of componentwise DC functions or there are some componentwise DC mappings that are
not DC (a characterization of such functions would provide a deep insight into the structure of DC matrix-valued
mappins). Another interesting question is whether the matrix DC property is preserved under standard operations, such
as the Hadamard/Kronecker product and inversion. Arguing in the same way as in the proof of
Theorem~\ref{thrm:DCDecompMatrixValued}, one can easily check that for twice continuously differentiable matrix-valued
mappings the answer to this question is positive, provided one considers \textit{locally} DC functions. However, it is
unclear whether the classes of locally and globally DC mappings coincide in the matrix-valued case (for
componentwise DC functions this statement is true due to the celebrated result of Hartman \cite{Hartman}). 

In the end of this section, let us present several simple examples of DC semidefinite constraints appearing in
applications and their DC decompositions. These examples, in particular, demonstrate some benefits of using
matrix-valued DC mappingss in comparison with componentwise DC mappings.

\begin{example}[Quadratic/Bilinear Constraints]
Suppose that
\begin{equation} \label{eq:QuadraticMatrixFunc}
  F(x) = C + \sum_{i = 1}^d x_i B_i + \sum_{i, j = 1}^d x_i x_j A_{ij}
\end{equation}
for some matrices $C, B_i, A_{ij} \in \mathbb{S}^{\ell}$. In particular, one can suppose that the map $F(x)$ is
bilinear/biaffine, that is,
\[
  F(x, y) = A_{00} + \sum_{i = 1}^d x_i A_{i0} + \sum_{j = 1}^m y_j A_{0j} + 
  \sum_{i = 1}^d \sum_{j = 1}^m x_i y_j A_{ij},
  \quad \forall x \in \mathbb{R}^d, \: y \in \mathbb{R}^m
\]
for some matrices $A_{ij} \in \mathbb{S}^{\ell}$. Such nonlinear matrix constraints appear in problems of simultaneous
stabilisation of single-input single-output linear systems by one fixed controller of a given order
\cite{HenrionTarbouriech,Stingl}, robust gain-scheduling and some decentralized control problems 
\cite{GohSafonov95,GohSafonov96}, problems of maximizing the minimal eigenfrequency of a given structure \cite{Stingl},
etc.

By Theorem~\ref{thrm:DCDecompMatrixValued} the map $F$ of the form \eqref{eq:QuadraticMatrixFunc} is DC and for
any $\mu \ge \ell M$, where
\[
  M^2 = \max_{s, k \in \{ 1, \ldots, \ell \}} \sum_{i, j = 1}^d [A_{ij}]_{sk}^2,
\] 
the pair
\[
  G(x) = C + \sum_{i = 1}^d x_i B_i + \sum_{i, j = 1}^d x_i x_j A_{ij} + \frac{\mu}{2} |x|^2 I_{\ell}, \quad
  H(x) = \frac{\mu}{2} |x|^2 I_{\ell}
\]
is a DC decomposition of $F$. Note that to compute a componentwise DC decomposition of $F$ one would have to compute DC
decompositions of $\ell^2$ quadratic functions of the form
\[
  \sum_{i, j = 1}^d [A_{ij}]_{sk} x_i x_j, \quad s, k \in \{ 1, \ldots, \ell \}.
\]
Moreover, in the general case the mapping $H$ (the concave part) from a componentwise DC decomposition of $F$ would not
be diagonal.

It should be noted that a different DC decomposition of the mapping $F$ can be constructed. Namely, as was shown in
\cite[Example~5.74]{BonnansShapiro}, a matrix-valued map $F$ of the form \eqref{eq:QuadraticMatrixFunc} is convex,
if the $\ell d \times \ell d$ block matrix $A = (A_{ij})_{i, j = 1}^d$ is positive semidefinite (note that replacing, if
necessary, $A$ with $0.5(A + A^T)$, one can suppose that the block matrix $A$ is symmetric). Therefore, if a
decomposition $A = A_+ + A_-$ of the matrix $A$ onto positive semidefinite and negative semidefinite parts is known,
one can define
\[
  G(x) = C + \sum_{i = 1}^d x_i B_i + \sum_{i, j = 1}^d x_i x_j (A_+)_{ij}, \quad
  H(x) = -\sum_{i, j = 1}^d x_i x_j (A_-)_{ij}
\]

Such DC decomposition can be used, if the block matrix $A$ has a relatively simple structure, e.g. when only the
diagonal blocks $A_{ii}$ are nonzero.
\end{example}

\begin{example}[Bilinear/Biaffine Matrix Constraints]
Consider the map
\[
  R(X_1, X_2, X_3) = \begin{bmatrix} X_1 & (A + B X_2 C)X_3 \\ X_3 (A + B X_2 C)^T & X_3 \end{bmatrix}
\]
for all $X_1, X_3 \in \mathbb{S}^{\ell}$, $X_2 \in \mathbb{R}^{m \times m}$, and for some matrices 
$A \in \mathbb{R}^{\ell \times \ell}$, $B \in \mathbb{R}^{\ell \times m}$, and $C \in \mathbb{R}^{m \times \ell}$.
Nonlinear semidefinite constraints involving such mappings $R$ (or similar ones) appear, e.g. in optimal
$\mathcal{H}_2$/$\mathcal{H}_{\infty}$-static output feedback problems \cite{Stingl,Compleib}.

To apply the results presented in this section to the mapping $R$, define 
$d = 0.5 \ell(\ell + 1) + m^2 + 0.5 \ell(\ell + 1)$ (here we used the fact that a matrix $X \in \mathbb{S}^{\ell}$ is
defined by $\ell(\ell + 1) / 2$ variables). For any $x \in \mathbb{R}^d$ let $(X_1, X_2, X_3)$ be the corresponding
triplet of matrices from $\mathbb{S}^{\ell} \times \mathbb{R}^{m \times m} \times \mathbb{S}^{\ell}$, and let
$F(x) = R(X_1, X_2, X_3)$. 

By Theorem~\ref{thrm:DCDecompMatrixValued} the map $F$ is DC and for any $\mu \ge \ell M$, where
\[
  M^2 = \max_{i \in \{ 1, \ldots, \ell \}} 
  \sum_{k_1 = 1}^m \sum_{k_2 = 1}^m \sum_{k_3 = 1}^{\ell} \big( B_{ik_1} C_{k_2 k_3} \big)^2,
\]
the pair
\[
  G(x) = F(x) + \frac{\mu}{2} \big( \| X_2 \|_F^2 + \| X_3 \|_F^2 \big) I_{2\ell},
  \quad
  H(x) = \frac{\mu}{2} \big( \| X_2 \|_F^2 + \| X_3 \|_F^2 \big) I_{2\ell}
\]
is a DC decomposition of $F$. 
\end{example}

\begin{example}[The Stiefel Manifold/Orthogonality Constraint] \label{ex:StiefelManifold}
Let $d = m \times \ell$ for some $m \in \mathbb{N}$, i.e. $x$ is a real matrix of order $m \times \ell$, which we denote
by $X$. Consider the equality constraint
\begin{equation} \label{eq:SteifelManifold}
  X^T X = I_{\ell},
\end{equation}
which is known as the Stiefel manifold or orthogonality constraint appearing in many applications
\cite{EdelmanThomas,Manton,AbsilMahony,LippBoyd}.

Following Lipp and Boyd \cite{LippBoyd}, we rewrite equality constraint \eqref{eq:SteifelManifold} as two matrix
inequality constraints:
\[
  G(X) = X^T X - I_{\ell} \preceq 0, \quad H(X) = I_{\ell} - X^T X \preceq 0.
\]
Let, as above, $G_z(X) = \langle z, G(X) z \rangle$. Observe that for any $X_1, X_2 \in \mathbb{R}^{m \times \ell}$ and
$\alpha \in [0, 1]$ one has
\begin{align*}
  \alpha G_z(X_1) &+ (1 - \alpha) G_z(X_2) - G_z(\alpha X_1 + (1 - \alpha) X_2) 
  \\
  &= (\alpha - \alpha^2) \langle z, X_1^T X_1 z \rangle 
  + \big( (1 - \alpha) - (1 - \alpha)^2 \big) \langle z, X_2^T X_2 z \rangle  
  \\
  &- \alpha (1 - \alpha) \langle z, (X_1^T X_2 + X_2^T X_1) z \rangle 
  \\
  &= \alpha (1 - \alpha) \Big( |X_1 z|^2 + |X_2 z|^2 - 2 \langle X_1 z, X_2 z \rangle \Big) 
  \\
  &= \alpha (1 - \alpha) \big| X_1 z - X_2 z \big|^2 \ge 0.
\end{align*}
Consequently, the function $G_z$ is convex for any $z \in \mathbb{R}^{\ell}$, which implies that the functions $G$ and
$-H$ are matrix convex. Thus, equality constraint \eqref{eq:SteifelManifold} can be rewritten as two DC semidefinite
constraints. It should be noted that although this transformation is degenerate (we rewrite an equality constraint as
two inequality constraints), numerical experiments reported in \cite{LippBoyd} demonstrate the effectiveness of an
optimization method based on such transformation.
\end{example}

\section{DC Structure of the Maximal Eigenvalue Function}
\label{sect:MaximalEigenvalue}

Since there is no obvious connection between componentwise convexity and the L\"{o}ewner partial order/matrix convexity,
componentwise DC matrix-valued mappings cannot be utilised directly in the abstract setting of nonlinear semi\-definite
programming problems. Instead, it is natural to apply componentwise DC property to a reformulation of such problems,
in which the semidefinite constraint $F(x) \preceq 0$ is replaced by the equivalent inequality constraint 
$\lambda_{\max}(F(x)) \le 0$, where $\lambda_{\max}(A)$ is the maximal eigenvalue of a symmetric matrix $A$. 

Our aim is to show that for componentwise DC mappings $F$ the inequality constraint $\lambda_{\max}(F(x)) \le 0$ is
also DC, and one can compute a DC decomposition of the maximal eigenvalue function $\lambda_{\max}(F(\cdot))$, if 
a componentwise DC decomposition of the map $F$ is known. With the use of this result one can easily extend
standard results and algorithms from the theory of DC constrained DC optimization problems to the case of 
DC semidefinite programming problems.

\begin{theorem} \label{thrm:MaxEigenvalue_DC}
Let $F \colon \mathbb{R}^d \to \mathbb{S}^{\ell}$ be a componentwise DC mapping and $F_{ij} = G_{ij} - H_{ij}$ be a DC
decomposition of each component of $F$, $i, j \in \{ 1, \ldots, \ell \}$. Then the function $\lambda_{\max}(F(\cdot))$
is DC and the pair $(g, h)$ with
\begin{equation} \label{eq:MaxEigenvalue_DCDecom}
\begin{split}
  g(x) &= \max_{|v| \le 1} \sum_{i, j = 1}^{\ell} \Big( (v_i v_j + 1) G_{ij}(x) + (1 - v_i v_j) H_{ij}(x) \Big), \\
  h(x) &= \sum_{i, j = 1}^{\ell} \Big( G_{ij}(x) + H_{ij}(x) \Big)
\end{split}
\end{equation}
for all $x \in \mathbb{R}^d$ is a DC decomposition of the function $\lambda_{\max}(F(\cdot))$.
\end{theorem}

\begin{proof}
Fix any $x \in \mathbb{R}^d$. As is well-known and easy to check, the following equality holds true:
\[
  \lambda_{\max}(F(x)) = \max_{|v| \le 1} \langle v, F(x) v \rangle 
  = \max_{|v| \le 1} \sum_{i, j = 1}^{\ell} v_i v_j F_{ij}(x).
\]
Adding and subtracting $G_{ij}(x) + H_{ij}(x)$ for all $i, j \in \{ 1, \ldots, \ell \}$ and taking into account 
the equality $F_{ij}(x) = G_{ij}(x) - H_{ij}(x)$, one obtains that
\begin{align*}
  \lambda_{\max}(F(x)) &= \max_{|v| \le 1} 
  \sum_{i, j = 1}^{\ell} \Big( (v_i v_j + 1) G_{ij}(x) + (1 - v_i v_j) H_{ij}(x) \Big) \\
  &- \sum_{i, j = 1}^{\ell} \Big( G_{ij}(x) + H_{ij}(x) \Big) =: g(x) - h(x).
\end{align*}
The function $h$ is obviously convex as the sum of convex functions. Moreover, note that $v_i v_j + 1 \ge 0$ and
$1 - v_i v_j \ge 0$ for all $|v| \le 1$. Therefore, the function $g$ is also convex as the maximum of the family of
convex functions
\[
  (v_i v_j + 1) G_{ij}(x) + (1 - v_i v_j) H_{ij}(x), \quad |v| \le 1.
\]
Thus, the function $\lambda_{\max}(F(\cdot))$ is DC and the pair $(g, h)$ defined in \eqref{eq:MaxEigenvalue_DCDecom}
is a DC decomposition of this function.
\end{proof}

\begin{remark}
Let us make an almost trivial, yet useful observation. By definition $g(x) = \lambda_{\max}(F(x)) + h(x)$. Therefore,
there is no need to directly compute the maximum in the definition of $g$ in order to compute $g(x)$. One simply has to
to find the maximal eigenvalue of the matrix $F(x)$ and then add $h(x)$.
\end{remark}

For the sake of completeness, let us point out explicit formulae for the subdifferentials of the convex functions $g$
and $h$ from the theorem above. To this end, for any matrix $A \in \mathbb{S}^{\ell}$ denote by $\mathcal{E}_{\max}(A)$
the eigenspace of $\lambda_{\max}(A)$. 

\begin{proposition} \label{prp:MaxEigenvalue_DCQuasidiff}
Under the assumptions of Theorem~\ref{thrm:MaxEigenvalue_DC} for any $x \in \mathbb{R}^d$ one has 
\begin{align*}
  \partial g(x) = \co\Big\{ &\sum_{i, j = 1}^{\ell} 
  \Big( (v_i v_j + 1) \partial G_{ij}(x) + (1 - v_i v_j) \partial H_{ij}(x) \Big) \Bigm|
  \\
  &v \in \mathcal{E}_{\max}(F(x)), \enspace |v| = 1 \Big\}
\end{align*}
and $\partial h(x) = \sum_{i, j = 1}^{\ell} (\partial G_{ij}(x) + \partial H_{ij}(x))$, where `$\co$' stands for the
convex hull.
\end{proposition}

\begin{proof}
The expression for $\partial h(x)$ follows directly from the standard rules of subdifferential calculus. Let us prove
the equality for $\partial g(x)$.

Indeed, fix any $x \in \mathbb{R}^d$ and denote by $V(F(x))$ the set of all those $v \in \mathbb{R}^{\ell}$ with 
$|v| \le 1$ for which the maximum in the definition of $g(x)$ is attained. Clearly, $V(F(x))$ is a compact set. With the
use of the theorem on the subdifferential of the supremum of an infinite family of convex functions (see, e.g.
\cite[Thm.~4.2.3]{IoffeTihomirov}) one obtains that
\[
  \partial g(x) = \co\Big\{ \sum_{i, j = 1}^{\ell} (v_i v_j + 1) \partial G_{ij}(x) + (1 - v_i v_j) \partial H_{ij}(x) 
  \Bigm| v \in V(F(x)) \Big\}.
\]
Note that this convex hull is closed as the convex hull of a compact set. Therefore, it remains to show that 
$v \in V(F(x))$ if and only if $v \in \mathcal{E}_{\max}(F(x))$ and $|v| = 1$.

Observe that for any $v \in \mathbb{R}^{\ell}$ one has
\begin{multline*}
  \sum_{i, j = 1}^{\ell} \Big( (v_i v_j + 1) G_{ij}(x) + (1 - v_i v_j) H_{ij}(x) \Big) \\
  = \sum_{i, j = 1}^{\ell} v_i v_j \big( G_{ij}(x) - H_{ij}(x) \big) + h(x) 
  = \langle v, F(x) v \rangle + h(x).
\end{multline*}
Therefore, the maximum over all $v \in \mathbb{R}^{\ell}$ with $|v| \le 1$ of the left-hand side of this equality (which
is equal to $g(x)$) is attained at exactly the same $v$ as the maximum over all $v \in \mathbb{R}^{\ell}$ with 
$|v| \le 1$ of the right-hand side of this equality (which is equal to $\lambda_{\max}(F(x)) + h(x)$). Consequently,
one has
\[
  V(F(x)) = \Big\{ v \in \mathbb{R}^{\ell} \Bigm| |v| \le 1, \: \langle v, F(x) v \rangle = \lambda_{\max}(F(x)) \Big\}.
\]
With the use of the spectral decomposition of the matrix $F(x)$ one can easily verify that 
$\lambda_{\max}(F(x)) = \langle v, F(x) v \rangle$ for some $|v| \le 1$ if and only if $|v| = 1$ and $v$ is an
eigenvector of the matrix $F(x)$ corresponding to its maximal eigenvalue (i.e. $v \in \mathcal{E}_{\max}(F(x))$), which
implies the required result.  
\end{proof}

Thus, if an eigenvector $v$ with $|v| = 1$ of the matrix $F(x)$ corresponding to the maximal eigenvalue
$\lambda_{\max}(F(x))$ is computed, one can easily compute subgradients of DC components of the function
$\lambda_{\max}(F(\cdot))$ at the point $x$ with the use of subgradients of the functions $G_{ij}$ and $H_{ij}$.

\begin{remark}
Let us note once again that one can rewrite nonlinear semidefinite programming problem
\[
  \minimize \enspace f_0(x) \quad \text{subject to} \quad F(x) \preceq 0, \quad x \in Q.	
\]
where $Q$ is a closed convex set, as the following equivalent inequality constrained problem
\begin{equation} \label{prob:SemidefProg}
  \minimize \enspace f_0(x) \quad \text{subject to} \quad \lambda_{\max}(F(x)) \le 0, \quad x \in Q.
\end{equation}
In the case when the function $f_0$ is DC and the map $F$ is componentwise DC, one can easily extend all existing 
results and methods for inequality constrained DC optimization problems to the case of problem \eqref{prob:SemidefProg}
with use of Theorem~\ref{thrm:MaxEigenvalue_DC} and Proposition~\ref{prp:MaxEigenvalue_DCQuasidiff}. For the sake of
shortness, we leave the tedious task of explicitly reformulating existing results and methods in terms of problem
\eqref{prob:SemidefProg} to the interested reader.
\end{remark}

\section{Cone Constrained DC Optimization}
\label{sect:ConeConstrainedDC}

In the previous section, we pointed out how methods and results of DC optimization can be applied to nonlinear
semidefinite optimization problems with componentwise DC constraints. Let us now show how one can extend standard
results from DC optimization to the case when the semidefinite constraint is DC in the order-theoretic sense. Since such
extension does not rely on any particular properties of semidefinite problems (i.e. any properties of matrix-valued
mappings, the L\"{o}wner partial order, etc.) or the finite dimensional nature of the problem, following Lipp and Boyd
\cite{LippBoyd}, below we study optimality conditions for DC semidefinite programming problems in the more general
setting of DC cone constrained problems of the form
\begin{equation*}
\begin{split}
  &\minimize \enspace f_0(x) = g_0(x) - h_0(x), \\
  &\text{subject to} \enspace F(x) = G(x) - H(x) \preceq_K 0, \quad x \in Q.
\end{split}
  \qquad \qquad (\mathcal{P})
\end{equation*}
Here $g_0, h_0$ are real-valued closed convex functions defined on $\mathbb{R}^d$, $K$ is a proper cone
in a real Banach space $Y$ (that is, $K$ is a closed convex cone such that $K \cap (-K) = \{ 0 \}$), $\preceq_K$ is the
partial order induced by the cone $K$, i.e. $x \preceq_K y$ if and only if $y - x \in K$, the mappings 
$G, H \colon \mathbb{R}^d \to Y$ are convex with respect to the cone $K$ (or $K$-convex), that is,
\[
  G(\alpha x_1 + (1 - \alpha) x_2) \preceq_K \alpha G(x_1) + (1 - \alpha) G(x_2)
  \quad \forall \alpha \in [0, 1], \: x_1, x_2 \in \mathbb{R}^d
\]
and the same inequality holds for $H$, and, finally, $Q \subseteq \mathbb{R}^d$ is a closed convex set. Note that the
constraint $F(x) \preceq_K 0$ can be rewritten as $F(x) \in - K$.

Thus, the problem $(\mathcal{P})$ is a cone constrained DC optimization problem that consists in minimizing
the DC objective function $f_0$ subject to the generalized inequality (or cone) constraint that is DC with respect to
the cone $K$. In the case when $Y = \mathbb{S}^{\ell}$ and $K$ is the cone of positive semidefinite matrices, 
the problem $(\mathcal{P})$ becomes a standard nonlinear semidefinite programming problem.

\subsection{Some Properties of Convex Mappings}
\label{sect:ConvexMappings}

Before we proceed to the study of cone constrained DC optimization problems, let us first present two well-known
auxiliary results on convex mappings and convex multifunctions, whose formulations are tailored to our specific setting.
For the sake of completeness, we provide detailed proofs of these results.

We start with the following well-known characterisation of $K$-convex mappings in terms of their derivatives.

\begin{lemma} \label{lem:GenConvexityViaDerivative}
Let $X$ be a real Banach space. A G\^{a}teaux differentiable mapping $\Phi \colon X \to Y$ is $K$-convex if and only if
\begin{equation} \label{eq:GenConvexityViaDerivative}
  \Phi(x_1) - \Phi(x_2) \succeq_K \Phi'(x_2)(x_1 - x_2) \quad \forall x_1, x_2 \in X,
\end{equation}
where $\Phi'(x)$ is the G\^{a}teaux derivative of $\Phi$ at $x$.
\end{lemma}

\begin{proof}
Let $\Phi$ be convex. Then by definition
\[
  \alpha \Phi(x_1) + (1 - \alpha) \Phi(x_2) - \Phi(\alpha x_1 + (1 - \alpha) x_2) \in K
  \quad \forall \alpha \in [0, 1], \: x_1, x_2 \in X.
\]
Since $K$ is a cone, for any $\alpha \in (0, 1]$ one has
\[
  \Phi(x_1) - \Phi(x_2) - \frac{1}{\alpha} \big( \Phi(x_2 + \alpha (x_1 - x_2)) - \Phi(x_2) \big) \in K.
\]
Passing to the limit as $\alpha \to + 0$ and taking into account the fact that the cone $K$ is closed, one obtains that
\[
  \Phi(x_1) - \Phi(x_2) - \Phi'(x_2)(x_1 - x_2) \in K \quad \forall x_1, x_2 \in X
\]
or, equivalently, condition \eqref{eq:GenConvexityViaDerivative} holds true.

Conversely, if condition \eqref{eq:GenConvexityViaDerivative} holds true then for all $x_1, x_2 \in X$ and
for any $\alpha \in [0, 1]$ one has
\begin{align*}
  \Phi(x_1) - \Phi(x(\alpha)) - (1 - \alpha) \Phi'(x(\alpha))(x_1 - x_2) &\in K, 
  \\
  \Phi(x_2) - \Phi(x(\alpha)) - \alpha \Phi'(x(\alpha))(x_2 - x_1) &\in K.
\end{align*}
where $x(\alpha) = \alpha x_1 + (1 - \alpha) x_2$. Multiplying the first expression by $\alpha$ and the second
expression by $1 - \alpha$ and bearing in mind the fact that a convex cone is closed under addition, one obtains that
\[
  \alpha \Phi(x_1) + (1 - \alpha) \Phi(x_2) - \Phi(x(\alpha)) \in K
  \quad \forall \alpha \in [0, 1], \: x_1, x_2 \in X,
\]
that is, $\Phi$ is $K$-convex.  
\end{proof}

Let us also present a lemma on solutions of perturbed convex generalized equations, based on some well-known results on
metric regularity of convex multifunctions (see, e.g. \cite{Robinson76}). For any metric space $(X, \rho)$ and all 
$x \in X$ denote $B(x, r) = \{ x' \in X \mid \rho(x', x) \le r \}$. If $X$ is a normed space, then $B_X = B(0, 1)$.

\begin{lemma} \label{lem:PerturbedSolution}
Let $X$ be a real Banach space and $Z$ be a metric space. Suppose that $M_z \colon X \rightrightarrows Y$, $z \in Z$, is
a family of closed convex multifunctions such that for some $z_* \in Z$ and $\overline{x} \in X$ one has 
$0 \in \interior M_{z_*}(X)$ and $0 \in M_{z_*}(\overline{x})$. Suppose also that the function 
$z \mapsto \dist(0, M_z(\overline{x}))$ is continuous at $z_*$ and for any $\varepsilon > 0$ there exists $\delta > 0$
such that
\[
  M_{z_*}(\overline{x} + B_X) \subseteq M_z(\overline{x} + B_X) + \varepsilon B_Y
  \quad \forall z \in B(z_*, \delta).
\]
Then there exist a neighbourhood $U$ of $z_*$ and a mapping $\xi \colon \mathcal{U} \to X$ such that
$0 \in M_z(\xi(z))$ for all $z \in U$, $\xi(z_*) = \overline{x}$, and $\xi(z) \to \overline{x}$ as $z \to z_*$.
\end{lemma}

\begin{proof}
Since $0 \in \interior M_{z_*}(X)$ and $0 \in M_{z_*}(\overline{x})$, by \cite[Thm.~1]{Robinson76} there exists 
$\eta > 0$ such that $\eta B_Y \subseteq M_{z_*}(\overline{x} + B_X)$. By our assumption there exists $\delta > 0$ such
that
\[
  \eta B_Y \subseteq M_{z_*}(\overline{x} + B_X) \subseteq M_z(\overline{x} + B_X) + \frac{\eta}{3} B_Y
  \quad \forall z \in B(z_*, \delta),
\]
which by \cite[Thm.~2]{Robinson76} implies that for all $x \in X$ and $z \in B(z_*, \delta)$ one has
\[
  \dist(x, M_z^{-1}(0)) \le \frac{2}{\eta} \big( 1 + \| x - \overline{x} \| \big) \dist(0, M_z(x)).
\]
Putting $x = \overline{x}$ one gets that for any $z \in B(z_*, \delta)$ there exists $\xi(z) \in M_z^{-1}(0)$ such
that $\| \overline{x} - \xi(z) \| \le (4 / \eta) \dist(0, M_z(\overline{x}))$. Note that $\xi(z_*) = \overline{x}$,
since $0 \in M_{z_*}(\overline{x})$. Moreover, from the fact that the function $z \mapsto \dist(0, M_z(\overline{x}))$
is continuous at $z_*$ it follows that $\xi(z) \to \overline{x}$ as $z \to z_*$, which completes the proof.  
\end{proof}

\begin{remark}
Roughly speaking, the previous lemma states that if $0 \in \interior M_{z_*}(X)$ and $0 \in M_{z_*}(\overline{x})$, then
under certain semicontinuity assumptions for any $z$ in a neighbourhood of $z_*$ there exists a solution
$\xi(z)$ of the generalized equation $0 \in M_z(x)$ continuously depending on $z$ and such that 
$\xi(z_*) = \overline{x}$.
\end{remark}

\begin{corollary} \label{crlr:PerturbedSolution}
Let $X$ be a real Banach space, $W \subseteq X$ be a closed convex set, $E \subseteq Y$ be a proper cone, and 
$\Phi, \Psi \colon X \to Y$ be $E$-convex mappings. Suppose that $\Phi$ is continuous on $W$, $\Psi$ is continuously
Fr\'{e}chet differentiable on $W$, and the following constraint qualification holds true
\begin{equation} \label{eq:RegularPointConeConstr}
  0 \in \interior\Big\{ \Phi(x) - \Psi(x_*) - D \Psi(x_*)(x - x_*) + E \Bigm| x \in W \Big\}
\end{equation}
for some $x_* \in W$ such that $\Phi(x_*) - \Psi(x_*) \preceq_E 0$, where $D \Psi(x_*)$ is the Fr\'{e}chet derivative
of $\Psi$ at $x_*$. Then for any $\overline{x} \in W$ such that
\[
  \Phi(\overline{x}) - \Psi(x_*) - D \Psi(x_*)(\overline{x} - x_*) \preceq_E 0
\]
there exists a neighbourhood $\mathcal{U}$ of $x_*$ and a mapping $\xi \colon \mathcal{U} \cap W \to W$ such that
\[
  \Phi(\xi(z)) - \Psi(z) - D \Psi(z)(\xi(z) - z) \preceq_E 0 
  \quad \forall z \in \mathcal{U} \cap W
\]
$\xi(x_*) = \overline{x}$, and $\xi(z) \to \overline{x}$ as $z \to x_*$.
\end{corollary}

\begin{proof}
For any $z \in X$ introduce the $E$-convex function $\Phi_z \colon X \to Y$ defined as 
$\Phi_z(x) = \Phi(x) - \Psi(z) - D \Psi(z)(x - z)$ and the set-valued mapping
\begin{equation} \label{eq:ConeConstrMultifunc}
  M_z(x) = \begin{cases}
    \Phi_z(x) + E, & \text{if } x \in W, \\
    \emptyset, & \text{if } x \notin W.
  \end{cases}
\end{equation}
The multifunction $M_z$ is closed due to the facts that the mapping $\Phi_z(\cdot)$ is continuous and the sets
$W$ and $E$ are closed. Moreover, this multifunction is convex. 

Indeed, by the convexity of $\Phi$ for any $x_1, x_2 \in W$ and all $\alpha \in [0, 1]$ one has
\[
  \alpha \Phi_z(x_1) + (1 - \alpha) \Phi_z(x_2) \in \Phi_z(\alpha x_1 + (1 - \alpha) x_2)) + E,
\]
which due to the convexity of the cone $E$ implies that
\begin{align*}
  \alpha M_z(x_1) + (1 - \alpha) M_z(x_2) 
  &\subseteq \Phi_z(\alpha x_1 + (1 - \alpha) x_2)) + E + \alpha E + (1 - \alpha) E
  \\
  &\subseteq M_z(\alpha x_1 + (1 - \alpha) x_2) 
\end{align*}
for all $x_1, x_2 \in W$ and $\alpha \in [0, 1]$, that is, the graph of $M_z$ is convex. 

Our aim is to apply Lemma~\ref{lem:PerturbedSolution} with $Z = W$ and $z_* = x_*$. Indeed, by definition 
$0 \in M_{z_*}(\overline{x})$, while condition \eqref{eq:RegularPointConeConstr} implies that 
$0 \in \interior M_{z_*}(X)$. 

From the fact that $\Psi$ is continuously Fr\'{e}chet differentiable on $W$ it follows that for any
$\varepsilon > 0$ there exists $\delta < \min\{ 1, \varepsilon/ 3(1 + \| D \Psi(z_*) \|) \}$ such that
\[
  \| \Psi(z) - \Psi(z_*) \| < \frac{\varepsilon}{3}, \quad
  \| D \Psi(z) - D \Psi(z_*) \| < \frac{\varepsilon}{3(2 + \| \overline{x} \| + \| z_* \|)}
\]
for all $z \in B(z_*, \delta) \cap W$. Choose any $y \in M_{z_*}(\overline{x} + B_X)$. By definition there
exist $x \in (\overline{x} + B_X) \cap W$ and $v \in E$ such that $y = \Phi_{z_*}(x) + v$. Observe that
\begin{multline*}
  \| \Phi_z(x) + v - y \| = \| \Phi_z(x) - \Phi_{z_*}(x) \| 
  \\
  \le \| \Psi(z) - \Psi(z_*) \| + \| D \Psi(z) - D \Psi(z_*) \| \| x - z \| + \| D \Psi(z_*) \| \| z - z_* \|  
  < \varepsilon
\end{multline*}
for all $z \in B(z_*, \delta) \cap W$, which implies that
\[
  M_{z_*}(\overline{x} + B_X) \subseteq M_z(\overline{x} + B_X) + \varepsilon B_Y 
  \quad \forall z \in B(z_*, \delta) \cap W.
\]
Thus, it remains to show that the restriction of the function $\dist(0, M_z(\overline{x}))$ to $W$ is
continuous. 

By definition $\dist(0, M_z(\overline{x})) = \dist(\Phi_z(\overline{x}), - E)$ (see \eqref{eq:ConeConstrMultifunc}).
With the use of the fact that $\Psi$ is continuously Fr\'{e}chet differentiable one obtains that for any 
$\varepsilon > 0$ there exists $r < \min\{ 1, \varepsilon/3 (1 + \| D \Psi(z_*) \|) \}$ such that
\[
  \| \Psi(z) - \Psi(z_*) \| < \frac{\varepsilon}{3}, \quad
  \| D \Psi(z) - D \Psi(z_*) \| < \frac{\varepsilon}{3 (\| \overline{x} \| + \| z_* \| + 1)}
\]
for all $z \in B(z_*, r) \cap W$. Therefore for any such $z$ one has
\begin{align*}
  \| \Phi_z(\overline{x}) - \Phi_{z_*}(\overline{x}) \| &\le \| \Psi(z) - \Psi(z_*) \| 
  \\
  &+ \| D \Psi(z) - D \Psi(z_*) \| \| \overline{x} - z \| + \| D \Psi(z_*) \| \| z - z_* \| 
  < \varepsilon,
\end{align*}
which implies that for any $z \in B(x_*, r) \cap W$ the following inequality holds true:
\[
  \dist(0, M_z(\overline{x})) = \dist(\Phi_z(\overline{x}), -E)
  \le \| \Phi_z(\overline{x}) - \Phi_{z_*}(\overline{x}) \| < \varepsilon
\]
(here we used the fact that $\Phi_{z_*}(\overline{x}) \in -E$). Thus, all assumptions of
Lemma~\ref{lem:PerturbedSolution} with $Z = W$ and $z_* = x_*$ are valid, and by this lemma there exists a
required mapping $\xi(z)$.  
\end{proof}

\subsection{Optimality Conditions}
\label{sect:OptimalityConditions}

Let us extend well-known local optimality conditions for constrained DC optimization problems to the case of 
the problem $(\mathcal{P})$. To the best of the author's knowledge, standard subdifferential calculus cannot be extended
to the case of convex matrix-valued mappings and many other $K$-convex vector-valued maps, which makes it very difficult
to deal with subdifferentials of such functions. Therefore, below we suppose that the mapping $H$ (the $K$-concave part
of $F$) is continuously differentiable, but do not impose any smoothness assumptions on the objective function $f_0$.

\begin{theorem} \label{thrm:OptCond}
Let $x_*$ be a locally optimal solution of the problem $(\mathcal{P})$ and the mapping $H$ be Fr\'{e}chet
differentiable at $x_*$. Then for any $v \in \partial h_0(x_*)$ the point $x_*$ is a globally optimal solution of the
following convex programming problem:
\begin{equation} \label{prob:AuxiliaryConvexProb}
\begin{split}
  &\minimize \enspace g_0(x) - h_0(x_*) - \langle v, x - x_* \rangle \\
  &\mathrm{subject~to} \enspace G(x) - H(x_*) - D H(x_*)(x - x_*) \preceq_K 0, \quad x \in Q,
\end{split}
\end{equation}
where $D H(x_*)$ is the Fr\'{e}chet derivative of $H$ at $x_*$.
\end{theorem}

\begin{proof}
Denote by $\omega_v(x) = g_0(x) - h_0(x_*) - \langle v, x - x_* \rangle$, $x \in \mathbb{R}^d$, the objective function
of problem \eqref{prob:AuxiliaryConvexProb}. This function is convex. Moreover, taking into account the fact that by 
the definition of subgradient $h_0(x) \ge h_0(x_*) + \langle v, x - x_* \rangle$, one obtains that 
$\omega_v(x) \ge f_0(x)$ for all $x \in \mathbb{R}^d$ and $\omega_v(x_*) = f_0(x_*)$.

By contradiction, suppose that there exists $v \in \partial h_0(x_*)$ such that the point $x_*$ is not a
globally optimal solution of problem \eqref{prob:AuxiliaryConvexProb}, i.e. there exists a feasible point $x$ of this
problem such that $\omega_v(x) < \omega_v(x_*)$. Define $x(\alpha) = \alpha x + (1 - \alpha) x_*$. Then
\begin{equation} \label{eq:StrictDescentDirect}
  f_0(x(\alpha)) \le \omega_v(x(\alpha)) \le \alpha \omega_v(x) + (1 - \alpha) \omega_v(x_*) < \omega_v(x_*) = f_0(x_*)
\end{equation}
for all $\alpha \in (0, 1]$, thanks to the convexity of $\omega_v$.

Let us check that $x(\alpha)$ is a feasible point of the problem $(\mathcal{P})$ for all $\alpha \in [0, 1]$.
Then with the use of \eqref{eq:StrictDescentDirect} one can conclude that $x_*$ is not a locally optimal solution of
the problem $(\mathcal{P})$, which contradicts the assumption of the theorem.

Indeed, by Lemma~\ref{lem:GenConvexityViaDerivative} one has $H(x(\alpha)) - H(x_*) - D H(x_*)(x(\alpha) - x_*) \in K$
for all $\alpha \in [0, 1]$. Adding and subtracting $G(x(\alpha))$, one obtains that 
\[
  - F(x(\alpha)) + G(x(\alpha)) - H(x_*) - D H(x_*)(x(\alpha) - x_*) \in K \quad \forall \alpha \in [0, 1]
\]
or, equivalently, 
\[
  F(x(\alpha)) \preceq_K G(x(\alpha)) - H(x_*) - D H(x_*)(x(\alpha) - x_*) \quad \forall \alpha \in [0, 1]. 
\]
Hence taking into account the fact that the point $x(\alpha)$ is feasible for problem \eqref{prob:AuxiliaryConvexProb}
due to the convexity of this problem, one can conclude that $F(x(\alpha)) \preceq_K 0$. Thus, $x(\alpha)$ is a feasible
point of the problem $(\mathcal{P})$ and the proof is complete.  
\end{proof}

Let us reformulate optimality conditions from the previous theorem. Denote by $\Omega(x_*)$ the feasible region of
problem \eqref{prob:AuxiliaryConvexProb} and for any convex set $V \subseteq \mathbb{R}^d$ and $x \in V$ denote by 
$N_V(x) = \{ v \in \mathbb{R}^d \mid \langle v, z - x \rangle \le 0 \: \forall z \in V \}$ \textit{the normal cone}
to $V$ at $x$.

\begin{corollary} \label{crlr:Criticality}
Let $x_*$ be a locally optimal solution of the problem $(\mathcal{P})$ and the map $H$ be Fr\'{e}chet differentiable at
$x_*$. Then 
\[
  \partial h_0(x_*) \subseteq \partial g_0(x_*) + N_{\Omega(x_*)}(x_*).
\]
\end{corollary}

\begin{proof}
Fix any $v \in \partial h_0(x_*)$. By Theorem~\ref{thrm:OptCond} the point $x_*$ is a globally optimal solution of the
convex problem \eqref{prob:AuxiliaryConvexProb}. Applying standard necessary and sufficient optimality conditions for a
convex function on a convex set (see, e.g. \cite[Thm.~1.1.2']{IoffeTihomirov}), one obtains that
$0 \in \partial \omega_v(x_*) + N_{\Omega(x_*)}(x_*)$, where, as above,
$\omega_v(x) = g_0(x) - h_0(x_*) - \langle v, x - x_* \rangle$ is the objective function of problem
\eqref{prob:AuxiliaryConvexProb}. Since $\partial \omega(x_*) = \partial g_0(x_*) - v$, one gets that
$v \in \partial g_0(x_*) + N_{\Omega(x_*)}(x_*)$, which implies the desired result.  
\end{proof}

In the case when a natural constraint qualification (namely, Slater's condition for problem
\eqref{prob:AuxiliaryConvexProb}) holds at $x_*$, one can show that optimality conditions from
Theorem~\ref{thrm:OptCond} coincide with standard optimality conditions for cone constrained optimization problems (see,
e.g. \cite{BonnansShapiro}). To this end, denote by $Y^*$ the topological dual space of $Y$ and by $\langle \cdot, \cdot
\rangle$ the canonical duality pairing between $Y$ and $Y^*$, that is, $\langle y^*, y \rangle = y^*(y)$ for any $y^*
\in Y^*$ and $y \in Y$. 

Let $K^* = \{ y^* \in Y^* \mid \langle y^*, y \rangle \ge 0 \: \forall y \in K \}$ be \textit{the dual cone} of $K$ and
for any $\lambda \in Y^*$ define $L(x, \lambda) = f_0(x) + \langle \lambda, F(x) \rangle$.

\begin{corollary} \label{crlr:OptimalityCond_Lagrange}
Let $x_*$ be a locally optimal solution of the problem $(\mathcal{P})$ and the mappings $G$ and $H$ be
Fr\'{e}chet differentiable at $x_*$. Suppose also that the following constraint qualification holds true:
\[
  0 \in \interior\Big\{ G(x) - H(x_*) - D H(x_*)(x - x_*) + K \Bigm| x \in Q \Big\}
\]
(if $K$ has nonempty interior, it is sufficient to suppose that there exists $x \in Q$ such that
$G(x) - H(x_*) - D H(x_*)(x - x_*) \in - \interior K$). Then for any $v \in \partial h_0(x_*)$ there exists a multiplier
$\lambda_* \in K^*$ such that $\langle \lambda_*, F(x_*) \rangle = 0$ and
\[
  v \in \partial g_0(x_*) + D \Big( \langle \lambda_*, F(\cdot) \rangle \Big) (x_*) + N_Q(x_*).
\] 
In particular, if both $g_0$ and $h_0$ are differentiable at $x_*$, then there exists $\lambda_* \in K^*$ such that
$\langle \lambda_*, F(x_*) \rangle = 0$ and $\langle D_x L(x_*, \lambda_*), x - x_* \rangle \ge 0$ for all $x \in Q$.
\end{corollary}

\begin{proof}
Rewriting problem \eqref{prob:AuxiliaryConvexProb} as the convex cone constrained problem
\begin{align*}
  &\minimize \enspace g_0(x) - h_0(x_*) - \langle v, x - x_* \rangle
  \\
  &\text{subject to} \enspace G(x) - H(x_*) - D H(x_*)(x - x_*) \in -K, \quad x \in Q
\end{align*}
and applying standard necessary and sufficient optimality conditions for convex cone constrained optimization problems
(see, for example, \cite[Thm.~3.6 and Prp.~2.106]{BonnansShapiro}), we arrive at the required result.  
\end{proof}

\begin{remark}
In the case of semidefinite programs, i.e. when $Y = \mathbb{S}^{\ell}$ and $K$ is the cone of positive semidefinite
matrices, the dual cone $K^*$ coincides with $K$ (if we identify the dual of $\mathbb{S}^{\ell}$ with the space
$\mathbb{S}^{\ell}$ itself), and thus the multiplier $\lambda_*$ from the previous corollary is a positive semidefinite
matrix. In addition, the constraint qualification from the corollary takes the form: there exists $x \in Q$ such that
the matrix $G(x) - H(x_*) - D H(x_*)(x - x_*)$ is negative definite.
\end{remark}

\section{DCA for Cone Constrained DC Optimization}
\label{sect:DCA}

The optimality conditions from Theorem~\ref{thrm:OptCond} can be applied to a convergence analysis of a method for
solving cone constrained DC optimization problems proposed in \cite{LippBoyd}, which can be viewed as an extension of
the renown DCA \cite{LeThiPhamDinh2014,LeThiPhamDinh2014b,LeThiDinh2018} to the case of DC problems with cone
constraints. A general scheme of this method for the problem $(\mathcal{P})$ is given in Algorithmic
Pattern~\ref{alg:CCP}. Following \cite{AckooijDeOliveira2019}, we use the term \textit{algorithmic pattern}, since
Algorithmic Pattern~\ref{alg:CCP} is not an algorithm per se, but rather a theoretical scheme (a pattern) that can be
used to define a local search method for the problem $(\mathcal{P})$ by specifying a method for solving the convex
subproblem on Step~2 and a stopping criterion.

\begin{algorithm}[ht!]	\label{alg:CCP}
\caption{DCA for Cone Constrained DC Optimization.}

\noindent\textbf{Initialization.} {Choose a feasible initial point $x_0$ and set $n := 0$.}

\noindent\textbf{Step~1.} {Compute $v_n \in \partial h_0(x_n)$ and $D H(x_n)$.}

\noindent\textbf{Step~2.} {Set the value of $x_{n + 1}$ to an optimal solution of the convex problem
\begin{align*}
  &\minimize \enspace g_0(x) - \langle v_n, x \rangle \\
  &\text{subject to} \enspace G(x) - H(x_n) - D H(x_n)(x - x_n) \preceq_K 0, \quad x \in Q.
\end{align*}
If a stopping criterion is met, \textbf{Stop}. Otherwise, put $n := n + 1$ and go to \textbf{Step 1}.
}
\end{algorithm}

A stopping criterion for Algorithmic Pattern~\ref{alg:CCP} is discussed below (see
Remark~\ref{rmrk:CCP_StoppingCriterion}). Here we only impose one assumption. Namely, we suppose that this criterion is
satisfied, if $x_{n + 1} = x_n$, i.e. the method terminates, if it fails to improve the current iterate $x_n$.

Let us also note that the point $x_{n + 1}$ on Step~2 of Algorithmic Pattern~\ref{alg:CCP} is not correctly defined in
the general case, since the corresponding convex problem might not have optimal solutions. One can ensure the existence
of optimal solutions by imposing suitable coercivity/compactness assumptions on the original problem 
(cf.~Lemma~\ref{lem:PenalizedProblem_WellPosedness}). For the sake of shortness, below we always assume that iterations
of Algorithmic Pattern~\ref{alg:CCP} are correctly defined. Finally, it should be mentioned that the convex subproblem
on Step~2 of Algorithmic Pattern~\ref{alg:CCP} can be solved with the use of interior point methods (see, e.g. 
\cite[Sect.~11.6]{BoydVandenberghe} and \cite{NesterovNemirovski,BenTalNemirovski}), augmented Lagrangian methods
\cite{KocvaraStingl,Stingl}, etc.

Our aim is to prove a convergence theorem for Algorithmic Pattern~\ref{alg:CCP}. Clearly, in the nonsmooth case (more
precisely, when $h_0$ is nonsmooth) one cannot expect a sequence $\{ x_n \}$ generated by this algorithm to converge to
a point $x_*$ satisfying optimality conditions from Theorem~\ref{thrm:OptCond} for all $v \in \partial h(x_*)$.
Furthermore, these optimality conditions are often too restrictive for applications, since they require the knowledge of
the entire subdifferential $\partial h(x_*)$, which might make verification of these conditions too computationally
expensive or even impossible. That is why one usually establishes a convergence of DC optimization methods to so-called
\textit{critical} points \cite{LeThiDinh2018,AckooijDeOliveira2019,JokiBagirov2020}. Recall that a point $x_*$ is said
to be \textit{critical} for the problem $(\mathcal{P})$, if the following condition holds true:
\[
  \partial h_0(x_*) \cap \Big( \partial g_0(x_*) + N_{\Omega(x_*)}(x_*) \Big) \ne \emptyset.
\]
Note that this condition is satisfied if and only if \textit{there exists} $v \in \partial h_0(x_*)$ such that $x_*$ is
a globally optimal solution of convex problem \eqref{prob:AuxiliaryConvexProb} (see Theorem~\ref{thrm:OptCond}). Hence,
in particular, if a point $x_n$ on Step~2 of Algorithmic Pattern~\ref{alg:CCP} is not critical for the problem
$(\mathcal{P})$, then $x_n$ is \textit{not} an optimal solution of the corresponding convex subproblem. In other words,
if Algorithmic Pattern~\ref{alg:CCP} terminates on iteration $n \in \mathbb{N}$, then $x_n$ is a critical point for the
problem $(\mathcal{P})$.

The proof of the following theorem was largely inspired by the convergence analysis of an algorithmic pattern for
inequality constrained DC optimization problems from \cite[Section~3.1]{AckooijDeOliveira2019}. However, let us note
that we prove the global convergence of Algorithmic Pattern~\ref{alg:CCP} to a critical point under assumptions that are
different from the ones used in \cite{AckooijDeOliveira2019}.

\begin{theorem} \label{thrm:CCP_Convergence}
Let the function $f_0$ be bounded below on the feasible region of the problem $(\mathcal{P})$, and $\{ x_n \}$ be 
the sequence generated by Algorithmic Pattern~\ref{alg:CCP}. Then the following statements hold true:
\begin{enumerate}
\item{the feasible region $\Omega(x_n)$ of the convex subproblem on Step~2 of the algorithic pattern is nonempty for
all $n \in \mathbb{N}$, and the sequence $\{ x_n \}$ is feasible for the problem $(\mathcal{P})$;
\label{st:feasibility}}

\item{for any $n \in \mathbb{N}$ either $x_n$ is a critical point of the problem $(\mathcal{P})$
and the process terminates at step $n$ or $f_0(x_{n + 1}) < f_0(x_n)$; moreover, if the algorithmic pattern does not
terminate, then the sequence $\{ f_0(x_n) \}$ converges;
\label{st:relaxation}}

\item{if the function $h_0$ is strongly convex with constant $\mu > 0$, then
\begin{equation} \label{eq:StrongRelaxation}
  f_0(x_{n + 1}) \le f_0(x_n) - \frac{\mu}{2} |x_{n + 1} - x_n|^2
\end{equation}
for all $n \in \mathbb{N}$;
\label{st:strong_relaxation}}

\item{if $x_*$ is a limit point of the sequence $\{ x_n \}$ such that
\[
  0 \in \interior\big\{ G(x) - H(x_*) - D H(x_*)(x - x_*) + K \bigm| x \in Q \big\}
\]
(that is, Slater's condition holds for problem \eqref{prob:AuxiliaryConvexProb}), then $x_*$ is a critical point for
the problem $(\mathcal{P})$.
\label{st:convergence}}
\end{enumerate}
\end{theorem}

\begin{proof}
\ref{st:feasibility}. Let us prove this statement by induction in $n$. By our assumption $x_0$ is feasible for 
the problem $(\mathcal{P})$, which implies that $x_0 \in \Omega(x_0)$, that is, the feasible region $\Omega(x_0)$ of
the convex subproblem is nonempty.

\textit{Inductive step.} Suppose that for some $n \in \mathbb{N}$ the point $x_n$ is feasible for 
the problem $(\mathcal{P})$ and $\Omega(x_n)$ is nonempty. Let us prove that $x_{n + 1}$ is feasible for 
the problem $(\mathcal{P})$. Then $x_{n + 1} \in \Omega(x_{n + 1})$, i.e. $\Omega(x_{n + 1}) \ne \emptyset$, and 
the proof of the first statement is complete.

Indeed, by definition the point $x_{n + 1}$ is a globally optimal solution of the convex subproblem on Step~2 of the
algorithmic pattern, which implies that
\[
  G(x_{n + 1}) - H(x_n) - D H(x_n)(x_{n + 1} - x_n) \preceq_K 0, \quad x_{n + 1} \in Q.
\]
By Lemma~\ref{lem:GenConvexityViaDerivative} (see page~\pageref{lem:GenConvexityViaDerivative}) one has
\[
  - H(x_{n + 1}) \preceq_K - H(x_n) - D H(x_n)(x_{n + 1} - x_n).
\]
Therefore $F(x_{n + 1}) = G(x_{n + 1}) - H(x_{n + 1}) \preceq_K 0$, i.e. the point $x_{n + 1}$ is feasible for the
problem $(\mathcal{P})$.

\ref{st:relaxation}. If a point $x_n$ is not critical, then, as was noted above, $x_n$ is not a solution of the convex
subproblem on Step~2 of Algorithmic Pattern~\ref{alg:CCP}, which implies that
\[
  g_0(x_{n + 1}) - \langle v_n, x_{n + 1} - x_n \rangle < g_0(x_n).
\]
Subtracting $h_0(x_n)$ from both sides of this inequality and applying the definition of subgradient, one obtains that
$f_0(x_{n + 1}) < f_0(x_n)$. Hence bearing in mind the facts that the sequence $\{ x_n \}$ is feasible and $f_0$ is
bounded below on the feasible region, one gets that the sequence $\{ f(x_n) \}$ converges.

\ref{st:strong_relaxation}. Fix any $n \in \mathbb{N}$. Due to the strong convexity of $h_0$ one has
\[
  h_0(x_{n + 1}) - h_0(x_n) \ge \langle v_n, x_{n + 1} - x_n \rangle + \frac{\mu}{2} |x_{n + 1} - x_n|^2.
\]
Furthermore, by the definition of $x_{n + 1}$ one has
\[
  g_0(x_n) \ge g_0(x_{n + 1}) - \langle v_n, x_{n + 1} - x_n \rangle.
\]
Summing up these two inequalities one obtains that \eqref{eq:StrongRelaxation} holds true.

\ref{st:convergence}. By our assumption there exists a subsequence $\{ x_{n_k} \}$ converging to $x_*$. The
corresponding sequence $\{ v_{n_k} \}$ of subgradients of the function $h_0$ is bounded, since the subdifferential
mapping of a finite convex function is locally bounded (see, e.g. \cite[Cor.~24.5.1]{Rockafellar}). Therefore,
replacing, if necessary, the sequence $\{ x_{n_k} \}$ with its subsequence one can suppose that the sequence of
subgradients $\{ v_{n_k} \}$ converges to some vector $v_*$ belonging to $\partial h_0(x_*)$ due to the fact that 
the graph of the subdifferential is closed (see, e.g. \cite[Thm.~24.4]{Rockafellar}).

By contradiction, suppose that $x_*$ is not a critical point of the problem $(\mathcal{P})$. Then, in particular, 
\[
  v_* \notin \partial g_0(x_*) + N_{\Omega(x_*)}(x_*),
\]
which by Theorem~\ref{thrm:OptCond} implies that $x_*$ is not a globally optimal solution of the convex problem
\eqref{prob:AuxiliaryConvexProb}. Consequently, there exists a feasible point $\overline{x}$ of this problem and 
$\theta > 0$ such that $g_0(\overline{x}) - \langle v_*, \overline{x} - x_* \rangle < g(x_*) - \theta$. 

Applying Lemma~\ref{lem:PerturbedSolution} with $\Phi = G$, $\Psi = H$, $W = Q$, and $E = K$,
one obtains that for any $z \in Q$ lying in a neighbourhood of $x_*$ one can find a point $\xi(z) \in Q$ such
that $G(\xi(z)) - H(z) - D H(z)(\xi(z) - z) \preceq_K 0$ and $\xi(z) \to \overline{x}$ as 
$z \to x_*$. Hence taking into account the facts that the subsequence $\{ x_{n_k} \}$ converges to $x_*$, while
$\{ v_{n_k} \}$ converges to $v_*$, one obtains that there exists $k_0 \in \mathbb{N}$ such that for all $k \ge k_0$ one
has
\begin{gather*}
  g_0(\xi(x_{n_k})) - \langle v_{n_k}, \xi(x_{n_k}) - x_{n_k} \rangle 
  \le g_0(x_{n_k}) - \frac{\theta}{2}, 
  \\
  G(\xi(x_{n_k})) - H(x_{n_k}) - D H(x_{n_k})(\xi(x_{n_k}) - x_{n_k}) \preceq_K 0.
\end{gather*}
Note that $\xi(x_{n_k})$ is a feasible point of the convex subproblem on Step~2 of Algorithmic Pattern~\ref{alg:CCP} for
any $k \ge k_0$. Consequently, by the definition of $x_{n_k + 1}$ one has
\begin{align*}
  g_0(x_{n_k + 1}) - \langle v_{n_k}, x_{n_k + 1} - x_{n_k} \rangle
  &\le g_0(\xi(x_{n_k})) - \langle v_{n_k}, \xi(x_{n_k}) - x_{n_k} \rangle 
  \\
  &\le g_0(x_{n_k}) - \frac{\theta}{2}
\end{align*}
for all $k \ge k_0$. Subtracting $h_0(x_{n_k})$ from both sides of this inequality and applying the definition of
subgradient, one gets that $f_0(x_{n_k + 1}) \le f_0(x_{n_k}) - \theta / 2$ for any $k \ge k_0$. Hence with the use
of the second part of this theorem one can conclude that $f_0(x_n) \to - \infty$, which contradicts the facts that
$f_0$ is bounded below on the feasible set by our assumption and the sequence $\{ x_n \}$ is feasible by the first part
of the theorem.  
\end{proof}

\begin{remark} \label{rmrk:CCP_StoppingCriterion}
{(i)~Note that the assumption on the strong convexity of the function $h_0$ is not restrictive, since if this assumption
is not satisfied, for any $\mu > 0$ one can replace the DC decomposition $f_0 = g_0 - h_0$ of the objective function
$f_0$ with the following one: 
\[
  f_0(x) = \left( g_0(x) + \frac{\mu}{2} |x|^2 \right) - \left( h_0(x) + \frac{\mu}{2} |x|^2 \right),
  \quad x \in \mathbb{R}^d.
\]
}
\vspace{-3mm}

\noindent{(ii)~Since by the previous theorem the sequence $\{ f(x_n) \}$ converges, one can use the inequalities
$|f_0(x_{n + 1}) - f_0(x_n)| \le \varepsilon$ and/or $\| x_{n + 1} - x_n \| \le \varepsilon$ as a stopping criterion
for Algorithmic Pattern~\ref{alg:CCP}. Note also that by definitions
\begin{align*}
  0 &\le g_0(x_n) - \langle v_n, x_n \rangle - \Big( g_0(x_{n + 1}) - \langle v_n, x_{n + 1} \rangle \Big)
  \\
  &= g_0(x_n) - g_0(x_{n + 1}) + \langle v_n, x_{n + 1} - x_n \rangle
  \\
  &\le g_0(x_n) - g_0(x_{n + 1}) + h_0(x_{n + 1}) - h_0(x_n) \le f_0(x_n) - f_0(x_{n + 1})
\end{align*}
and the first inequality turns into an equality if and only if $x_n$ is a critical point of the problem $(\mathcal{P})$.
Therefore, one can replace the stopping criterion $|f_0(x_{n + 1}) - f_0(x_n)| \le \varepsilon$ with
$|g_0(x_{n + 1}) - \langle v_n, x_{n + 1} \rangle - (g_0(x_n) - \langle v_n, x_n \rangle)| \le \varepsilon$ to avoid
the computation of $h_0(x_n)$ and $h_0(x_{n + 1})$. For a discussion of more elaborate stopping criteria for DC
optimization methods involving approximate optimality conditions see \cite{Strekalovsky_Collect}.
}
\end{remark}

\section{DCA2/The Penalty Convex-Concave Procedure}
\label{sect:PenalizedCCP}

In order to apply Algorithmic Pattern~\ref{alg:CCP}, one needs to find a feasible point of the problem under
consideration. In the case when such point is unknown in advance and is hard to compute, one can use a combination of
the DCA and exact penalty techniques that allows one to start iterations at infeasible points. Such modifications of 
Algorithmic Pattern~\ref{alg:CCP} were discussed in \cite{LippBoyd} (and in \cite{LeThiPhamDinh2014,LeThiPhamDinh2014b}
in the case of inequality constrained problems). Here we present and analyze one such method, which is a slight
modification of the penalty convex-concave procedure \cite[Algorithm~4.2]{LippBoyd}. This method can be viewed as
an extension of DCA2 algorithm from \cite{LeThiPhamDinh2014,LeThiPhamDinh2014b} to the case of cone constrained DC
optimization problems.

A general scheme of DCA2/Penalty CCP for the problem $(\mathcal{P})$ is given in Algorithmic
Pattern~\ref{alg:CCP_penalty}. The only difference between our method and \cite[Algorithm~4.2]{LippBoyd} is the penalty
updates. Namely, in contrast to \cite{LippBoyd}, we increase the penalty parameter, only if the infeasibility measure at
the current iteration exceeds a prespecified threshold. Let us also note that the inequality $t_0 \succ_{K^*} 0$ means
that $t_0 \in K^*$ and $\langle t_0, y \rangle > 0$ for any $y \in K$, $y \ne 0$. Finally, a stopping criterion for
Algorithmic Pattern~\ref{alg:CCP_penalty} is discussed in Remark~\ref{rmrk:PenaltyCCP_StoppinCriterion} below.

\begin{algorithm}[ht!]	\label{alg:CCP_penalty}
\caption{DCA2/Penalty CCP/Exact Penalty DCA.}

\noindent\textbf{Initialization.} {Choose an initial point $x_0 \in Q$, penalty parameter $t_0 \succ_{K^*} 0$, the
maximal norm of the penalty parameter $\tau_{\max} > 0$, $\mu > 1$, infeasibility tolerance $\varkappa \ge 0$, and set
$n := 0$.}

\noindent\textbf{Step~1.} {Compute $v_n \in \partial h_0(x_n)$ and $D H(x_n)$.}

\noindent\textbf{Step~2.} {Set the value of $(x_{n + 1}, s_{n + 1})$ to an optimal solution of the convex problem
\begin{align*}
  &\minimize_{(x, s)} \enspace g_0(x) - \langle v_n, x \rangle + \langle t_n, s \rangle \\
  &\text{subject to} \enspace G(x) - H(x_n) - D H(x_n)(x - x_n) \preceq_K s, \quad s \succeq_K 0, \quad x \in Q.
\end{align*}
If a stopping criterion is satisfied, \textbf{Stop}.
}

\noindent\textbf{Step~3.} {Define
\[
  t_{n + 1} = \begin{cases}
    \mu t_n, & \text{if } \| s_{n + 1} \| \ge \varkappa \text{ and } \mu \| t_n \| \le \tau_{\max},
    \\
    t_n, & \text{otherwise.}
  \end{cases}
\]
Put $n := n + 1$ and go to \textbf{Step 1}.
}
\end{algorithm}

\begin{remark}
Let us point out how the penalized subproblem on Step~2 of Algorithmic Pattern~\ref{alg:CCP_penalty} is connected with
standard exact penalty methods for cone constrained optimization \cite{BonnansShapiro,Auslender}. Following the standard
exact penalty methodology, one can base an exact penalty local search method for the problem $(\mathcal{P})$ on the
function $\Phi_c(\cdot) = f_0(\cdot) + c \dist(F(\cdot), -K)$. Taking into account the fact that the map 
$y \mapsto \dist(y, -K)$ is monotone with respect to the partial order $\preceq_K$ and utilizing the DC structure of the
problem $(\mathcal{P})$ one can define a global convex majorant of the function $\Phi_c$ of the form
\begin{align*}
  \Psi_c(x; x_n; v_n) &= g_0(x) - h_0(x_n) - \langle v_n, x - x_n \rangle
  \\
  &+ c \dist(G(x) - H(x_n) - DH(x_n)(x - x_n), - K),
\end{align*}
and propose a DCA-type method for the problem $(\mathcal{P})$ based on sequential minimization of this function. To be
able to utilize efficient convex cone constrained optimization methods
\cite{NesterovNemirovski,BenTalNemirovski,BoydVandenberghe} and corresponding software, one can rewrite the problem of
minimizing the function $\Psi_c(\cdot; x_n; v_n)$ over the set $Q$ as the equivalent convex cone constrained problem
\begin{equation} \label{prob:DifferentPenalty}
\begin{split}
  &\minimize_{(x, s, \xi)} \enspace g_0(x) - \langle v_0, x \rangle + c \xi 
  \\
  &\text{s.t.} \enspace G(x) - H(x_n) - D H(x_n)(x - x_n) \preceq_K s, 
  \quad (\xi, s) \in \mathcal{K}_L, \quad x \in Q,
\end{split}
\end{equation}
where $\mathcal{K}_L = \{ (\xi, s) \in \mathbb{R} \times Y \mid \xi \ge \| s \| \}$ is the generalized Lorentz (second
order) cone. Alternatively, following the approach of \cite{LippBoyd}, one can consider the penalized
subproblem from Step~2 of Algorithmic Pattern~\ref{alg:CCP_penalty}. This subproblem can be derived in exactly the same
way as problem \eqref{prob:DifferentPenalty}, if one replaces the function $c \dist(F(\cdot), - K)$ with the following
one:
\[
  \varphi_{t_n}(x) = \inf\Big\{ \langle t_n, s \rangle \Bigm| s \succeq_K F(x), \: s \succeq_K 0 \Big\}
\]
(as will be shown below, under some natural assumptions on the cone $K$ and the space $Y$, one has 
$\varphi_{t_n}(\cdot) \ge \tau \dist(F(\cdot), -K)$ for some $\tau > 0$). Let us note that all results below can be
easily extended to the case when the auxiliary subproblem on Step~2 of Algorithmic Pattern~\ref{alg:CCP_penalty} is
replaced by problem \eqref{prob:DifferentPenalty}. For the sake of shortness, we do note present this extension here.
\end{remark}

Let us analyze convergence of Algorithmic Pattern~\ref{alg:CCP_penalty}. Firstly, we show that under some standard
assumptions the penalized convex subproblem on Step~2 of this algorithm is \text{exact}, in the sense that if the norm
of the penalty parameter $t_n$ is sufficiently large, then a solution of the subproblem on Step~2 of
Algorithmic Pattern~\ref{alg:CCP_penalty} coincides with the solution of the corresponding non-penalized problem
\begin{equation} \label{prob:NonPenalizedSequence}
\begin{split}
  &\minimize \enspace g_0(x) - \langle v, x \rangle
  \\
  &\text{subject to} \enspace G(x) - H(x_n) - D H(x_n)(x - x_n) \preceq_K 0, \quad x \in Q,
\end{split}
\end{equation}
provided the feasible region of this problem is nonempty. This result implies, in particular, that if for some 
$n \in \mathbb{N}$ the norm of the penalty parameter $t_n$ exceeds a certain threshold and the feasible region of
problem \eqref{prob:NonPenalizedSequence} is nonempty, then the next point $x_{n + 1}$ is feasible for 
the problem $(\mathcal{P})$ and the rest of the iterations of Algorithmic Pattern~\ref{alg:CCP_penalty} coincide with 
the iterations of Algorithmic Pattern~\ref{alg:CCP}. Thus, in this case one can ensure the convergence of a sequence
generated by
Algorithmic Pattern~\ref{alg:CCP_penalty} to a critical point for the problem $(\mathcal{P})$.

Before we proceed to the proof of the exactness of the subproblem from Step~2 of Algorithmic
Pattern~\ref{alg:CCP_penalty}, let
us first provide simple sufficient conditions for the existence of globally optimal solutions of this problem and the
corresponding non-penalized problem \eqref{prob:NonPenalizedSequence}. To this end, recall that a function 
$\varphi \colon \mathbb{R}^d \to \mathbb{R}$ is called \textit{coercive} on the set $Q$, if $\varphi(x_n) \to + \infty$
as $n \to \infty$ for any sequence $\{ x_n \} \subset Q$ such that $\| x_n \| \to + \infty$ as $n \to \infty$.

\begin{lemma} \label{lem:PenalizedProblem_WellPosedness}
Let the space $Y$ be finite dimensional, the cone $K$ be generating (i.e. $K - K = Y$), and the penalty function 
$\Phi_c(\cdot) = f_0(\cdot) + c \dist(F(\cdot), - K)$ be coercive on $Q$ for some $c > 0$. Then there exists 
$\mu_* \ge 0$ such that for any $\mu \ge \mu_*$ and for all $z \in Q$ and $v \in \partial h(z)$ there exists
a globally optimal solution of the penalized problem
\begin{equation} \label{prob:Penalized}
\begin{split}
  &\minimize_{(x, s)} \enspace g_0(x) - \langle v, x \rangle + \mu \langle t_0, s \rangle \\
  &\mathrm{subject~to} \enspace 
  G(x) - H(z) - D H(z)(x - z) \preceq_K s, \quad s \succeq_K 0,  \quad x \in Q.
\end{split}
\end{equation}
Moreover, if the feasible region of the corresponding non-penalized problem
\begin{equation} \label{prob:NonPenalized}
\begin{split}
  &\minimize \enspace g_0(x) - \langle v, x \rangle \\
  &\mathrm{subject~to} \enspace G(x) - H(z) - D H(z)(x - z) \preceq_K 0, \quad x \in Q
\end{split}
\end{equation}
is nonempty, then this problem has a globally optimal solution as well.
\end{lemma}

\begin{proof}
Indeed, fix any $z \in Q$. Suppose at first that the feasible region of problem \eqref{prob:NonPenalized} is nonempty.
By Lemma~\ref{lem:GenConvexityViaDerivative} (see page~\pageref{lem:GenConvexityViaDerivative}) one has
\[
  - H(x) \preceq_K - H(z) - D H(z)(x - z) \quad \forall x \in \mathbb{R}^d.
\]
Adding $G(x)$ to both sides of this inequality, one obtains that
\begin{equation} \label{eq:VectorConvexity_DerivTest}
  F(x) \preceq_K G(x) - H(z) - D H(z)(x - z) \quad
  \forall x \in \mathbb{R}^d,
\end{equation}
which implies that the feasible region of problem \eqref{prob:NonPenalized} is contained in the feasible region of
the problem $(\mathcal{P})$. 

From the the coercivity of the function $\Phi_c(\cdot) = f_0(\cdot) + c \dist(F(\cdot), - K)$ on the set $Q$ it
follows that the function $f_0$ is coercive on the feasible region of the problem $(\mathcal{P})$ and, therefore, on the
feasible region of problem \eqref{prob:NonPenalized} as well (recall that $F(x) \preceq_K 0$ if and only if 
$F(x) \in - K$). Hence taking into account the fact that by the definition of subgradient
\[
  g_0(x) - \langle v, x \rangle \ge f_0(x) + h_0(z) - \langle v, z \rangle 
  \quad \forall x \in \mathbb{R}^d,
\]
one gets that the objective function of problem \eqref{prob:NonPenalized} is coercive on the feasible region of this
problem, which is closed by virtue of our assumptions on $G$ and $H$. Consequently, there exists a globally optimal
solution of problem \eqref{prob:NonPenalized}.

Let us now consider problem \eqref{prob:Penalized}. Our assumptions on $G$ and $H$ guarantee that the feasible region of
this problem is closed. Note that a pair $(x, s) \in Q \times K$ is feasible for this problem if and only if
\[
  G(x) - H(z) - D H(z)(x - z) \in s - K.
\]
Hence bearing in mind the fact that the cone $K$ is generating, one gets that the feasible region of
problem \eqref{prob:Penalized} is nonempty.

Let us check that the objective function 
\[
  \omega(x, s) = g_0(x) - \langle v, x \rangle + \mu \langle t_0, s \rangle
\]
of problem \eqref{prob:Penalized} is coercive on the feasible region of this problem, provided $\mu \ge 0$ is
sufficiently large. Then one can conclude that a globally optimal solution of problem \eqref{prob:Penalized} exists as
well.

Indeed, by contradiction, suppose that $\omega$ is not coercive on the feasible region of problem 
\eqref{prob:Penalized}. Then there exist $M > 0$ and a sequence $\{ (x_n, s_n) \}$ of feasible points of
problem \eqref{prob:Penalized} such that $\| x_n \| + \| s_n \| \to + \infty$ as $n \to \infty$, but
$\omega(x_n, s_n) \le M$ for all $n \in \mathbb{N}$. Observe that $F(x_n) \preceq_K s_n$ for all $n \in \mathbb{N}$ due
to \eqref{eq:VectorConvexity_DerivTest}, which implies that
\[
  M \ge \omega(x_n, s_n) \ge \inf\Big\{ \omega(x_n, s) \Bigm| s \succeq_K F(x_n), \: s \succeq_K 0 \Big\} 
  \quad \forall n \in \mathbb{N}.
\]
Let us estimate the infimum on the right-hand side of this inequality. Bearing in mind the facts that $t_0$ is a 
continuous linear functional, $t_0 \succ_{K^*} 0$, and $K$ is a closed subset of a finite dimensional normed space, one
obtains that
\[
  \tau := \min\Big\{ \langle t_0, s \rangle \Bigm| s \in K, \: \| s \| = 1 \Big\} > 0, \quad
  \langle t_0, s \rangle \ge \tau \| s \| \quad \forall s \in K.
\]
Therefore for any $n \in \mathbb{N}$ one has
\begin{align*}
  M \ge \omega(x_n, s_n) &\ge g_0(x_n) - \langle v, x_n \rangle
  + \mu \tau \inf\Big\{ \| s \| \Bigm| s \in F(x_n) + K, \: s \in K \Big\} 
  \\
  &\ge f_0(x_n) + h_0(z) - \langle v, z \rangle + \mu \tau \inf_{y \in K} \| F(x_n) + y \|
  \\
  &= f_0(x_n) + h_0(z) - \langle v, z \rangle + \mu \tau \dist(F(x_n), - K).
\end{align*}
Hence taking into account the fact that by our assumption the penalty function 
$\Phi_c(\cdot) = f_0(\cdot) + c \dist(F(\cdot), - K)$ is coercive on $Q$, one obtains that the sequence $\{ x_n \}$ is
bounded, provided $\mu \ge c / \tau$. Consequently, $\| s_n \| \to + \infty$ as $n \to \infty$, which contradicts the
fact that
\[
  M \ge \omega(x_n, s_n) \ge \min_{\| x \| \le r} \big( g_0(x) - \langle v, x \rangle \big)
  + \mu \tau \| s_n \|
\]
for all $n \in \mathbb{N}$, where $r = \sup_{n \in \mathbb{N}} \| x_n \|$. Thus, the function $\omega$ is coercive
on the feasible region of problem \eqref{prob:Penalized} for any $\mu \ge c / \tau$ and for any such $\mu$ there exists
a globally optimal solution of this problem.  
\end{proof}

\begin{remark}
Note that the assumptions on the space $Y$ and the cone $K$ are not used in the proof of the existence of globally
optimal solutions of the non-penalized problem~\eqref{prob:NonPenalized}.
\end{remark}

\subsection{Exactness of the Convex Subproblem}

Now we can turn to the proof of the exactness of the penalized problem \eqref{prob:Penalized}. Introduce the set 
\[
  \mathscr{D} = \Big\{ \overline{x} \in Q \Bigm| 
  \exists x \in Q \colon G(x) - H(\overline{x}) - DH(\overline{x})(x - \overline{x}) \preceq_K 0 \Big\},
\]
i.e. $\mathscr{D}$ is the set of all those $\overline{x} \in Q$ for which the feasible region of the non-penalized
problem \eqref{prob:NonPenalized} is nonempty. Observe that the feasible region of the problem $(\mathcal{P})$ is
contained in $\mathscr{D}$, but in the general case $\mathscr{D} \ne \mathbb{R}^d$. Denote by
\[
  \mathscr{D}_s = \Big\{ \overline{x} \in Q \Bigm| 
  0 \in \interior\big\{ G(x) - H(\overline{x}) - D H(\overline{x})(x - \overline{x}) + K \bigm| x \in Q \big\} \Big\}.
\]
the set of all those $\overline{x} \in Q$ for which the constraint qualification from
Corollary~\ref{crlr:OptimalityCond_Lagrange} holds true. It should be noted that in the case when the cone $K$ has
nonempty interior, this constraint qualification is satisfied if and only if there exists $x \in Q$ such that
\[
  G(x) - H(\overline{x}) - D H(\overline{x})(x - \overline{x}) \in - \interior K,
\]
that is, if and only if Slater's condition for the non-penalized problem \eqref{prob:NonPenalized} holds true (see, e.g.
\cite[Prop.~2.106]{BonnansShapiro}). Note that by definition $\mathscr{D}_s \subseteq \mathscr{D}$. Thus,
$\mathscr{D}_s$
is the subset of $\mathscr{D}$ consisting of all those $\overline{x}$ for which Slater's condition holds true for 
the non-penalized problem.

Under some natural assumptions one can verify that the set $\mathscr{D}$ is closed (in particular, it is sufficient to
suppose that the feasible region of the problem $(\mathcal{P})$ is bounded, $G$ is continuous, and $H$ is continuously
differentiable), while the set $\mathscr{D}_s$ is open in $Q$. Therefore, there are some degenerate points 
$\overline{x} \in \mathscr{D} \setminus \mathscr{D}_s$ (e.g. the ones that lie on the boundary of $\mathscr{D}$ in $Q$)
for which one must impose some additional assumptions. Our aim is to first provide somewhat cumbersome sufficient
conditions for the exactness of the penalized problem \eqref{prob:Penalized} for the entire set $\mathscr{D}$ or its
arbitrary subset, and then show that these conditions are satisfied for any compact subset of $\mathscr{D}_s$. The
sufficient conditions presented here are based on a uniform local error bound for the non-penalized problem
\eqref{prob:NonPenalized}.

To simplify the formulations and proofs of the statements below, for any $z \in \mathbb{R}^d$ introduce the convex
mapping $F_z(x) = G(x) - H(z) - D H(z)(x - z)$, $x \in \mathbb{R}^d$, and the set-valued mapping
\[
  M_z(x) = \begin{cases}
    G(x) - H(z) - D H(z)(x - z) + K, & \text{if } x \in Q, \\
    \emptyset, & \text{if } x \notin Q.
  \end{cases}
\]
The multifunction $M_z$ is convex and closed, provided the map $G$ is continuous on $Q$. For any metric space 
$(X, \rho)$ denote $B(x, r) = \{ x' \in X \mid \rho(x', x) \le r \}$ for all $x \in X$.

\begin{proposition} \label{prp:ExactPenaltySubproblem}
Let $K$ be finite dimensional and there exist $c \ge 0$ such that the penalty function 
$\Phi_c(\cdot) = f_0(\cdot) + c \dist(F(\cdot), - K)$ is coercive on $Q$. Let also $\mathscr{D}_0 \subseteq \mathscr{D}$
be a nonempty set for which one can find $a > 0$, $L_g > 0$, and $L_h > 0$ such that for any $z \in \mathscr{D}_0$ and
$v \in \partial h_0(z)$ one has $\| v \| \le L_h$ and there exist $r > 0$ and a globally optimal solution $x_*$ of the
non-penalized problem \eqref{prob:NonPenalized} such that $g_0$ is Lipschitz continuous near $x_*$ with Lipschitz
constant $L_g$ and
\begin{equation} \label{eq:UniformErrorBound}
  \dist(F_{z}(x), -K) \ge a \dist(x, M_{z}^{-1}(0)) \quad \forall x \in B(x_*, r) \cap Q.
\end{equation}
Then there exists $\mu_* \ge 0$ such that for all $\mu \ge \mu_*$ and for any $z \in \mathscr{D}_0$ and 
$v \in \partial h_0(z)$ there exists a globally optimal solution of the penalized problem \eqref{prob:Penalized}
and a pair $(x_*, s_*)$ is a solution of this problem if and only if $s_* = 0$ and $x_*$ is a solution of 
the corresponding non-penalized problem \eqref{prob:NonPenalized}.
\end{proposition}

\begin{proof}
Fix any $z \in \mathscr{D}_0$ and $v \in \partial h_0(z)$, and denote by 
\[
  \omega_{\mu}(x, s) = g_0(x) - \langle v, x - z \rangle + \mu \langle t_0, s \rangle.
\]
the objective function of problem \eqref{prob:Penalized}, shifted by the constant $\langle v, z \rangle$ for the sake
of convenience. 

Arguing in the same way as in the proof of Lemma~\ref{lem:PenalizedProblem_WellPosedness}, one can check that there
exists $\tau > 0$ such that $\langle t_0, s \rangle \ge \tau \| s \|$ for all $s \in K$. Therefore, for any feasible
point $(x, s)$ of problem \eqref{prob:Penalized} one has
\begin{align*}
  \omega_{\mu}(x, s) &\ge g_0(x) - \langle v, x - z \rangle 
  + \mu \tau \inf\big\{ \| s \| \bigm| s \succeq_K F_z(x), \: s \succeq_K 0 \big\}
  \\
  &\ge g_0(x) - \langle v, x - z \rangle + \mu \tau \inf\big\{ \| s \| \bigm| s \in F_z(x) + K \big\}
  \\
  &= g_0(x) - \langle v, x - z \rangle + \mu \tau \dist(F_z(x), - K).
\end{align*}
Let $x_*$ be a globally optimal solution of the non-penalized problem \eqref{prob:NonPenalized} (optimal solutions of
this problem exist by Lemma~\ref{lem:PenalizedProblem_WellPosedness}). Observe that by definition the set $M_z^{-1}(0)$
coincides with the feasible region of problem \eqref{prob:NonPenalized}. Therefore, by 
\cite[Prop.~2.7]{DolgopolikExactPenalty} there exists $\delta > 0$ such that
\[
  g_0(x) - \langle v, x - z \rangle \ge
  g_0(x_*) - \langle v, x_* - z \rangle - (L_g + L_h) \dist(x, M_z^{-1}(0))
\]
for all $x \in B(x_*, \delta) \cap Q$. Consequently, applying inequality \eqref{eq:UniformErrorBound}, one obtains that
\[
  \omega_{\mu}(x, s) \ge g_0(x_*) - \langle v, x_* - z \rangle
  + \left( \mu \tau a - L_g - L_h \right) \dist(x, M_z^{-1}(0))
\]
for any feasible point $(x, s)$ of problem \eqref{prob:Penalized} such that $x \in B(x_*, \delta)$. Hence for any
such $(x, s)$ one has
\[
  \omega_{\mu}(x, s) \ge g_0(x_*) - \langle v, x_* - z \rangle = \omega(x_*, 0)
  \quad \forall \mu \ge \mu_* := \frac{L_g + L_h}{\tau a},
\]  
that is, $(x_*, 0)$ is a locally optimal solution of problem \eqref{prob:Penalized} for any 
$\mu \ge \mu_*$. Taking into account the fact that this problem is convex, one gets that for any such
$\mu$ the pair $(x_*, 0)$ is a globally optimal solution of problem \eqref{prob:Penalized}. Furthermore, since
for any other globally optimal solution $\widehat{x}$ of the non-penalized problem \eqref{prob:NonPenalized} one
has $\omega_{\mu}(x_*, 0) = \omega_{\mu}(\widehat{x}, 0)$, one obtains that for \textit{any} globally optimal solution 
$\widehat{x}$ of the non-penalized problem \eqref{prob:NonPenalized} and for all $\mu \ge \mu_*$ the pair 
$(\widehat{x}, 0)$ is a globally optimal solution of the penalized problem \eqref{prob:Penalized}.

Observe that if a pair $(\widehat{x}, 0)$ is a globally optimal solution of the penalized problem
\eqref{prob:Penalized}, then $\widehat{x}$ is necessarily a globally optimal solution of the non-penalized problem
\eqref{prob:NonPenalized}. In addition, for any $x \in Q$ and $s \in K \setminus \{ 0 \}$ one has
\begin{align*}
  \omega_{\mu}(x, s) &= g_0(x) - \langle v, x - z \rangle + \mu \langle t_0, s \rangle
  > g_0(x) - \langle v, x - z \rangle + \mu_* \langle t_0, s \rangle
  \\
  &= \omega_{\mu_*}(x, s) \ge \omega_{\mu_*}(x_*, 0) = \omega_{\mu}(x_*, 0)
\end{align*}
for all $\mu > \mu_*$, that is, for any $\mu > \mu_*$ globally optimal solution of problem
\eqref{prob:Penalized} necessarily have the form $(\widehat{x}, 0)$. Thus, for any $\mu > \mu_*$ a pair 
$(\widehat{x}, \widehat{s})$ is a globally optimal solution of the penalized problem \eqref{prob:Penalized} if and
only if $\widehat{s} = 0$ and $\widehat{x}$ is a solution of the corresponding non-penalized problem. 
Since $z \in \mathscr{D}_0$ and $v \in \partial h_0(z)$ were chosen arbitrarily and $\mu_*$ does not depend on $z$ and
$v$, one can conclude that the statement of the proposition holds true.  
\end{proof}

\begin{corollary} \label{crlr:ExactPenaly_SlaterCond}
Let $K$ be finite dimensional and there exist $c \ge 0$ such that the penalty function 
$\Phi_c(\cdot) = f_0(\cdot) + c \dist(F(\cdot), - K)$ is coercive on $Q$. Then for any compact subset 
$\mathscr{D}_0 \subseteq \mathscr{D}_s$ there exists $\mu_* \ge 0$ such that for all $\mu \ge \mu_*$ and for any 
$z \in \mathscr{D}_0$ and $v \in \partial h_0(z)$ there exists a globally optimal solution of the penalized problem
\eqref{prob:Penalized} and this problem is exact, in the sense that a pair $(x_*, s_*)$ is a solution of this
problem if and only if $s_* = 0$ and $x_*$ is a solution of the corresponding non-penalized problem
\eqref{prob:NonPenalized}.
\end{corollary}

\begin{proof}
Let us verify that for any $z \in \mathscr{D}_s$ there exists $r > 0$ such that the assumptions of the previous
proposition are satisfied for $\mathscr{D}_0 = B(z, r) \cap Q$. Then one can easily verify that these assumptions are
satisfied for any compact subset $\mathscr{D}_0 \subseteq \mathscr{D}_s$. 

\textbf{Uniform error bound}. Fix any $z \in \mathscr{D}_s$ and choose some $x_0 \in Q$ such that $0 \in M_z(x_0)$. 
By the definition of the set $\mathscr{D}_s$ one has $0 \in \interior M_z(\mathbb{R}^d)$. Hence by
\cite[Thm.~1]{Robinson76} there exists $\eta > 0$ such that $\eta B_Y \subseteq M_z(x_0 + B(0, 1))$.

From the fact that $H$ is continuously Fr\'{e}chet differentiable it follows that there exists 
$r < \min\{ 1, \eta/9(1 + \| D H(z) \|) \}$ such that
\[
  \| H(u) - H(z) \| \le \frac{\eta}{9}, \quad
  \| D H(u) - D H(z) \| \le \frac{\eta}{9(2 + \| x_0 \| + \| z \|)}
\]
for any $u \in B(z, r) \cap Q$. Choose any $y \in M_{z}(x_0 + B(0, 1))$. By definition there
exist $x \in (x_0 + B(0, 1)) \cap Q$ and $w \in K$ such that $y = F_z(x) + w$. Observe that
\begin{align*}
  \| F_u(x) + w - y \| &= \| F_u(x) - F_z(x) \| 
  \le \| H(u) - H(z) \| 
  \\
  &+ \| D H(u) - D H(z) \| \| x - u \| + \| D H(z) \| \| u - z \|
  \le \frac{\eta}{3}
\end{align*}
for all $u \in B(z, r) \cap Q$, which implies that
\[
  \eta B_Y \subseteq M_z(x_0 + B(0, 1)) \subseteq
  M_u(x_0 + B(0, 1)) + \frac{\eta}{3} B_Y \quad \forall u \in B(z, r) \cap Q.
\]
Consequently, by \cite[Lemma~2]{Robinson76} one has
\begin{equation} \label{eq:UniformCoverging}
  \frac{\eta}{2} B_Y \subseteq M_u(x_0 + B(0, 1)) \quad \forall u \in B(z, r) \cap Q,
\end{equation}
which with the use of \cite[Thm.~2]{Robinson76} yields that 
\begin{equation} \label{eq:UniformMetricReg}
  \dist(x, M_u^{-1}(0)) \le \frac{\eta}{2}\big( 1 + \| x - x_0 \| \big) \dist(0, M_u(x))
\end{equation}
for all $x \in \mathbb{R}^d$ and $u \in B(z, r) \cap Q$.

\textbf{The existence of $L_g$ and $L_h$.}~Let us show that one can find $R > 0$ such that for all 
$u \in B(z, r) \cap Q$ and $v \in \partial h_0(u)$ globally optimal solutions of the problem
\begin{equation} \label{prob:NonPenalized_BoundedSolutions}
  \min \: g_0(x) - \langle v, x \rangle \quad
  \text{s.t.} \quad G(x) - H(u) - D H(u)(x - u) \preceq_K 0, \quad x \in Q
\end{equation}
(which exist by Lemma~\ref{lem:PenalizedProblem_WellPosedness}) lie in the ball $B(0, R)$. Then taking into account 
the fact that by definition $\dist(0, M_u(x)) = \dist(F_u(x), -K)$, one obtains that for all $u \in B(z, r) \cap Q$
and $v \in \partial h_0(u)$, and for any globally optimal solution $x_*$ of problem 
\eqref{prob:NonPenalized_BoundedSolutions} the following inequality holds true:
\[
  \dist(F_z(x), -K) \ge \frac{2}{\eta(2 + R + \| x_0 \|)} \dist(x, M_z^{-1}(0)) 
  \quad \forall x \in B(x_*, 1) \cap Q
\]
(cf. \eqref{eq:UniformErrorBound}). Moreover, one can take as $L_g > 0$ a Lipschitz constant of $g_0$ on the set
$B(0, R + 1)$ (recall that a convex function finite on $\mathbb{R}^d$ is Lipschitz continuous on bounded sets; see,
e.g. \cite[Thm.~10.4]{Rockafellar}), while the existence of $L_h$ such that $\| v \| \le L_h$ for all
$v \in \partial h_0(u)$ and $u \in B(z, r)$ follows from the local boundedness of the subdifferential mapping
\cite[Cor.~24.5.1]{Rockafellar}. Therefore, all assumptions of Proposition~\ref{prp:ExactPenaltySubproblem} are
satisfied for $\mathscr{D}_0 = B(z, r) \cap Q$, and one can conclude that the corollary holds true.

Thus, it remains to prove that globally optimal solutions of problem \eqref{prob:NonPenalized_BoundedSolutions} lie
within some ball $B(0, R)$. 

\textbf{The boundedness of globally optimal solutions.}~Indeed, denote $C_1 := \min\{ h_0(u) \mid u \in B(z, r) \cap Q
\}$. By the definition of subgradient one has
\[
  g_0(x) - \langle v, x - u \rangle \ge g_0(x) - h_0(x) + h_0(u) \ge f_0(x) + C_1.
\]
Furthermore, from inclusion \eqref{eq:UniformCoverging} it follows that for any $u \in B(z, r) \cap Q$ there exists
$x(u) \in x_0 + B(0, 1)$ such that $0 \in M_u(x(u))$, i.e. $x(u)$ is a feasible point of problem
\eqref{prob:NonPenalized_BoundedSolutions}. Finally, as was noted in the proof of
Lemma~\ref{lem:PenalizedProblem_WellPosedness}, the feasible region of problem
\eqref{prob:NonPenalized_BoundedSolutions} is contained in the feasible region of the problem $(\mathcal{P})$, which we
denote by $\Omega$. Therefore globally optimal solutions of problem \eqref{prob:NonPenalized_BoundedSolutions} are
contained in the set $S := \{ x \in \Omega \mid f_0(x) \le |C_1| + C_2 \}$, where
\begin{align*}
  C_2 &:= \sup_{u \in B(z, r) \cap Q} \big( g_0(x(u)) - \langle v, x(u) - u \rangle \big)
  \\0
  &\le \sup_{x \in x_0 + B(0, 1)} g_0(x) + L_h \Big( \| x_0 \| + 1 + \| z \| + r \Big) < + \infty.
\end{align*}
It remains to note that the set $S$ does not depend on $u \in B(z, r) \cap Q$ and $v \in \partial h_0(u)$, and is
contained in some ball $B(0, R)$, since $\Omega = \{ x \in Q \mid F(x) \in - K \}$ and by our assumption the penalty
function $\Phi_c(\cdot) = f_0(\cdot) + c \dist(F(\cdot), - K)$ is coercive on $Q$.  
\end{proof}

\begin{remark} \label{rmrk:ConvergenceViaExactPenalization}
Let a sequence $\{ x_n \}$ be generated by Algorithmic Pattern~\ref{alg:CCP_penalty} and suppose that there exists 
$m \in \mathbb{N}$ such that either the assumptions of Proposition~\ref{prp:ExactPenaltySubproblem} are satisfied for
some set $\mathscr{D}_0 \subseteq \mathscr{D}$ containing the sequence $\{ x_n \}_{n \ge m}$ or this sequence is
contained in a compact subset of the set $\mathscr{D}_s$ (note that since $\mathscr{D}_s$ is an open set, it is
sufficient to suppose that the sequence $x_n$ converges to a point $x_* \in \mathscr{D}_s$). Then there exists a
threshold $\tau_* > 0$ such that if for some $k \ge m$ one has $\| t_k \| \ge \tau_*$, then the sequence 
$\{ x_n \}_{n \ge k + 1}$ is feasible for the problem $(\mathcal{P})$ and coincides with a sequence generated
by Algorithmic Pattern~\ref{alg:CCP} with starting point $x_{k + 1}$. In this case one can apply
Theorem~\ref{thrm:CCP_Convergence} to analyze the behaviour of the sequence $\{ x_n \}_{n \ge k + 1}$ and its
convergence to a critical point for the problem $(\mathcal{P})$. Note that to prove this result one must suppose that
$\tau_{\max} > \tau_*$, i.e. the maximal admissible norm of the penalty parameters $t_n$ is sufficiently large.
\end{remark}

Let us give a simple example illustrating Proposition~\ref{prp:ExactPenaltySubproblem} and
Corollary~\ref{crlr:ExactPenaly_SlaterCond}, as well as behaviour of sequences generated by
Algorithmic Patterns~\ref{alg:CCP} and \ref{alg:CCP_penalty}.

\begin{example} \label{ex:AnalyticalComputDCA}
Let $d = 1$, $Y = \mathbb{R}$, and $K = \mathbb{R}_+$, i.e. $y_1 \preceq_K y_2$ means that $y_1 \le y_2$ for all
$y_1, y_2 \in \mathbb{R}$. Consider the following inequality constrained DC optimization problem:
\begin{equation} \label{prob:OneInequalConstr}
  \min \enspace (x - 0.5)^2 \quad \text{subject to} \quad x^2 - x^4 \le 0.
\end{equation}
We define $g_0(x) = (x - 0.5)^2$, $h_0(x) = 0$, $G(x) = x^2$, and $H(x) = x^4$ for all $x \in \mathbb{R}$.
The feasible region has the form $\Omega = (- \infty, -1] \cup \{ 0 \} \cup [1, + \infty)$. The points $x_* = 1$ and
$x_* = 0$ are globally optimal solutions of problem \eqref{prob:OneInequalConstr}, while the point $x_* = -1$ is a
locally optimal solution. All these points are critical, and one can easily verify that there are no other critical
points of the problem under consideration.

\textbf{Algorithmic Pattern~\ref{alg:CCP}}.~For any $z \in \mathbb{R}$ the linearized convex problem for problem
\eqref{prob:OneInequalConstr} has the form 
\begin{equation} \label{prob:OneInequalConstr_Linearized}
  \min_x \enspace (x-0.5)^2 \quad \text{subject to} \quad x^2 - z^4 - 4 z^3 (x - z) \le 0.
\end{equation}
The inequality constraint can be rewritten as follows:
\[
  \big( x - 2 z^3 \big)^2 - 4 z^4 \left( z^2 - \frac{3}{4} \right) \le 0.
\]
Therefore
\[
  \mathscr{D} = \left( - \infty, - \frac{\sqrt{3}}{2} \right] \cup \{ 0 \} \cup 
  \left[ \frac{\sqrt{3}}{2}, + \infty \right), \quad
  \mathscr{D}_s = \interior \mathscr{D},
\]
that is, the feasible region of problem \eqref{prob:OneInequalConstr_Linearized} is nonempty if and only if 
$z \in \mathscr{D}$, and Slater's condition holds true for this problem if and only if 
$z \in \interior \mathscr{D} = \mathscr{D}_s$. Furthermore, for $z = 0$ the feasible region of problem 
\eqref{prob:OneInequalConstr_Linearized} consists of the single point $x = 0$,
while for any $z \in \mathscr{D}$, $z \ne 0$, the feasible region has the form
\[
  \left[ 2 z^3 - 2 z \sqrt{z^2 - \frac{3}{4}}, 2 z^3 + 2 z \sqrt{z^2 - \frac{3}{4}} \right].
\]
As was noted multiple times above, this set is contained in the feasible region of problem
\eqref{prob:OneInequalConstr}, which implies that for any $z \ge \sqrt{3} / 2$ it is contained in $[1, + \infty)$, while
for any $z \le - \sqrt{3} / 2$ it is contained in $(- \infty, -1]$. Consequently, a sequence $\{ x_n \}$ generated
by Algorithmic Pattern~\ref{alg:CCP}, i.e. $x_{n + 1}$ is defined as an optimal solution of the problem 
\[
  \min_x \enspace (x - 0.5)^2 \quad \text{subject to} \quad x^2 - x_n^4 - 4 x_n^3 (x - x_n) \le 0,
\]
is contained in the set $(- \infty, -1]$, if $x_0 \le - 1$, and in the set $[1, + \infty)$, if $x_0 \ge 1$. In the case
$x_0 = 0$ one has $x_n \equiv 0$. Moreover, one can easily check that all assumptions of
Theorem~\ref{thrm:CCP_Convergence} are satisfied, and $x_{n + 1} > x_n$ for all $n \in \mathbb{N}$, if $x_0 < - 1$,
while $x_{n + 1} < x_n$ for all $n \in \mathbb{N}$, if $x_0 > 1$. Therefore, a sequence $\{ x_n \}$ generated by
Algorithmic Pattern~\ref{alg:CCP} converges to the locally optimal solution $x_* = -1$, if $x_0 \le - 1$, and it
converges to the
globally optimal solution $x_* = 1$, if $x_0 \ge 1$. This example shows that if the feasible region of a problem under
consideration consists of several disjoint convex components, then a sequence generated by Algorithmic
Pattern~\ref{alg:CCP} lies
within the component containing the initial guess $x_0$ and converges to a critical point from this component, i.e. a
sequence generated by Algorithmic Pattern~\ref{alg:CCP} cannot jump from one convex component of the feasible region to
another.
Let us note that one can easily prove this result in the general case.

\textbf{The penalized subproblem.}~Let us now consider Algorithmic Pattern~\ref{alg:CCP_penalty}. To this end, we
first analyze the exactness of the penalized subproblem that has the form
\begin{equation} \label{prob:OneInequalConstr_Penalized}
  \min_{(x, s)} \: (x - 0.5)^2 + \mu t_0 s \quad \text{subject to} \quad x^2 - z^4 - 4 z^3 (x - z) \le s, \quad s \ge 0,
\end{equation}
where $t_0 > 0$. One can easily verify that $(x_*, s_*)$ is a globally optimal solution of this problem if and only if
$s_* = \max\{ x^2 - z^4 - 4 z^3 (x - z), 0 \}$ and $x_*$ is a globally optimal solution of the unconstrained problem
\[
  \min \enspace (x - 0.5)^2 + \mu t_0 \max\{ x^2 - z^4 - 4 z^3 (x - z), 0 \}.
\]
Note that for $z \in \mathscr{D} \setminus \mathscr{D}_s$ this problem takes the form
\[
  \min \enspace (x - 0.5)^2 + \mu t_0 (x - 2 z^3)^2.
\]
Clearly, for any $\mu > 0$ a unique point of global minimum of this problem does not belong to the feasible region of
the corresponding non-penalized problem \eqref{prob:OneInequalConstr_Linearized}, which consists of the single point 
$\{ 2 z^3 \}$. Thus, the penalized problem \eqref{prob:OneInequalConstr_Penalized} is not exact for all 
$z \in \mathscr{D} \setminus \mathscr{D}_s$ (one can verify that this result is connected to the fact that error
bound \eqref{eq:UniformErrorBound} from Proposition~\ref{prp:ExactPenaltySubproblem} is not valid for such $z$). 

By Corollary~\ref{crlr:ExactPenaly_SlaterCond} for any compact subset $\mathscr{D}_0 \subset \mathscr{D}_s$ the
penalized problem \eqref{prob:OneInequalConstr_Penalized} is exact for all $z \in \mathscr{D}_0$, in the sense that
there exists $\mu_* \ge 0$ such that for all $\mu \ge \mu_*$ a pair $(x_*, s_*)$ is a globally optimal solution of
problem \eqref{prob:OneInequalConstr_Penalized} if and only if $s_* = 0$ and $x_*$ is a globally optimal solution of the
non-penalized problem \eqref{prob:OneInequalConstr_Linearized}. Denote the greatest lower bound of all such $\mu_*$ by
$\mu_*(\mathscr{D}_0)$. 

One can verify that problem \eqref{prob:OneInequalConstr_Penalized} is not exact for all $z \in \mathscr{D}_s$
simultaneously, due to the fact that $\mu_*(\{ z \}) \to + \infty$ as $z$ tends to the boundary of $\mathscr{D}_s$.
For the sake of shortness, we do not present a detailed proof of this result and leave it to the interested reader. Here
we only mention that this result can be proved by noting that $\mu_*(\{ z \})$ is equal to the norm of an optimal
solution of the dual problem of \eqref{prob:OneInequalConstr_Linearized} divided by $t_0$.

\textbf{Algorithmic Pattern~\ref{alg:CCP_penalty}}.~Let us now consider the performance of
Algorithmic Pattern~\ref{alg:CCP_penalty}. To this end, put $x_0 = -1$, $t_0 = 1$, $\mu = 2$, $\varkappa = 10^{-6}$, and
$\tau_{\max} = 1024$ in Algorithmic Pattern~\ref{alg:CCP_penalty}. Note that the initial point $x_0$ is critical for
problem \eqref{prob:OneInequalConstr}, but is not a globally optimal solution of this problem. Solving the penalized
problem \eqref{prob:OneInequalConstr_Penalized} with $z = x_0$ one obtains that $x_1 = -0.75$. Thus,
Algorithmic Pattern~\ref{alg:CCP_penalty}, unlike the DCA, managed to escape a convex component of the feasible region
containing the initial guess and, furthermore, to ``jump off'' from a point of local minimum. Numerical simulation
showed that the sequence $\{ x_n \}$ generated by Algorithmic Pattern~\ref{alg:CCP_penalty} converges to the point 
$x_* \approx 0.001$. If one sets $\tau_{\max} = + \infty$ and $\varkappa = 0$, then the sequence converges to the
globally optimal solution $x_* = 0$. However, note that if one chooses $t_0 \ge \mu_*(\{ -1 \}) = 1.5$, then the method
terminates after the first iteration with $x_1 = x_0$.

Thus, it seems advisable to choose $t_0$ with sufficiently small norm (and maybe even perform several iterations before
increasing the penalty parameter), to enable Algorithmic Pattern~\ref{alg:CCP_penalty} to find a better solution (see
the appendix). Moreover, even if a feasible point $x_0$ is known, it is reasonable to use Algorithmic
Pattern~\ref{alg:CCP_penalty} instead of Algorithmic Pattern~\ref{alg:CCP} due to the ability of the penalized method to
escape convex components of the feasible region and find better locally optimal solutions than the original method.
\end{example}

\subsection{Two Approaches to Convergence Analysis}
\label{sect:TwoApproaches}

In the general case, the feasible region of the non-penalized problem \eqref{prob:NonPenalizedSequence} might be empty
for all $n \in \mathbb{N}$. Then a sequence $\{ x_n \}$ generated by Algorithmic Pattern~\ref{alg:CCP_penalty} is
infeasible for the problem $(\mathcal{P})$, and Proposition~\ref{prp:ExactPenaltySubproblem} along with
Corollary~\ref{crlr:ExactPenaly_SlaterCond} do not allow one to say anything about convergence of the method.
Moreover, even if $\tau_{\max} = + \infty$, i.e. the norm of $t_n$ can increase unboundedly, there is no guarantee that
limit points of the sequence $\{ x_n \}$ are feasible for the original problem. 

To avoid such pathological cases, one usually either adopts  an `a priori approach' and supposes that a
suitable regularity assumption on constraints (``constraint qualification'') holds true at all infeasible points (this
approach was widely used, e.g. for convergence analysis of exact penalty methods in \cite{Polak}) or adopts an `a
posteriori approach' and supposes that a sequence generated by the method converges to a point, at which an appropriate
constraint qualification holds true (such approach was used, e.g. for an analysis of trust region methods in
\cite{ConnGouldToint}). For the sake of completeness, we present two convergence theorems for Algorithmic
Pattern~\ref{alg:CCP_penalty}, one of which is based on the a priori approach, while the other one is based on the a
posteriori one and was hinted at in Remark~\ref{rmrk:ConvergenceViaExactPenalization}. Both these theorems ensure the
convergence of Algorithmic Pattern~\ref{alg:CCP_penalty} to a feasible and critical point, provided $\tau_{\max}$ is
sufficiently large.

We start with the a priori approach. To this end we need to introduce the following extension of the definition of
critical point to the case of infeasible points.

\begin{definition} \label{def:GenCriticalPoint}
Let $t \succ_{K^*} 0$ be given. A point $x_* \in Q$ is said to be a \textit{generalized $t$-critical point} of 
the problem $(\mathcal{P})$, if there exist $v_* \in \partial h_0(x_*)$ and $s_* \succeq_K 0$ such that the pair
$(x_*, s_*)$ is a globally optimal solution of the problem
\begin{equation} \label{prob:GenCriticality}
\begin{split}
  &\min_{(x, s)} \enspace g_0(x) - \langle v_*, x \rangle + \langle t, s \rangle \quad
  \\
  &\text{s.t.} \enspace G(x) - H(x_*) - D H(x_*)(x - x_*) \preceq_K s, \quad s \succeq_K 0, \quad x \in Q.
\end{split}
\end{equation}
\end{definition}

Let us give two useful characterizations of the generalized criticality.

\begin{proposition} \label{prp:GeneralizedCriticality}
Let $x_* \in Q$ and $t \succ_{K^*} 0$ be given. The following statements hold true:
\begin{enumerate}
\item{$x_*$ is a generalized $t$-critical point if and only if there exist $v_* \in \partial h_0(x_*)$, 
$s_* \succeq_K 0$, and $\lambda_*, \mu_* \in K^*$ such that $F(x_*) - s_* \preceq_K 0$, $t = \lambda_* + \mu_*$, and
\[
  0 \in \partial_x L(x_*, \lambda_*) + N_Q(x_*), \quad
  \langle \lambda_*, F(x_*) - s_* \rangle = 0, \quad 
  \langle \mu_*, s_* \rangle = 0,
\]
where $L(x, \lambda) = g_0(x) - \langle v_*, x \rangle + \langle \lambda, G(x) - H(x_*) - D H(x_*)(x - x_*) \rangle$
and $\partial_x L(x_*, \lambda_*)$ is the subdifferential of $L(\cdot, \lambda_*)$ at $x_*$ in the sense of convex
analysis;
\label{stat:GenCritical_KKT}}

\item{if $x_*$ is feasible for the problem $(\mathcal{P})$ and is a generalized $t$-critical point, then $x_*$ is a
critical point for the problem $(\mathcal{P})$; conversely, if $x_*$ is a critical point for the problem $(\mathcal{P})$
satisfying optimality conditions from Corollary~\ref{crlr:OptimalityCond_Lagrange} for some $\lambda_* \in K^*$ such
that $t \succeq_{K^*} \lambda_*$, then $x_*$ is a generalized $t$-critical point.
\label{stat:GenCritical_Feasible}}
\end{enumerate}
\end{proposition}

\begin{proof}
\ref{stat:GenCritical_KKT}. Problem \eqref{prob:GenCriticality} can be rewritten as a convex cone
constrained optimization problem of the form
\begin{equation} \label{prob:ConeConstrained_Infeasible}
\begin{split}
  &\minimize_{(x, s)} \: g_0(x) - \langle v_*, x \rangle + \langle t, s \rangle \quad
  \text{subject to}  \quad x \in Q,
  \\
  &
  \widehat{F}(x, s) = \begin{pmatrix} G(x) - H(x_*) - D H(x_*)(x - x_*) - s \\ - s \end{pmatrix} \in 
  \begin{pmatrix} -K \\ -K \end{pmatrix}.
\end{split}
\end{equation}
Note that the following constraint qualification holds true for this problem:
\[
  0 \in \interior\Big\{ \widehat{F}(x, s) + K \times K \Bigm| x \in Q, \: s \in Y \Big\}.
\]
Therefore, $x_*$ is a generalized $t$-critical point if and only if there exist $v_* \in \partial h_0(x_*)$ and 
$s_* \succeq_K 0$ such that the pair $(x_*, s_*)$ satisfies the KKT optimality conditions for problem 
\eqref{prob:ConeConstrained_Infeasible} (see, e.g. \cite[Thm.~3.6]{BonnansShapiro}). Rewriting the KKT optimality
conditions in terms of problem \eqref{prob:GenCriticality} we arrive at the required result.

\ref{stat:GenCritical_Feasible}. Let $x_*$ be a generalized $t$-critical point. Then by definition there exist 
$v_* \in \partial h_0(x_*)$ and $s_* \succeq_K 0$ such that the pair $(x_*, s_*)$ is a globally optimal solution of
problem \eqref{prob:GenCriticality}. Since the point $x_*$ is feasible for the problem $(\mathcal{P})$, the pair 
$(x_*, 0)$ is feasible for problem \eqref{prob:GenCriticality}. Moreover, one 
has $g_0(x_*) \le g_0(x_*) + \langle t, s_* \rangle$, since $s_* \succeq_K 0$ and $t \succ_{K^*} 0$. Therefore, the pair
$(x_*, 0)$ is a globally optimal solution of problem \eqref{prob:GenCriticality}, which obviously implies that $x_*$ is
a globally optimal solution of problem \eqref{prob:AuxiliaryConvexProb} or, equivalently, $x_*$ is a critical point for
the problem $(\mathcal{P})$.

Suppose now that $x_*$ is a critical point for the problem $(\mathcal{P})$ satisfying optimality conditions
from Corollary~\ref{crlr:OptimalityCond_Lagrange} for some $\lambda_* \in K^*$ such that $t \succeq_{K^*} \lambda_*$.
Then one can easily verify that the pair $(x_*, 0)$ satisfies optimality conditions from the first part of this
proposition with $\mu_* = t - \lambda_*$, which implies that $x_*$ is a generalized $t$-critical point.  
\end{proof}

\begin{remark}
{(i)~From the proposition above it follows that if $x_*$ is a critical point, but the inequality $t \succeq_K \lambda_*$
is not satisfied for any corresponding Lagrange multiplier $\lambda_*$ (roughly speaking, the penalty parameter is
smaller then the norm of the Lagrange multiplier), then $x_*$ cannot be a generalized $t$-critical point. Indeed, if
$x_*$ is a generalized $t$-critical point, then from the proof of the second part of the proposition it follows that
$(x_*, 0)$ is a globally optimal solution of problem \eqref{prob:ConeConstrained_Infeasible}. Applying the KKT
optimality conditions to this problem, one gets that $t = \lambda_* + \mu_*$ for some $\mu_* \in K^*$ and some Lagrange
multiplier $\lambda_*$. Consequently, $t \succeq_{K^*} \lambda_*$, which is impossible.
}

\noindent{(ii)~With the use of the first part of the previous proposition one can readily verify that a (not necessarily
feasible) point $x_*$ is a generalized $t$-critical point for some $t \in \mathbb{R}^m$ with $t^{(i)} > 0$, 
$i \in I := \{ 1, \ldots, m \}$, of the smooth inequality constrained DC optimization problem
\[
  \min \: f_0(x) = g_0(x) - h_0(x) \quad \text{s.t.} \quad f_i(x) = g_i(x) - h_i(x) \le 0
\]
if and only if for the penalty function $\Phi_t(x) = f_0(x) + \sum_{i = 1}^m t^{(i)} \max\{ 0, f_i(x) \}$ one has
\begin{align*}
  0 \in \partial \Phi_t(x_*) = \nabla f_0(x_*) &+ \sum_{i \in I \colon f_i(x_*) > 0} t^{(i)} \nabla f_i(x_*) 
  \\
  &+ \sum_{i \in I \colon f_i(x_*) = 0} t^{(i)} \co\{ 0, \nabla f_i(x_*) \}
\end{align*}
or, equivalently, if and only if there exists $\lambda_* \in \mathbb{R}^m$ such that
\[
  \nabla f_0(x_*) + \sum_{i = 1}^m \lambda_*^{(i)} \nabla f_i(x_*) = 0, \quad
  t^{(i)} \ge \lambda_*^{(i)} \ge 0 \quad \forall i \in I,
\]
and for all $i \in I$ one has $\lambda_*^{(i)} = 0$ whenever $f_i(x_*) < 0$, while $t^{(i)} = \lambda_*^{(i)}$ whenever
$f_i(x_*) > 0$. With the use of this result one can show that the point $x_*$ is not a generalized $\mu t$-critical
point with $\mu > 1$, provided a suitable constraint qualification holds true at $x_*$. Thus, generalized
$t$-criticality depends on the choice of the penalty parameter $t$ and in many cases its increase or decrease might help
to escape a generalized $t$-critical point.
}

\noindent{(iii)~As was noted above, a generalized $t$-critical point $x_*$ is, in essence, a critical point of 
the penalty function $\Phi_t$, i.e. such point that $0 \in \partial \Phi_t(x_*)$, where $\partial \Phi_t(x)$ is the Dini
subdifferential of $\Phi_t$ at $x$. Various conditions ensuring that there are no infeasible critical points of a
penalty function were studied in detail in
\cite{DolgopolikExactPenalty,DolgopolikExactPenaltyII,DolgopolikExactPenaltyIII}.
}
\end{remark}

Before we proceed to convergence analysis, let us also establish an important property of a sequence generated by
Algorithmic Pattern~\ref{alg:CCP_penalty}, which, in particular, leads to a natural stopping criterion for this method.

\begin{lemma} \label{lem:PenalyFuncDecrease}
Let $\{ (x_n, s_n) \}$ be the sequence generated by Algorithmic Pattern~\ref{alg:CCP_penalty}. Then
\begin{equation} \label{eq:PenaltyFunctionDecay}
  f_0(x_{n + 1}) + \langle t_n, s_{n + 1} \rangle \le f_0(x_n) + \langle t_n, s_n \rangle,
  \quad \forall n \in \mathbb{N}.
\end{equation}
and this inequality is strict, if $x_n$ is not a generalized $t_n$-critical point.
\end{lemma}

\begin{proof}
By definition $(x_{n + 1}, s_{n + 1})$ is a globally optimal solution of the problem
\begin{equation} \label{prob:PenalizedSubproblem_Decay}
\begin{split}
  &\min_{(x, s)} \: g_0(x) - \langle v_n, x \rangle + \langle t_n, s \rangle
  \\
  &\text{s.t.} \quad 
  G(x) - H(x_n) - D H(x_n)(x - x_n) \preceq_K s, \quad s \succeq_K 0,  \quad x \in Q,
\end{split}
\end{equation}
while the pair $(x_n, s_n)$ satisfies the following conditions:
\[
  G(x_n) - H(x_{n - 1}) - D H(x_{n - 1})(x_n - x_{n - 1}) \preceq_K s_n, \quad
  s_n \succeq_k 0, \quad x_n \in Q.
\]
With the use of Lemma~\ref{lem:GenConvexityViaDerivative} one obtains that $G(x_n) - H(x_n) \preceq_K s_n$, which
implies that $(x_n, s_n)$ is a feasible point of problem \eqref{prob:PenalizedSubproblem_Decay}. Therefore
\begin{equation} \label{eq:LinearizedPenaltyDecay}
  g_0(x_{n + 1}) - \langle v_n, x_{n + 1} - x_n \rangle + \langle t_n, s_{n + 1} \rangle
  \le g_0(x_n) + \langle t_n, s_n \rangle \quad \forall n \in \mathbb{N}.
\end{equation}
Subtracting $h_0(x_n)$ from both sides of this inequality and applying the definition of subgradient, one obtains that
inequality \eqref{eq:PenaltyFunctionDecay} holds true. It remains to note that if $x_n$ is not a generalized
$t_n$-critical point, then by definition the inequality in \eqref{eq:LinearizedPenaltyDecay} is strict, which implies
that inequality \eqref{eq:PenaltyFunctionDecay} is also strict.  
\end{proof}

\begin{remark} \label{rmrk:PenaltyCCP_StoppinCriterion}
From the lemma above it follows that one can use the inequality
\begin{equation} \label{eq:BasicStoppingCriterion}
  \Big| f_0(x_{n + 1}) + \langle t_n, s_{n + 1} \rangle - f_0(x_n) - \langle t_n, s_n \rangle \Big| \le \varepsilon
\end{equation}
along with the inequality $\| s_{n + 1} \| \le \varepsilon$ on the infeasibility measure as a stopping criterion for
Algorithmic Pattern~\ref{alg:CCP_penalty}. Taking into account \eqref{eq:LinearizedPenaltyDecay} one can replace
inequality \eqref{eq:BasicStoppingCriterion} with the following one
\[
  \Big| g_0(x_{n + 1}) - \langle v_n, x_{n + 1} \rangle + \langle t_n, s_{n + 1} \rangle - 
  \Big( g_0(x_n) - \langle v_n, x_n \rangle  + \langle t_n, s_n \rangle \Big) \Big| \le \varepsilon
\]
to avoid the computation of $h_0(x_n)$ and $h_0(x_{n + 1})$ (cf. Remark~\ref{rmrk:CCP_StoppingCriterion}).
\end{remark}

Now we can provide sufficient conditions for the convergence of a sequence generated by Algorithmic
Pattern~\ref{alg:CCP_penalty}
to a feasible and critical point for the problem $(\mathcal{P})$, based on the a priori approach to convergence
analysis.

\begin{theorem}
Let the space $Y$ be finite dimensional, the cone $K$ be generating, and the penalty function 
$\Phi_c(x) = f_0(x) + c \dist(F(x), -K)$ be bounded below on $Q$ for 
$c = \min\{ \langle t_0, s \rangle \mid s \in K, \: \| s \| = 1 \} > 0$. Then all limits points of the sequence
$\{ x_n \}$ generated by Algorithmic Pattern~\ref{alg:CCP_penalty} are generalized $t_*$-critical points of the problem
$(\mathcal{P})$ with $t_* = \lim t_n$. 

Suppose, in addition, that all points from the set 
\[
  \Big\{ x \in Q \Bigm| \dist(F(x), -K) > \varkappa \Big\}
\]
are not generalized $\widehat{t}$-critical points with $\widehat{t} = \mu^p t_0$, where $p \in \mathbb{N}$ is the
largest natural number satisfying the inequality $\| \mu^p t_0 \| \le \tau_{\max}$. Then all limit points $x_*$ of the
sequence $\{ x_n \}$ satisfy the inequality $\dist(F(x_*), -K) \le \varkappa$. In particular, if $\varkappa = 0$, then
all limits points of the sequence $\{ x_n \}$ are feasible and critical for the problem $(\mathcal{P})$.
\end{theorem}

\begin{proof}
For the sake of convenience we divide the proof of the theorem into several parts.

\textbf{Proof of the second statement.} Suppose that the first part of the theorem holds true, i.e. all limit points of
the sequence $\{ x_n \}$ are generalized $t_*$-critical points with $t_* = \lim t_n$ (note that this limit exists, since
according to Step~3 of Algorithmic Pattern~\ref{alg:CCP_penalty} the penalty parameter can be updated only a finite
number of times). Let us show that the second part of the theorem holds true.

Indeed, let $x_*$ be a limit point of the sequence $\{ x_n \}$, that is, there exists a subsequence $\{ x_{n_k} \}$
converging to $x_*$. Let us consider two cases. Suppose at first that the norm of the penalty parameter $t_n$ does not
reach the upper bound $\tau_{\max}$ (see Step~3 of Algorithmic Pattern~\ref{alg:CCP_penalty}), that is, the penalty
parameter is updated less than $p$ times. Then according to the penalty updating rule on Step~3 of Algorithmic
Pattern~\ref{alg:CCP_penalty} there exists $n_0 \in \mathbb{N}$ such that $\| s_n \| < \varkappa$ for all $n \ge n_0$.
By definition
\[
  G(x_n) - H(x_{n - 1}) - D H(x_{n - 1})(x_n - x_{n - 1}) \preceq_K s_n \quad \forall n \in \mathbb{N},
\]
which thanks to Lemma~\ref{lem:GenConvexityViaDerivative} implies that $F(x_n) \preceq_K s_n$ or, equivalently, one has
$F(x_n) - s_n \in - K$. Therefore $\dist(F(x_n), -K) \le \| s_n \| < \varkappa$ for all $n \ge n_0$. Consequently, 
passing to the limit in the inequality $\dist(F(x_{n_k}), - K) < \varkappa$ with the use of the fact that both $G$ and
$H$ are continuous, one obtains that $\dist(F(x_*), - K) \le \varkappa$.

Suppose now that the norm of $t_n$ reaches the upper bound $\tau_{\max}$ after a finite number of iterations. Then
according to Step~3 of Algorithmic Pattern~\ref{alg:CCP_penalty} there exists $n_0 \in \mathbb{N}$ such that 
$t_n = \mu^p t_0$ for all $n \ge n_0$. By our assumption $x_*$ is a generalized $t_*$-critical point with 
$t_* = \widehat{t} = \mu^p t_0$, which implies that it cannot belong to the set 
$\{ x \in Q \colon \dist(F(x), -K) > \varkappa \}$ by the assumption of the theorem. Therefore, 
$\dist(F(x_*), -K) \le \varkappa$.

Finally, if $\varkappa = 0$, then $\dist(F(x_*), -K) = 0$, that is, $F(x_*) \in - K$, since the cone $K$ is closed.
Consequently, the point $x_*$ is feasible for the problem $(\mathcal{P})$. Hence by the second part of
Proposition~\ref{prp:GeneralizedCriticality} the point $x_*$ is also a critical for 
the problem $(\mathcal{P})$.

Thus, it remains to prove that all limit points of the sequence $\{ x_n \}$ are generalized $t_*$-critical points of
the problem $(\mathcal{P})$. 

\textbf{Proof of the first statement.}~Let a subsequence $\{ x_{n_k} \}$ converge to some point $x_*$. Then the
corresponding sequence $\{ v_{n_k} \}$ of subgradients  of the function $h_0$ is bounded due to the local boundedness of
the subdifferential mapping \cite[Cor.~24.5.1]{Rockafellar}. Therefore, replacing, if necessary, the sequence 
$\{ x_{n_k} \}$ with its subsequence, one can suppose that the sequence $\{ v_{n_k} \}$ converges to some vector $v_*$
belonging to $\partial h_0(x_*)$ by virtue of the fact that the graph of the subdifferential is closed
\cite[Thm.~24.4]{Rockafellar}.

\textbf{Step~1.}~Let us show that the sequence $\{ s_{n_k} \} \subset K$ is bounded. Then taking into account the facts
that the space $Y$ is finite dimensional and the cone $K$ is closed, and replacing, if necessary, the sequence 
$\{ x_{n_k} \}$ with its subsequence, one can suppose that $\{ s_{n_k} \}$ converges to some $s_* \in K$. 

Indeed, since the penalty parameter $t_n$ can be updated only a finite number of times, there exists 
$n_0 \in \mathbb{N}$ such that $t_n = t_{n_0}$ for all $n \ge n_0$. Consequently, by Lemma~\ref{lem:PenalyFuncDecrease}
the sequence $\{ f_0(x_n) + \langle t_{n_0}, s_n \rangle \}_{n \ge n_0}$ is non-increasing and, in particular, bounded
above. Therefore the sequence $\{ f_0(x_{n_k}) + \langle t_0, s_{n_k} \rangle \}$ is bounded above as well.

By contradiction, suppose that the sequence $\{ s_{n_k} \}$ is unbounded. Then taking into account the fact that 
the sequence $\{ f_0(x_{n_k}) \}$ is bounded below, since the sequence $\{ x_{n_k} \}$ converges, and applying 
the inequality
\[
  f_0(x_{n_k}) + \langle t_0, s_{n_k} \rangle \ge f_0(x_{n_k}) + c \| s_{n_k} \|,
\]
one gets that $\limsup_{k \to \infty} (f_0(x_{n_k}) + \langle t_0, s_{n_k} \rangle) = + \infty$, which is impossible.
Thus, without loss of generality one can suppose that the sequence $\{ s_{n_k} \}$ converges to some $s_*$. Note that
from the definition of $(x_n, s_n)$ and Lemma~\ref{lem:GenConvexityViaDerivative} it follows that 
$F(x_n) \preceq_K s_n$. Therefore $F(x_*) \preceq_K s_*$, thanks to the fact that the cone $K$ is closed. 

\textbf{Step~2.}~Now we can turn to the proof of the fact that the point $x_*$ is a generalized $t_*$-critical point.
By contradiction, suppose that this statement is false. Then, in particular, the point $(x_*, s_*)$ is not a globally
optimal solution of problem \eqref{prob:GenCriticality} (see Def.~\ref{def:GenCriticalPoint}). Therefore there exist a
feasible point $(\overline{x}, \overline{s})$ of this problem and $\theta > 0$ such that
\[
  g_0(\overline{x}) - \langle v_*, \overline{x} - x_* \rangle + \langle t_*, \overline{s} \rangle
  < g_0(x_*) + \langle t_*, s_* \rangle - \theta.
\]
Applying Lemma~\ref{lem:PerturbedSolution} with $X = \mathbb{R}^d \times Y$ and
\[
  \Phi(x, s) = \begin{pmatrix} G(x) - s \\ - s \end{pmatrix}, \quad
  \Psi(x, s) = \begin{pmatrix} H(x) \\ 0 \end{pmatrix}, \quad
  W = Q \times Y, \quad 
  E = K \times K,
\]
(it is easy to see that condition \eqref{eq:RegularPointConeConstr} holds true in the case) one obtains that for any 
$z = (x, s) \in Q \times Y$ lying in a neighbourhood of $(x_*, s_*)$ one can find $(\xi(z), \zeta(z)) \in Q \times K$
such that 
\[
  G(\xi(z)) - H(x) - D H(x)(\xi(z) - x) \preceq_K \zeta(z)
\]
and $(\xi(z), \zeta(z)) \to (\overline{x}, \overline{s})$ as $z \to (x_*, s_*)$. Consequently, there exists 
$k_0 \in \mathbb{N}$ such that for any $k \ge k_0$ the point $(\xi(z_{n_k}), \zeta(z_{n_k}))$ with 
$z_{n_k} = (x_{n_k}, s_{n_k})$ is feasible for the problem
\begin{align*}
  &\min_{(x, s)} \: g_0(x) - \langle v_{n_k}, x \rangle + \langle t_{n_k}, s \rangle 
  \\
  &\text{s.t.} \quad G(x) - H(x_{n_k}) - D H(x_{n_k})(x - x_{n_k}) \preceq_K s, \quad s \succeq_K 0, \quad x \in Q.
\end{align*}
(note that one can suppose that $t_{n_k} = t_*$, since the penalty parameter is updated only a finite number of times)
and
\[
  g_0(\xi(z_{n_k})) - \langle v_{n_k}, \xi(z_{n_k}) - x_{n_k} \rangle + 
  \langle t_*, \zeta(z_{n_k}) \rangle
  < g_0(x_{n_k}) + \langle t_*, s_{n_k} \rangle - \frac{\theta}{2}.
\]
Therefore by the definition of $(x_n, s_n)$ for any $k \ge k_0$ one has
\[
  g_0(x_{n_k + 1}) - \langle v_{n_k}, x_{n_k + 1} - x_{n_k} \rangle + \langle t_*, s_{n_k + 1} \rangle 
  < g_0(x_{n_k}) + \langle t_*, s_{n_k} \rangle - \frac{\theta}{2}.
\]
Subtracting $h_0(x_{n_k})$ from both sides of this inequality and applying the definition of subgradient, one obtains
that
\[
  f_0(x_{n_k + 1}) + \langle t_*, s_{n_k + 1} \rangle < f_0(x_{n_k}) + \langle t_*, s_{n_k} \rangle - \frac{\theta}{2}
  \quad \forall k \ge k_0,
\]
which with the use of Lemma~\ref{lem:PenalyFuncDecrease} implies that 
$f_0(x_n) + \langle t_*, s_n \rangle \to - \infty$ as $n \to \infty$ (recall that $t_* = t_n$ for any sufficiently
large $n$, since the penalty parameter can be updated only a finite number of times). On the other hand, as was shown
above (see the proof of Lemma~\ref{lem:PenalizedProblem_WellPosedness}), one has
\[
  f_0(x_n) + \langle t_*, s_n \rangle \ge f_0(x_n) + \langle t_0, s_n \rangle
  \ge f_0(x_n) + c \dist(F(x_n), -K) =: \Phi_c(x_n).
\]
Consequently, $\Phi_c(x_n) \to - \infty$, which contradicts the fact that by our assumption this function is bounded
below on $Q$. Therefore one can conclude that $x_*$ is a generalized $t_*$-critical point.  
\end{proof}

\begin{remark}
In the previous theorem it is sufficient to suppose that the penalty function $\Phi_c$ is bounded below on
$Q$ for $c = \inf\{ \langle t_*, s \rangle \mid s \in K, \: \| s \| = 1 \}$, which is, in the general case, greater
than $c$ from the formulation of the theorem. However, such assumption is inconsistent with the a priori approach, since
it is based on the information about the behaviour of the sequence $\{ t_n \}$, which is not known in advance.
\end{remark}

Finally, let us consider the a posteriori approach to convergence analysis, which allows one to obtain sufficient
conditions for the convergence of Algorithmic Pattern~\ref{alg:CCP_penalty} to a critical point for the problem
$(\mathcal{P})$.

\begin{theorem}
Let $K$ be finite dimensional and there exist $c \ge 0$ such that the penalty function 
$\Phi_c(\cdot) = f_0(\cdot) + c \dist(F(\cdot), - K)$ is coercive on $Q$. Suppose also that the sequence 
$\{ x_n \}$ generated by Algorithmic Pattern~\ref{alg:CCP_penalty} with $\varkappa = 0$ and $\tau_{\max} = + \infty$
converges to a point $x_*$ satisfying the following constraint qualification:
\begin{equation} \label{eq:SlaterCondition_Penalty_LimPoint}
  0 \in \interior\big\{ G(x) - H(x_*) - D H(x_*)(x - x_*) + K \bigm| x \in Q \big\}
\end{equation}
(i.e. $x_* \in \mathscr{D}_s$). Then the sequence $\{ t_n \}$ is bounded, there exists $m \in \mathbb{N}$ such that
for all $n \ge m$ the point $x_n$ is feasible for the problem $(\mathcal{P})$, and the point $x_*$ is
feasible and critical for the problem $(\mathcal{P})$.
\end{theorem}

\begin{proof}
By our assumption $x_* \in \mathscr{D}_s$. Therefore, as was shown in the proof of
Corollary~\ref{crlr:ExactPenaly_SlaterCond}, there exist $r > 0$ and $\mu_* \ge 0$ such that for all $\mu \ge \mu_*$ and
for any $z \in B(x_*, r) \cap Q$ and $v \in \partial h_0(z)$ the penalized problem \eqref{prob:Penalized} is exact.
Define $\tau_* = \mu_* \| t_0 \|$. 

\textbf{Step~1.}~If the penalty parameter $t_n$ is updated only a finite number of times, then the sequence $\{ t_n \}$
is obviously bounded. Moreover, according to Step~3 of Algorithmic Pattern~\ref{alg:CCP_penalty} in this case there
exists 
$m \in \mathbb{N}$ such that $s_n = 0$ for all $n \ge m$, which implies that the sequence $\{ x_n \}_{n \ge m}$ is
feasible for the problem $(\mathcal{P})$. Therefore the point $x_*$ is also feasible for this problem, due to the fact
that under our assumptions the feasible region of the problem $(\mathcal{P})$ is closed.

On the other hand, if the penalty parameter $t_n$ is updated an infinite number of times, then according to Step~3 of
Algorithmic Pattern~\ref{alg:CCP_penalty} there exists $m \in \mathbb{N}$ such that $\| t_n \| \ge \tau_*$ for all 
$n \ge m$. Moreover, increasing $m$, if necessary, one can suppose that $x_n \in B(x_*, r)$ for all $n \ge m$.
Consequently, the penalized subproblem on Step~2 of Algorithmic Pattern~\ref{alg:CCP_penalty} is exact for all $n \ge m$
by
Corollary~\ref{crlr:ExactPenaly_SlaterCond}. Hence by the definition of exactness $s_n = 0$ for all $n \ge m + 1$, which
contradicts our assumption that $t_n$ is updated an infinite number of times.

Thus, the sequence $\{ t_n \}$ is bounded and the point $x_*$ is feasible for the problem $(\mathcal{P})$. It remains to
verify that $x_*$ is a critical point. 

\textbf{Step~2.}~Suppose at first that there exists $m \in \mathbb{N}$ such that $\| t_m \| \ge \tau_*$. Then 
$\| t_n \| \ge \tau_*$ for all $n \ge m$. Increasing $m$, if necessary, one can suppose that $x_n \in B(x_*, r)$ for all
$n \ge m$. Therefore by the definitions of exactness of the penalized problem and Algorithmic Patterns~\ref{alg:CCP} and
\ref{alg:CCP_penalty} the sequence $\{ x_n \}_{n \ge m + 1}$ is feasible for the problem $(\mathcal{P})$ and coincides
with the sequence generated by Algorithmic Pattern~\ref{alg:CCP} with starting point $x_{m + 1}$. Therefore, by
Theorem~\ref{thrm:CCP_Convergence} the point $x_*$ is critical for the problem $(\mathcal{P})$.

\textbf{Step~3.}~Suppose now that $\| t_n \| < \tau_*$ for all $n \in \mathbb{N}$. Then there exists 
$n_0 \in \mathbb{N}$ such that $t_n = t_{n_0}$ for all $n \ge n_0$. Since the sequence $\{ x_n \}$ generated by
Algorithmic Pattern~\ref{alg:CCP_penalty} converges to $x_*$, the corresponding sequence $\{ v_n \}$ of subgradients of
the function $h_0$ is bounded, thanks to the local boundedness of the subdifferential mapping
\cite[Cor.~24.5.1]{Rockafellar}. Consequently, there exists a subsequence $\{ v_{n_k} \}$ converging to some vector
$v_*$, which belongs to $\partial h_0(x_*)$ due to closedness of the graph of the subdifferential
\cite[Thm.~24.4]{Rockafellar}.

By contradiction, suppose that $x_*$ is not a critical point for the problem $(\mathcal{P})$. As was noted several times
above, it implies that $x_*$ is not a globally optimal solution of the problem
\begin{align*}
  &\minimize \enspace g_0(x) - \langle v_*, x \rangle
  \\
  &\text{subject to} \enspace G(x) - H(x_*) - D H(x_*)(x - x_*) \preceq_K 0, \quad x \in Q.
\end{align*}
Thus, there exist $\theta > 0$ and a feasible point $\overline{x}$ of this problem satisfying the inequality
$g_0(\overline{x}) - \langle v_*, \overline{x} - x_* \rangle < g_0(x_*) - \theta$. 

Applying Lemma~\ref{lem:PerturbedSolution} with $\Phi = G$, $\Psi = H$, $W = Q$, and $E = K$,
one obtains that for any $z \in Q$ lying in a neighbourhood of $x_*$ one can find $\xi(z) \in Q$ such that
$G(\xi(z)) - H(z) - D H(z)(\xi(z) - z) \preceq_K 0$ and $\xi(z) \to \overline{x}$ as $z \to x_*$. Hence bearing in mind
the facts that $x_{n_k} \to x_*$ and $v_{n_k} \to v_*$ as $k \to \infty$, one obtains that there exists 
$k_0 \in \mathbb{N}$ such that
\[
  g_0(\xi(x_{n_k})) - \langle v_{n_k}, \xi(x_{n_k}) - x_{n_k} \rangle 
  < g_0(x_{n_k}) - \frac{\theta}{2}	\quad \forall k \ge k_0.
\]
Clearly, one can suppose that $n_{k_0} \ge n_0$.

Recall that by definition $(x_{n_k + 1}, s_{n_k + 1})$ is a globally optimal solution of the penalized problem
\begin{align*}
  &\min_{(x, s)} \: g_0(x) - \langle v_{n_k}, x - x_{n_k} \rangle + \langle t_{n_k}, s \rangle
  \\
  &\text{s.t.} \quad 
  G(x) - H(x_{n_k}) - D H(x_{n_k})(x - x_{n_k}) \preceq_K s, \quad s \succeq_K 0,  \quad x \in Q.
\end{align*}
By definition the point $(\xi(x_{n_k}), 0)$ is feasible for this problem, which implies that
\begin{align*}
  g_0(x_{n_k + 1}) &- \langle v_{n_k}, x_{n_k + 1} - x_{n_k} \rangle + \langle t_{n_k}, s_{n_k + 1} \rangle
  \\
  &\le g_0(\xi(x_{n_k})) - \langle v_{n_k}, \xi(x_{n_k}) - x_{n_k} \rangle
  < g_0(x_{n_k}) - \frac{\theta}{2}
\end{align*}
for all $k \ge k_0$. Subtracting $h_0(x_{n_k})$ from both sides of this inequality and applying the definition of
subgradient and the fact that $t_n = t_{n_0}$ for all $n \ge n_0$, one obtains that
\[
  f_0(x_{n_k + 1}) + \langle t_{n_0}, s_{n_k + 1} \rangle < f_0(x_{n_k}) - \frac{\theta}{2}
  \le f_0(x_{n_k}) + \langle t_{n_0}, s_{n_k} \rangle - \frac{\theta}{2}
\]
for all $k \ge k_0$ (here we used the facts that by definition $s_{n_k} \in K$ and $\langle t_n, s \rangle \ge 0$ for
all $s \in K$ and $n \in \mathbb{N}$). By Lemma~\ref{lem:PenalyFuncDecrease} one has
\[
  f_0(x_{n + 1}) + \langle t_{n_0}, s_{n + 1} \rangle
  \le f_0(x_n) + \langle t_{n_0}, s_n \rangle \quad \forall n \ge n_0.
\]
Consequently, $f_0(x_n) + \langle t_{n_0}, s_n \rangle \to - \infty$ as $n \to \infty$, which contradicts
the facts that $x_n \to x_*$ as $n \to \infty$ and $f_0(x_n) + \langle t_{n_0}, s_n \rangle \ge f_0(x_n)$ 
for all $n \in \mathbb{N}$. Therefore, $x_*$ is a critical point, and the proof is complete.  
\end{proof}

Thus, one can conclude that if a sequence $\{ x_n \}$ generated by either Algorithmic Pattern~\ref{alg:CCP} or
Algorithmic Pattern~\ref{alg:CCP_penalty} converges to a point $x_*$ such that Slater's condition holds true for the
corresponding linearized convex problem, then under some natural assumptions the point $x_*$ is critical for 
the problem $(\mathcal{P})$.

\section{Numerical experiments}
\label{sect:NumerExperiments}

Let us present some results of numerical experiments for Algorithmic Pattern~\ref{alg:CCP_penalty}. We applied it to the
problem of computing compressed modes for variational problems \cite{OzolinsLaiCaflischOsher} and the sphere packing
problem on Grassmannian \cite{AbsilHosseini,DirrHelmkeLageman,ConwayHardinSloane}. Local search methods for the first
problem were considered in \cite{OzolinsLaiCaflischOsher,ChenJiYou,ChenMaSoZhang}, while local search methods for
the second problem were studied in \cite{DirrHelmkeLageman,GrohsHosseini,GrohsHosseini2,ConwayHardinSloane}. For an
interesting application of Algorithmic Pattern~\ref{alg:CCP_penalty} to multi-matrix principal component analysis see
\cite[Sect.~5.4]{LippBoyd}.

Algorithmic Pattern~\ref{alg:CCP_penalty} was implemented in \textsc{Matlab} on a 3.7 GHz Intel(R) Core(TM) i3 machine
with 16 GB of RAM. The parameters of the algorithmic pattern were chosen as follows. The penalty parameter 
$t_0 \succ_{K^*} 0$ was always chosen as the identity matrix $I_k$ of appropriate dimension $k \in \mathbb{N}$, which
means that for any $n \in \mathbb{N}$ one has $t_n = \tau_n I_k$ for $\tau_n = \mu^l$ and some $l \in \mathbb{N}$. We
also set $\tau_{\max} = 10^6 \| t_0 \|$, $\mu = 10$, and $\varkappa = 10^{-4}$, and used the first stopping criterion
from Remark~\ref{rmrk:PenaltyCCP_StoppinCriterion} with $\varepsilon = 10^{-3}$. Finally, we terminated the algorithm,
if the number of iterations exceeded $100$.

To numerically verify the observations made in Example~\ref{ex:AnalyticalComputDCA}, the rule for updating the penalty
parameter on Step~3 of Algorithmic Pattern~\ref{alg:CCP_penalty} was modified as follows:
\begin{equation} \label{eq:ModifiedPenaltyUpdate}
  t_{n + 1} = \begin{cases}
    \mu t_n, & \text{if } n \ge n_{\min} \text{ and } \| s_{n + 1} \| \ge \varkappa \text{ and } 
    \mu \| t_n \| \le \tau_{\max},
    \\
    t_n, & \text{otherwise},
  \end{cases}
\end{equation}
Here $n_{\min} \in \mathbb{N}$ is a parameter that defines the number of iterations, after which the algorithm starts
updating the penalty parameter. We tested 3 different values $n_{\min} = 0$, $n_{\min} = 3$, and $n_{\min} = 10$, to
determine how the value of $n_{\min}$ affects the overall performance of the method. It should be noted that the 
convergence analysis presented in the previous sections is applicable to Algorithmic Pattern~\ref{alg:CCP_penalty}
with the penalty updating rule \eqref{eq:ModifiedPenaltyUpdate}, since in this case the sequence 
$\{ x_n \}_{n \ge n_{\min}}$ coincides with the sequence generated by the original version of Algorithmic
Pattern~\ref{alg:CCP_penalty} with $x_{n_{\min}}$ chosen as the starting point.

Finally, the convex subproblems on Step~2 of the method were solved with the use of \texttt{cvx}, a {\sc Matlab} package
for specifying and solving convex programs \cite{CVXPackage,GrantBoyd_CVX}. This package was used as the inferface to
the SDPT3 solver \cite{TohToddTutuncu,TutuncuTohTodd}, a \textsc{Matlab} software package for semidefinite programming
based on infeasible path-following methods (a class of interior point methods). An attempt was also made to solve the
convex subproblems on Step~2 with the use of PENLAB \cite{Penlab}, an open source \textsc{Matlab} package for solving
nonlinear semidefinite optimization problems based on an augmented Lagrangian method from \cite{Stingl}. However, 
compared to \texttt{cvx}, PENLAB was harder to deploy and did not perform well in our test problems, possibly because
the augmented Lagrangian method implemented in PENLAB cannot efficiently exploit convexity.

\subsection{Compressed modes for variational problems}

In paper \cite{OzolinsLaiCaflischOsher}, the following \textit{nonsmooth} optimization problem for computing spatially
localized (``sparse'') solutions, called compressed modes, to a class of problems in mathematical physics was proposed:
\begin{equation} \label{prob:CompressedModes}
\begin{split}
  &\minimize_{\Psi \in \mathbb{R}^{N \times d}} \enspace f_0(\Psi) = \trace(\Psi^T A \Psi) + \nu \| \Psi \|_1 \\
  &\mathrm{subject~to} \enspace \Psi^T \Psi = I_d,
\end{split}
\end{equation}
Here columns of the matrix $\Psi$ are discretized compressed modes, $\Psi^T$ is the transpose of $\Psi$, $N$ is 
the number of discretization nodes, $\nu \ge 0$ is a parameter influencing the sparsity of solutions, 
$\| \Psi \|_1 = \sum_{i = 1}^N \sum_{j = 1}^d |\Psi_{ij}|$, $\trace(\cdot)$ is the trace of a square matrix, and $A$ is
the discretized Schr\"{o}dinger operator $- \frac{1}{2} \Delta + V(x)$, where $\Delta$ is the Laplace operator and
$V(x)$ is a potential.

Following \cite{OzolinsLaiCaflischOsher}, we consider problem \eqref{prob:CompressedModes} for the Kronig-Penney model
of a periodic one-dimensional crystal \cite{KronigPenney} on segment $[0, 50]$, in which the rectangular wells are
replaced by Gaussian potentials for the sake of simplicity, so that the potential is given by 
\[
  V(x) = - V_0 \sum_{j = 1}^{N_{el}} \exp\Big( - \frac{(x - y_j)^2}{2 \delta^2} \Big) \quad \forall x \in [0, 50].
\]
We chose the same values $V_0 = 1$, $N_{el} = 5$, $\delta = 3$, and $y_j = 10 j$ as in \cite{OzolinsLaiCaflischOsher},
and also discretized the segment $[0, 50]$ with $N = 128$ equally spaced nodes 
$x_1 = 0, x_2 = \sigma, x_3 = 2 \sigma, \ldots, x_{128} = 50$ as was done in \cite{OzolinsLaiCaflischOsher} (here 
$\sigma = 50/127$ is the length of the discretization interval). We also set $d = 5$, which makes the dimension of 
the problem equal to $640$.

For our choice of parameters, the matrix $A$ of the discretized Schr\"{o}dinger operator is neither positive nor
negative semidefinite. However, it can be easily rewritten as the difference of two positive semidefinite matrices
$A = A_{\Delta} - A_V$ with
\[
  A_{\Delta} = \frac{1}{2 \sigma^2} \begin{pmatrix}
    2 & -1 & 0 & 0 & \cdots & 0 & 0 & 0 & -1
    \\
    -1 & 2 & -1 & 0 & \cdots & 0 & 0 & 0 & 0
    \\
    \vdots & \vdots & \vdots & \vdots & \ddots & \vdots & \vdots & \vdots & \vdots
    \\
    0 & 0 & 0 & 0 & \cdots & 0 & -1 & 2 & -1
    \\
    -1 & 0 & 0 & 0 & \cdots & 0 & 0 & -1 & 2
  \end{pmatrix}
\]
and $A_V = \diag(-V(x_1), \ldots, -V(x_N))$, where $A_{\Delta}$ is the matrix of the discretized operator 
$- \frac{1}{2} \Delta$ with periodic boundary conditions and $\diag(\cdot)$ is the diagonal matrix. Thus, one can use
the DC decomposition of the objective function $f_0(\Psi) = g_0(\Psi) - h_0(\Psi)$ with
\[
  g_0(\Psi) = \trace(\Psi^T A_{\Delta} \Psi) + \nu \| \Psi \|_1, \quad h_0(\Psi) = \trace(\Psi^T A_V \Psi).
\]
Finally, following Example~\ref{ex:StiefelManifold}, the orthogonality constraint $\Psi^T \Psi = I_d$ was rewritten as
two semidefinite constraints
\[
  G(\Psi) \preceq \mathbb{O}_d, \quad - G(\Psi) \preceq \mathbb{O}_d, \quad
  G(\Psi) := \Psi^T \Psi - I_d,
\]
with convex mapping $G$ (it should be mentioned that this strategy precludes constraint qualification
\eqref{eq:SlaterCondition_Penalty_LimPoint} from holding). Here $\mathbb{O}_d$ is the zero matrix of order $d$ and
$\preceq$ is the L\"{o}wner partial order on the space $\mathbb{S}^d$ of all real symmetric matrices of order 
$d \in \mathbb{N}$, i.e. $A \preceq B$ for $A, B \in \mathbb{S}^d$ if and only if $B - A$ is positive semidefinite.

\begin{remark}
It should be noted that since \texttt{cvx} package cannot solve problems with nonlinear semidefinite constraints
(even if they are convex), we transformed the nonlinear convex constraint $G(\Psi) \preceq s$ with 
$s \succeq \mathbb{O}_d$, arising on Step~2 of Algorithmic Pattern~\ref{alg:CCP_penalty}, into the following equivalent
\textit{linear} matrix inequality
\[
  \begin{pmatrix} I_d + s & \Psi^T \\ \Psi & I_N \end{pmatrix} \preceq \mathbb{O}_{d + N}
\]
with the use of the Schur Complement Lemma (see, e.g. \cite[Sect.~A.5.5]{BoydVandenberghe}).
\end{remark}

\begin{figure}[t] 
\includegraphics[width=0.5\textwidth]{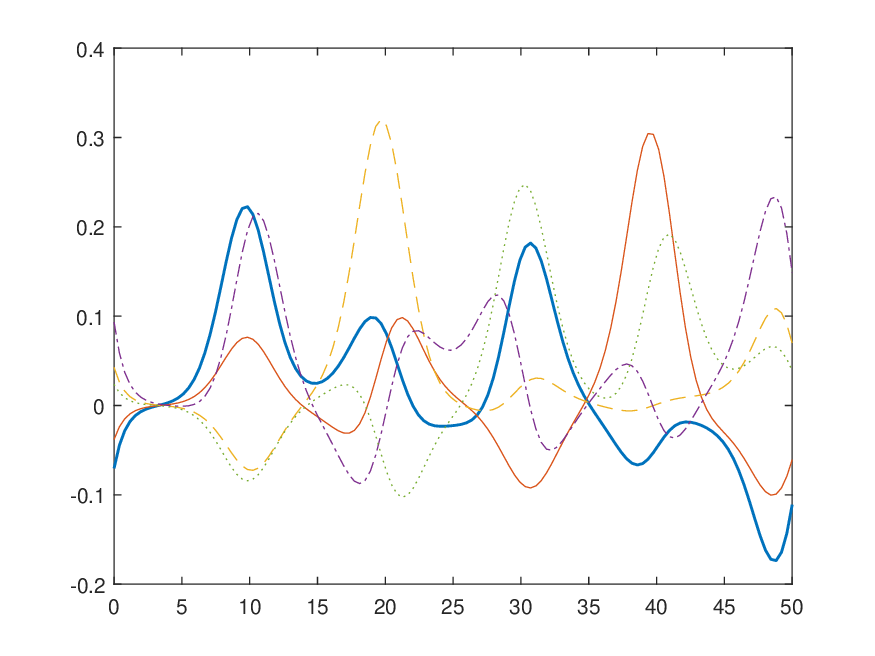}
\includegraphics[width=0.5\textwidth]{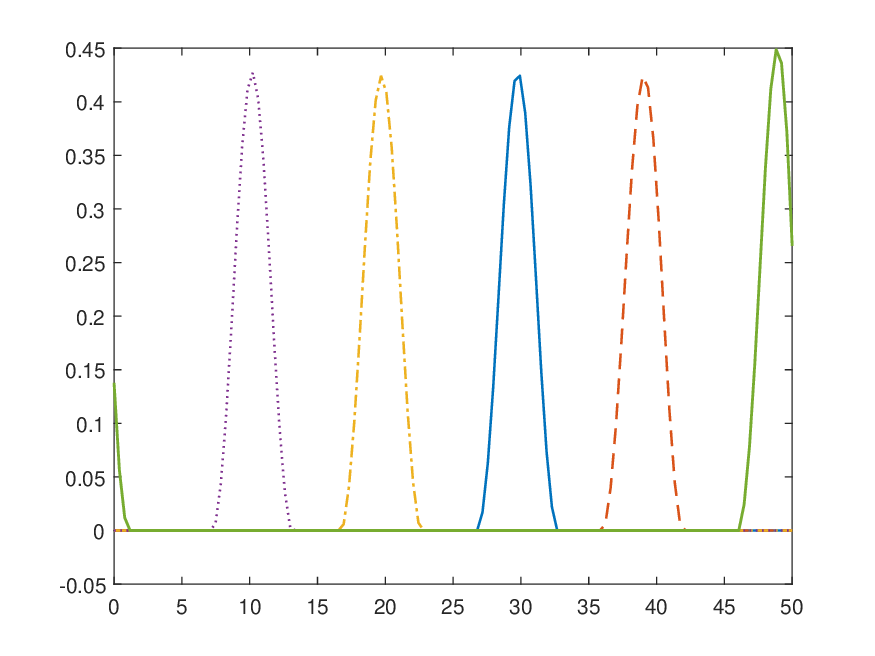}
\caption{The solutions of problem \eqref{prob:CompressedModes} (the columns of matrix $\Psi$) with 
$\nu = 0$ (left figure) and $\nu = 0.2$ (right figure) corresponding to the best value of the objective function
computed in our numerical experiments.}
\label{fig:CompressedModes}
\end{figure}

\begin{table} [ht!]
\caption{The results of numerical experiments for $10$ randomly generated starting points in the case $\nu = 0$.}
\begin{tabular}{| c | c | c | c | c |} 
  \hline
  {} & $time_{av}$ & $f_{best}$ & $n_{av}$ & $\tau_{av}$ \\
  \hline
  $n_{\min} = 0$ & 22.48 & -1.5919 & 9 & $10^6$ \\
  \hline
  $n_{\min} = 3$ & 30.76 & -3.816 & 11 & $10^5$ \\
  \hline
  $n_{\min} = 10$ & 53.91 & -4.1411 & 16.4 & $10^4$ \\
  \hline
\end{tabular}
\label{tab:CompressedModes_MuZero}
\end{table}

\begin{table} [ht!]
\caption{The results of numerical experiments for $10$ randomly generated starting points in the case $\nu = 0.2$.}
\begin{tabular}{| c | c | c | c | c |} 
  \hline
  {} & $time_{av}$ & $f_{best}$ & $n_{av}$ & $\tau_{av}$ \\
  \hline
  $n_{\min} = 0$ & 20.48 & 7.3279 & 9 & $10^6$ \\
  \hline
  $n_{\min} = 3$ & 234.51 & 0.3624 & 100 & $10^3$ \\
  \hline
  $n_{\min} = 10$ & 77.02 & -0.646 & 32.1 & 46 \\
  \hline
\end{tabular}
\label{tab:CompressedModes_MuPoint2}
\end{table}

We applied Algorithmic Pattern~\ref{alg:CCP_penalty} to problem \eqref{prob:CompressedModes} with two different values
of parameter $\nu$: $\nu = 0$ and $\nu = 0.2$. The method was initialized at 10 different starting points, randomly
generated with the use of the standard \textsc{Matlab} routine \texttt{rand}. The results of numerical experiments  
for each value of $\nu \in \{ 0, 0.2 \}$ and each $n_{\min} \in \{ 0, 3, 10 \}$ (see \eqref{eq:ModifiedPenaltyUpdate}) 
are given in Tables~\ref{tab:CompressedModes_MuZero} and \ref{tab:CompressedModes_MuPoint2}. We denote by $time_{av}$
the average run-time of the method in seconds, $f_{best}$ is the smallest value of the objective function for the given
choice of parameters, $n_{av}$ is the average number of iterations of the method, and $\tau_{av}$ is the average value
of the penalty parameter. The critical points of problem \eqref{prob:CompressedModes} with $\nu = 0$ and $\nu = 0.2$
corresponding to the best value of the objective function computed in our numerical experiments are depicted on
Fig.~\ref{fig:CompressedModes}, which clearly demonstrates the effect of parameter $\nu$ on the sparsity of solutions.

Let us note that in all cases for $\nu = 0.2$ and in most cases for $\nu = 0$, the best overall value of the objective
function for $n_{\min} = 0$ was greater than the values of the objective function at the points satisfying the stopping
criterion for $n_{\min} = 3$ (that is, the ``worst'' value for $n_{\min} = 3$ was better than the best value for 
$n_{\min} = 0$). Similarly, the best overall value of the objective function for $n_{\min} = 3$ was greater than the
values of the objective function at the points satisfying the stopping criterion in the case $n_{\min} = 10$. Thus, our
numerical experiments showed that an increase of the parameter $n_{\min}$ (i.e. the number of iterations during which
the method does not update the penalty parameter) allows the method to find a critical point with a better value of
the objective function, in accordance with the observation made in Example~\ref{ex:AnalyticalComputDCA}. Moreover, an
increase of $n_{\min}$ also reduces the final value of the penalty parameter used by the method. However, an increase of
$n_{\min}$ also increases computation time.

\subsection{Sphere packing on Grassmannian}

Recall that the real Grassmann manifold (or Grassmannian) is the smooth manifold of $k$-dimensional linear subspaces of
$\mathbb{R}^{\ell}$ for any given $k, \ell \in \mathbb{N}$ such that $k \le \ell$. Identifying a subspace with the
orhogonal projector onto this subspace, we can identify the Grassmannian with the set
\[
  \mathbf{Gr}(\ell, k) = \Big\{ P \in \mathbb{S}^{\ell} \Bigm| P^2 = P, \enspace \trace(P) = k \Big\}.
\]
The \textit{chordal distance} $\dist(P_1, P_2)$ on $\mathbf{Gr}(\ell, k)$ is defined by 
$\dist(P_1, P_2) = \sqrt{2} \| P_1 - P_2 \|_F$ (see, e.g. \cite{DirrHelmkeLageman,ConwayHardinSloane}).

The sphere packing problem on Grassmannian consists in finding $m \in \mathbb{N}$, $m > 1$, identical, non-overlapping
balls 
\[
  B_r(P_i) = \big\{ P \in \mathbf{Gr}(\ell, k) \mid \dist(P_i, P) < r \big\}, \quad
  P_i \in \mathbf{Gr}(\ell, k), \quad i \in \mathcal{I},
\]
where $\mathcal{I} = \{ 1, \ldots, m \}$, such that their radius is maximized. This problem can be formulated as the
following nonsmooth optimization problem:
\[
  \maximize_{(P_1, \ldots, P_m)} \min_{1 \le i < j \le m} \| P_i - P_j \|_F
  \quad \text{subject to} \quad P_i \in \textbf{Gr}(\ell, k), \quad i \in \mathcal{I}.
\]
To apply Algorithmic Pattern~\ref{alg:CCP_penalty} to this problem, we rewrite it as the following equivalent
minimization problem:
\begin{equation} \label{prob:SpherePacking}
\begin{split}
  &\minimize_{(P_1, \ldots, P_m) \in X} \enspace 
  f_0(P_1, \ldots, P_m) := \max_{1 \le i < j \le m} \Big( - \| P_i - P_j \|_F \Big)
  \\
  &\mathrm{s.t.} \enspace G_i(P_i) \preceq \mathbb{O}_{\ell}, \quad - G_i(P_i) \preceq \mathbb{O}_{\ell},
  \quad i \in \mathcal{I}, \quad (P_1, \ldots, P_m) \in Q, 
\end{split}
\end{equation}
where $X$ is the Cartesian product of $m$ copies of $\mathbb{S}^{\ell}$, $G_i(P_i) = P_i^2 - P_i$, and
\[
  Q = \Bigm\{ (P_1, \ldots, P_m) \in X \Bigm| \trace(P_i) = k, \: i \in \mathcal{I} \Big\}.
\]
One can readily check that the matrix-valued mappings $G_i$ are convex (in the order-theoretic sense), while the
objective function $f_0$ is DC and one can use the DC decomposition $f_0 = g_0 - h_0$ of this function with
\begin{align*}
  g_0(P_1, \ldots, P_m) &= \max_{1 \le i < j \le m} \big( h_0(P_1, \ldots, P_m) - \| P_i - P_j \|_F \big), 
  \\
  h_0(P_1, \ldots, P_m) &= \sum_{1 \le i < j \le m} \| P_i - P_j \|_F.
\end{align*}
Note, however, that even for relatively small $m$ the computation of values of the functions $g_0$ and $h_0$, as well
as their subgradients, is very expensive. Therefore, we also considered the following equivalent reformulation of
problem \eqref{prob:SpherePacking}, in which the max-function is replaced by the corresponding inequality constraints:
\begin{equation} \label{prob:SpherePackingInequal}
\begin{split}
  &\minimize_{(P_1, \ldots, P_m, r) \in X \times \mathbb{R}} \enspace r
  \quad \mathrm{subject~to} \enspace - h_{ij}(P_i, P_j, r) \le 0, \quad 1 \le i < j \le m,
  \\
  &G_i(P_i) \preceq \mathbb{O}_{\ell}, \quad - G_i(P_i) \preceq \mathbb{O}_{\ell},
  \quad i \in \mathcal{I}, \quad (P_1, \ldots, P_m) \in Q, 
\end{split}
\end{equation}
where $h_{ij}(P_i, P_j, r) := \| P_i - P_j \| + r$. Note that $h_{ij}$ are convex functions.

\begin{remark}
As in the case of the problem of computing compressed modes for variational problems, we transformed the nonlinear
convex constraints $G_i(P_i) \preceq s$ with $s \succeq \mathbb{O}_{\ell}$, arising on Step~2 of
Algorithmic Pattern~\ref{alg:CCP_penalty}, into the following equivalent linear matrix inequalities
\[
  \begin{pmatrix} P_i + s & P_i \\ P_i & I_{\ell} \end{pmatrix} \preceq \mathbb{O}_{2 \ell}
\]
with the use of the Schur Complement Lemma.
\end{remark}

\begin{table} [ht!]
\caption{The results of numerical experiments in the case $\ell = 10$, $k = 2$, and $m = 4$. Here $n$ is the number of
iterations and $\tau$ is the value of the penalty parameter.}
\begin{tabular}{| c | c | c | c | c |} 
  \hline
  {} & $time$ & $f_{best}$ & $n$ & $\tau$ \\
  \hline
  problem \eqref{prob:SpherePacking}, $n_{\min} = 0$ & 61.1 & -1.4044 & 9 & $10^5$ \\
  \hline
  problem \eqref{prob:SpherePackingInequal}, $n_{\min} = 0$ & 37.1 & -1.4036 & 7 & $10^5$ \\
  \hline
  problem \eqref{prob:SpherePacking}, $n_{\min} = 3$ & 58.1 & -1.4191 & 7 & 100 \\
  \hline
  problem \eqref{prob:SpherePackingInequal}, $n_{\min} = 3$ & 44.1 & -1.4186 & 8 & 1000 \\
  \hline
  problem \eqref{prob:SpherePacking}, $n_{\min} = 10$ & 106.9 & -1.5158 & 14 & 100 \\
  \hline
  problem \eqref{prob:SpherePackingInequal}, $n_{\min} = 10$ & 74.5 & -1.524 & 14 & 100 \\
  \hline
\end{tabular}
\label{tab:SpherePacking_Example}
\end{table}

\begin{table} [ht!]
\caption{The results of numerical experiments for the sphere packing problem on $\mathbf{Gr}(2k, k)$ with $m = 10$.}
\begin{tabular}{| c | c | c | c | c | c | c | c | c | c | c |} 
  \hline
  k & 3 & 4 & 5 & 6 & 7 & 8 & 9 & 10 & 11 & 12 \\
  \hline
  $time_{av}$ & 16.8 & 15.7 & 21.6 & 34.9 & 55.4 & 94.2 & 133.9 & 205.2 & 303.6 & 445.9 \\
  \hline
  $f_{best}$ & -1.5497 & -1.8672 & -2.1487 & -2.3887 & -2.6208 & -2.8084 & -3.007 & -3.1747 & -3.3497 & -3.4954 \\
  \hline
\end{tabular}
\label{tab:SpherePacking_Best}
\end{table}

We applied Algorithmic Pattern~\ref{alg:CCP_penalty} with the same randomly generated starting point and three
different values of $n_{\min} \in \{ 0, 3, 10 \}$ to two different problem formulations in order to determine the most
efficient way to apply Algorithmic Pattern~\ref{alg:CCP_penalty} to the sphere packing problem on Grassmannian. 
The results of a large number of numerical experiments showed that, when initialized at the same point, Algorithmic
Pattern~\ref{alg:CCP_penalty} in virtually all cases either (i) converges to the point with the same (up to the
tolerance specified in the stopping criterion) value of the objective function for both problem formulations or (ii) in
the case of the second problem formulation finds a point with the better value of the objective function than in the
case of the first one. Moreover, the run-time of Algorithmic Pattern~\ref{alg:CCP_penalty} for problem
\eqref{prob:SpherePackingInequal} was significantly smaller than its run-time for problem \eqref{prob:SpherePacking} and
the difference between the run-times grows very rapidly as $m$ increases. The results of numerical experiments in the
case $\ell = 10$, $k = 2$, and $m = 4$ (the dimension of the problem in this case is $220$), which are qualitatively the
same as results of numerical experiments for other values of parameters, are presented in
Table~\ref{tab:SpherePacking_Example}.

It should be noted that as in the case of the problem of computing compressed modes, the results of numerical
experiments showed that an increase of $n_{\min}$ always allows the method to find a critical point with the better
value of the objective function and reduces the value of the penalty parameter needed to find a point satisfying the
constraints up to a prespecified tolerance. However, an increase of $n_{\min}$ also increases the run-time of the
method.

Finally, to test the overall efficiency of the method, we applied Algorithmic Pattern~\ref{alg:CCP_penalty} with
$n_{\min} = 10$ to problem \eqref{prob:SpherePackingInequal} with $\ell = 2k$, $k \in \{ 3, 4, \ldots, 12 \}$, and 
$m = 10$. In all cases the penalty parameter was increased only once (i.e. the final value $t_* = 10 I_k$) and 
the method required between 10 and 15 iterations to find a point satisfying the stopping criterion. The results of
numerical experiments for $10$ randomly generated starting points are given in Table~\ref{tab:SpherePacking_Best}. 

\begin{remark}
{(i)~Let us note that the best value of the objective function for the sphere packing problem on $\mathbf{Gr}(16, 8)$
with $m = 10$ computed in our experiments is significantly better than the one given in
\cite[Fig.~2]{DirrHelmkeLageman}.
}

\noindent{(ii)~Let us point out that since the space $\mathbb{S}^{\ell}$ can be identified with 
$\mathbb{R}^{(\ell + 1)\ell / 2}$, the dimension of problem \eqref{prob:SpherePackingInequal} is equal to
$m (\ell + 1) \ell / 2 + 1$. In particular, in the case $\ell = 2 k$, $k = 12$, and $m = 10$, corresponding to the
largest problem that we solved, the dimension of problem \eqref{prob:SpherePackingInequal} is equal to $3001$. Note also
that the dimension of the convex subproblem on Step~2 of Algorithmic Pattern~\ref{alg:CCP_penalty} in this case is equal
to $9001$ due to the presence of the variables $s$ that are added to the right-hand side of matrix inequality
constraints in accordance with the theoretical scheme of Algorithmic Pattern~\ref{alg:CCP_penalty}.
}
\end{remark}

\section{Conclusions}

In this paper we developed a general theory of cone constrained DC optimization problems. The first part of the paper
was devoted to analysis of DC semidefinite programming problems. We studied two definition of DC matrix-valued mappings
(abstract and componentwise) and their interconnections. We proved that any DC matrix-valued map is componentwise DC and
demonstrated how one can compute a DC decomposition of several nonlinear semidefinite constraints appearing in
applications. We also constructed a DC decomposition of the maximal eigenvalue function, which allows one to apply
standard results and methods of inequality constrained DC optimization to problems with smooth and nonsmooth
componentwise DC semidefinite constraints. In the case of general cone constrained DC optimization problems, we obtained
local optimality conditions.

In the second part of the paper, we presented a detailed convergence analysis of the DCA for cone constrained DC
programs and its penalized version proposed in \cite{LippBoyd} (see also \cite{LeThiPhamDinh2014,LeThiPhamDinh2014b})
under the assumption that the concave part of the constraints is smooth. In particular, we obtained sufficient
conditions for the exactness of the penalty subproblem of the penalized version of the method and analyzed two types of
sufficient conditions for the convergence of this method to a feasible and critical point of a cone constrained DC
optimization problem from an infeasible starting point. The first type of sufficient conditions is the so-called a
priori conditions, which are based on general assumptions on the problem under consideration, while the second type is
the a posteriori conditions, which rely on some assumptions on a limit point of a sequence generated by an optimization
method. 

We also presented a simple example demonstrating that even if a feasible starting point is known, it might be reasonable
to use the penalized version of the method, since it is sometimes capable of finding deeper local minimum than the
standard method. 

In Section~\ref{sect:NumerExperiments}, applications of the exact penalty DCA to the problem of computing compressed
modes for variational problems and the sphere packing problem on Grassmannian were presented. The results of numerical
experiments confirmed the observations about the method made in Example~\ref{ex:AnalyticalComputDCA}. In particular,
they showed that it is advisable to let the exact penalty DCA to perform a certain number of iterations without
updating the penalty parameter to enable it to find a critical point with the better value of the objective function
(see the appendix).

The main results of our study pave the way for applications of DC optimization methods to various nonlinear
semidefinite programming problems and other nonconvex cone constrained optimization problems (such as nonconvex second
order cone and semi-infinite programming problems), as well as some nonlinear and nonsmooth optimization problems on
Riemannian manifolds (e.g. the Stiefel and the Grassmann manifolds).


\bibliographystyle{abbrv}  
\bibliography{Dolgopolik_bibl}

\begin{thebibliography}{10}

\bibitem{AbsilHosseini}
P.~Absil and S.~Hosseini.
\newblock A collection of nonsmooth {R}iemannian optimization problems.
\newblock In S.~Hosseini, B.~Mordukhovich, and A.~Uschmajew, editors, {\em
  Nonsmooth Optimization and Its Applications}, pages 1--15. Birkh\"{a}user,
  Cham, 2019.

\bibitem{AbsilMahony}
P.~A. Absil, R.~Mahony, and R.~Sepulchre.
\newblock {\em Optimization algorithms on matrix manifolds}.
\newblock Princeton University Press, Princeton, 2009.

\bibitem{AlizadehGoldfarb}
F.~Alizadeh and D.~Goldfarb.
\newblock Second-order cone programming.
\newblock {\em Math. Program.}, 95:3--51, 2003.

\bibitem{Auslender}
A.~Auslender.
\newblock An exact penalty method for nonconvex problems converging, in
  particular, nonlinear programming, semidefinite programming, and second-order
  cone programming.
\newblock {\em SIAM J. Optim.}, 25:1732--1759, 2015.

\bibitem{BenTalNemirovski}
A.~Ben-{T}al and A.~Nemirovski.
\newblock {\em Lectures on Modern Convex Optimization. Analysis, Algorithms,
  and Engineering Applications}.
\newblock SIAM, Philadelphia, 2001.

\bibitem{BianchiColesantiPucci}
G.~Bianchi, A.~Colesanti, and C.~Pucci.
\newblock On the second differentiability of convex surfaces.
\newblock {\em Geometriae Dedicata}, 60:39--48, 1996.

\bibitem{BonnansShapiro}
J.~F. Bonnans and A.~Shapiro.
\newblock {\em Perturbation Analysis of Optimization Problems}.
\newblock Springer, New York, 2000.

\bibitem{BoydVandenberghe}
S.~Boyd and L.~Vandenberghe.
\newblock {\em Convex Optimization}.
\newblock Cambridge University Press, Cambridge, 2004.

\bibitem{BurachikKayaPrice}
R.~S. Burachik, C.~Y. Kaya, and C.~J. Price.
\newblock A primal-dual penalty method via rounded weighted-$\ell_1$
  {L}agrangian duality.
\newblock {\em Optim.}, 71:3981--4017, 2022.

\bibitem{CanelasCarrasco}
A.~Canelas, M.~Carrasco, and J.~L\'{o}pez.
\newblock A feasible direction algorithm for nonlinear second-order cone
  programs.
\newblock {\em Optim. Meth. Softw.}, 34:1322--1341, 2019.

\bibitem{ChenMaSoZhang}
S.~Chen, S.~Ma, A.~{Man-{C}ho So}, and T.~Zhang.
\newblock Proximal gradient method for nonsmooth optimization over the
  {S}tiefel manifold.
\newblock {\em SIAM J. Optim.}, 30:210--239, 2020.

\bibitem{ChenJiYou}
W.~Chen, H.~Ji, and Y.~You.
\newblock An augmented {L}agrangian method for $\ell_1$-regularized
  optimization problems with orthogonality constraints.
\newblock {\em SIAM J. Sci. Comput.}, 38:570--592, 2016.

\bibitem{ConnGouldToint}
A.~R. Conn, \relax{N. I. M. Gould}, and P.~L. Toint.
\newblock {\em Trust-Region Methods}.
\newblock SIAM, Philadelphia, 2000.

\bibitem{ConwayHardinSloane}
J.~H. Conway, R.~H. Hardin, and N.~Sloane.
\newblock Packing lines, planes, ets.: {P}acking in {G}rassmanian spaces.
\newblock {\em Exp. Math.}, 5:139--159, 1996.

\bibitem{CorreaLopezPerezAros}
R.~Correa, M.~A. L\'{o}pez, and P.~P\'{e}rez-Aros.
\newblock Necessary and sufficient optimality conditions in {D}{C}
  semi-infinite programming.
\newblock {\em SIAM J. Optim.}, 31:837--865, 2021.

\bibitem{deOliveira2019}
W.~de~Oliveira.
\newblock Proximal bundle methods for nonsmooth {D}{C} programming.
\newblock {\em J. Glob. Optim.}, 75:523--563, 2019.

\bibitem{deOliveiraTcheou}
W.~de~Oliveira and M.~P. Tcheou.
\newblock An inertial algorithm for {D}{C} programming.
\newblock {\em Set-Valued Var. Anal.}, 27:895--919, 2019.

\bibitem{DirrHelmkeLageman}
G.~Dirr, U.~Helmke, and C.~Lageman.
\newblock Nonsmooth {R}iemannian optimization with applications to sphere
  packing and grasping.
\newblock In F.~Allg\"{u}wer, P.~Fleming, P.~Kokotovic, A.~B. Kurzhanski,
  H.~Kwakernaak, A.~Rantzer, J.~N. Tsitsiklis, F.~Bullo, and K.~Fujimoto,
  editors, {\em Lagrangian and Hamiltonian Methods for Nonlinear Control},
  pages 29--45. Springer, Berlin, Heidelberg, 2007.

\bibitem{DolgopolikExactPenalty}
M.~V. Dolgopolik.
\newblock A unifying theory of exactness of linear penalty functions.
\newblock {\em Optim.}, 65:1167--1202, 2016.

\bibitem{DolgopolikExactPenaltyII}
M.~V. Dolgopolik.
\newblock A unifying theory of exactness of linear penalty functions {I}{I}:
  parametric penalty functions.
\newblock {\em Optim.}, 66:1577--1622, 2017.

\bibitem{Dolgopolik_MultidimPen}
M.~V. Dolgopolik.
\newblock Exact penalty functions with multidimensional penalty parameter and
  adaptive penalty updates.
\newblock {\em Optim. Lett.}, 16:1281--1300, 2022.

\bibitem{DolgopolikExactPenaltyIII}
M.~V. Dolgopolik and A.~V. Fominyh.
\newblock Exact penalty functions for optimal control problems {I}: {M}ain
  theorem and free-endpoint problems.
\newblock {\em Optim. Control Appl. Method}, 40:1018--1044, 2019.

\bibitem{DurHorstLocatelli}
M.~D\"{u}r, R.~Horst, and M.~Locatelli.
\newblock Necessary and sufficient global optimality conditions for convex
  maximization revisited.
\newblock {\em J. Math. Analysis Appl.}, 217:637--649, 1998.

\bibitem{EdelmanThomas}
A.~Edelman, A.~A. Tom\'{a}s, and T.~S. Smith.
\newblock The geometry of algorithms with orthogonality constraints.
\newblock {\em SIAM J. Matrix Anal. Appl.}, 20:303--353, 1998.

\bibitem{Ferrer}
A.~Ferrer and J.~E. Mart\'{\i}nez-Legaz.
\newblock Improving the efficiency of {D}{C} global optimization methods by
  improving the {D}{C} representation of the objective function.
\newblock {\em J. Glob. Optim.}, 43:513--531, 2009.

\bibitem{Penlab}
J.~Fiala, M.~Ko\v{c}vara, and M.~Stingl.
\newblock \relax{PENLAB}: a \relax{MATLAB} solver for nonlinear semidefinite
  optimization.
\newblock {\em arXiv: 1311.5240}, 2013.

\bibitem{Gadhi}
N.~A. Gadhi.
\newblock Necessary optimality conditions for a nonsmooth semi-infinite
  programming problem.
\newblock {\em J. Glob. Optim.}, 74:161--168, 2019.

\bibitem{GaudiosoBagirov}
M.~Gaudioso, G.~Giallombardo, G.~Miglionico, and A.~M. Bagirov.
\newblock Minimizing nonsmooth {D}{C} functions via successive {D}{C}
  piecewise-affine approximations.
\newblock {\em J. Glob. Optim.}, 71:37--55, 2018.

\bibitem{GobernaLopez}
M.~A. Goberna and M.~A. L\'{o}pez, editors.
\newblock {\em Semi-Infinite Programming: Recent Advances}.
\newblock Kluwer Academic Publishers, Dordrecht, 2001.

\bibitem{GohSafonov96}
K.-C. Goh, M.~G. Safonov, and J.~H. Ly.
\newblock Robust synthesis via bilinear matrix inequalities.
\newblock {\em Int. J. Robust Nonlinear Control}, 6:1079--1095, 1996.

\bibitem{GohSafonov95}
K.-C. Goh, M.~G. Safonov, and G.~P. Papavassilopous.
\newblock Global optimization for the {B}iaffine {M}atrix {I}nequality problem.
\newblock {\em J. Glob. Optim.}, 7:365--380, 1995.

\bibitem{GoharyDavidson}
R.~H. Gohary and T.~N. Davidson.
\newblock Noncoherent {M}{I}{M}{O} communication: {G}rassmannian constellations
  and efficient detection.
\newblock {\em IEEE Trans. Inform. Theory}, 55:1176--1205, 2009.

\bibitem{GrantBoyd_CVX}
M.~Grand and S.~Boyd.
\newblock Graph implementations for nonsmooth convex programs.
\newblock In V.~Blondel, S.~Boyd, and H.~Kimura, editors, {\em Recent Advances
  in Learning and Control}, Lecture Notes in Control and Information Sciences,
  pages 95--110. Springer, London, 2008.

\bibitem{GrohsHosseini2}
P.~Grohs and S.~Hosseini.
\newblock Nonsmooth trust region algorithms for locally {L}ipschitz functions
  on {R}iemannian manifolds.
\newblock {\em IMA J. Numer. Analysis}, 36:1167--1192, 2016.

\bibitem{GrohsHosseini}
P.~Grohs and S.~Hosseini.
\newblock $\varepsilon$-subgradient algorithms for locally {L}ipschitz
  functions on {R}iemannian manifolds.
\newblock {\em Adv. Comput. Math.}, 42:333--360, 2016.

\bibitem{Hartman}
P.~Hartman.
\newblock On functions representable as a difference of convex functions.
\newblock {\em Pac. J. Math.}, 9:707--713, 1959.

\bibitem{HenrionTarbouriech}
D.~Henrion, S.~Tarbouriech, and M.~\v{S}ebek.
\newblock Rank-one \relax{LMI} approach to simultaneous stabilization of linear
  systems.
\newblock {\em Syst. Control Lett.}, 38:79--89, 1999.

\bibitem{HiriartUrruty85}
J.-B. Hiriart-{U}rruty.
\newblock Generalized differentiability/duality and optimization for problems
  dealing with differences of convex functions.
\newblock In J.~Ponstein, editor, {\em Convexity and Duality in Optimization},
  pages 37--70. Springer, Berlin, Heidelberg, 1985.

\bibitem{HiriartUrruty89}
J.-B. Hiriart-{U}rruty.
\newblock From convex optimization to nonconvex optimization. {N}ecessary and
  sufficient conditions for global optimality.
\newblock In F.~H. Clarke, V.~F. Dem'yanov, and F.~Giannessi, editors, {\em
  Nonsmooth Optimization and Related Topics}, pages 219--239. Springer, Boston,
  MA, 1989.

\bibitem{HiriartUrruty98}
J.-B. Hiriart-Urruty.
\newblock Conditions for global optimality 2.
\newblock {\em J. Glob. Optim.}, 13:349--367, 1998.

\bibitem{HorstThoai}
R.~Horst and N.~V. Thoai.
\newblock {D}{C} programming: Overview.
\newblock {\em J. Optim. Theory Appl.}, 103:1--43, 1999.

\bibitem{IoffeTihomirov}
A.~D. Ioffe and V.~M. Tihomirov.
\newblock {\em Theory of Extremal Problems}.
\newblock North-{H}olland, Amsterdam, 1979.

\bibitem{JokiBagirov2020}
K.~Joki and A.~M. Bagirov.
\newblock Bundle methods for nonsmooth {D}{C} optimization.
\newblock In A.~M. Bagirov, M.~Gaudioso, N.~Karmitsa, M.~M. M\"{a}kel\"{a}, and
  S.~Taheri, editors, {\em Numerical Nonsmooth Optimization. State of the Art
  Algorithms}, pages 263--296. Springer, Cham, 2020.

\bibitem{Kadison}
R.~V. Kadison.
\newblock Order properties of bounded self-adjoint operators.
\newblock {\em Proc. Amer. Math. Soc.}, 2:505--510, 1951.

\bibitem{Kanzi2011}
N.~Kanzi.
\newblock Necessary optimality conditions for nonsmooth semi-infinite
  programming problems.
\newblock {\em J. Glob. Optim.}, 49:713--725, 2011.

\bibitem{KatoFukushima}
H.~Kato and M.~Fukushima.
\newblock An {S}{Q}{P}-type algorithm for nonlinear second-order cone programs.
\newblock {\em Optim. Lett.}, 1:129--144, 2007.

\bibitem{KocvaraStingl}
M.~Ko\v{c}vara and M.~Stingl.
\newblock \relax{PENNON}: A code for convex nonlinear and semidefinite
  programming.
\newblock {\em Optim. Methods Softw.}, 18:317--333, 2003.

\bibitem{KronigPenney}
R.~L. Kronig and W.~G. Penney.
\newblock Quantum mechanics of electrons in crystal lattices.
\newblock {\em Proc. Royal Soc. London. Ser. A, containing papers of a
  mathematical and physical character}, 130:499--513, 1931.

\bibitem{KusraevKutateladze}
A.~G. Kusraev and S.~S. Kutateladze.
\newblock {\em Subdifferentials: Theory and Applications}.
\newblock Kluwer Academic Publishers, Dordrecht, 1995.

\bibitem{Lanckreit}
G.~R. Lanckreit and B.~K. Sriperumbudur.
\newblock On the convergence of the concave-convex procedure.
\newblock {\em Adv. Neural Inf. Process. Syst.}, 22:1759--1767, 2009.

\bibitem{Compleib}
F.~Leibfritz.
\newblock {C}{O}{M}{P}$l_e$ib: {C}{O}nstraint {M}atrix-optimization {P}roblem
  library --- a collection of test examples for nonlinear semidefinite
  programs, control system design and related problems.
\newblock Technical report, University of Trier, Department of Mathematics,
  2004.
\newblock Available at: http://www.compleib.de.

\bibitem{LeThiPhamDinh2014}
H.~A. {\relax Le Thi}, V.~N. Nuynh, and T.~{\relax Pham Dinh}.
\newblock {D}{C} programming and {D}{C}{A} for general {D}{C} programs.
\newblock In T.~\relax{van Do}, H.~A.~L. Thi, and N.~T. Nguyen, editors, {\em
  Advanced Computational Methods for Knowledge Engineering}, pages 15--35.
  Springer, Berlin, Heidelberg, 2014.

\bibitem{LeThiPhamDinh2018}
H.~A. {\relax Le Thi}, V.~N. Nuynh, and T.~{\relax Pham Dinh}.
\newblock Convergence analysis of difference-of-convex algorithm with
  subanalytic data.
\newblock {\em J. Optim. Theory Appl.}, 179:103--126, 2018.

\bibitem{LeThiPhamDinh2005}
H.~A. {\relax Le Thi} and T.~{\relax Pham Dinh}.
\newblock The {D}{C} (difference of convex functions) programming and {D}{C}{A}
  revisited with {D}{C} models of real world nonconvex optimization problems.
\newblock {\em Ann. Oper. Res.}, 133:23--46, 2005.

\bibitem{LeThiDinh2018}
H.~A. {\relax Le Thi} and T.~{\relax Pham Dinh}.
\newblock {D}{C} programming and {D}{C}{A}: thirty years of developments.
\newblock {\em Math. Program.}, 169:5--68, 2018.

\bibitem{PhamDinhLeThi96}
H.~A. {\relax Le Thi}, T.~{\relax Pham Dinh}, and L.~D. Muu.
\newblock Numerical solution for optimization over the efficient set by
  {D}.{C}. optimization algorithm.
\newblock {\em Oper. Res. Lett.}, 19:117--128, 1996.

\bibitem{LeThiPhamDinh2002}
H.~A. {\relax Le Thi}, T.~{\relax Pham Dinh}, and N.~V. Thoai.
\newblock Combination between global and local methods for solving an
  optimization problem over the efficient set.
\newblock {\em Eur. J. Oper. Res.}, 142:258--270, 2002.

\bibitem{LippBoyd}
T.~Lipp and S.~Boyd.
\newblock Variations and extension of the convex-concave procedure.
\newblock {\em Optim. Eng.}, 17:263--287, 2016.

\bibitem{Manton}
J.~H. Manton.
\newblock Optimization algorithms exploiting unitary constraints.
\newblock {\em IEEE Trans. Signal Process.}, 50:635--650, 2002.

\bibitem{MordukhovichNghia}
B.~S. Mordukhovich and T.~Nghia.
\newblock Nonsmooth cone-constrained optimization with applications to
  semi-infinite programming.
\newblock {\em Math. Oper. Res.}, 39:301--324, 2014.

\bibitem{NesterovNemirovski}
Y.~Nesterov and A.~Nemirovskii.
\newblock {\em Interior-Point Polynomial Algorithms in Convex Programming}.
\newblock SIAM, Philadelphia, 1994.

\bibitem{NiuDinh2014}
Y.-S. Niu and T.~P. Dinh.
\newblock {D}{C} programming approaches for {B}{M}{I} and {Q}{M}{I} feasibility
  problems.
\newblock In T.~\relax{van Do}, H.~Thi, and N.~Nguyen, editors, {\em Advanced
  Computational Methods for Knowledge Engineering.}, pages 37--63. Springer,
  Cham, 2014.

\bibitem{OzolinsLaiCaflischOsher}
V.~Ozolin\v{s}, R.~Lai, R.~Caflisch, and S.~Osher.
\newblock Compressed modes for variational problems in mathematics and physics.
\newblock {\em Proc. National Academy Sci.}, 110:18368--18373, 2013.

\bibitem{Papageorgiou}
N.~S. Papageorgiou.
\newblock Nonsmooth analysis on partially ordered vectors spaces: part 1 ---
  convex case.
\newblock {\em Pac. J. Math.}, 107:403--458, 1983.

\bibitem{PhamDinhLeThi97}
T.~{\relax Pham Dinh} and H.~A. {\relax Le Thi}.
\newblock Convex analysis approach to {D}{C} programming: theory, algorithms,
  and applications.
\newblock {\em Acta Math. Vietnamica}, 22:289--355, 1997.

\bibitem{PhamDinhLeThi98}
T.~{\relax Pham Dinh} and H.~A. {\relax Le Thi}.
\newblock D.{C}. optimization algorithms for solving the trust region
  subproblem.
\newblock {\em SIAM J. Optim.}, 8:476--505, 1998.

\bibitem{LeThiPhamDinh2014b}
T.~{\relax Pham Dinh} and H.~A. {\relax Le Thi}.
\newblock Recent advances in {D}{C} programming and {D}{C}{A}.
\newblock In N.~T. Nguyen and H.~A.~L. Thi, editors, {\em Transactions on
  Computational Intelligence {X}{I}{I}{I}}, pages 1--37. Springer, Berlin,
  Heidelberg, 2014.

\bibitem{Polak}
E.~Polak.
\newblock {\em Optimization: Algorithms and Consistent Approximations}.
\newblock Springer-Verlag, New York, 1997.

\bibitem{ReemtsenRuckmann}
R.~Reemtsen and J.-J. R\"{u}ckmann, editors.
\newblock {\em Semi-Infinite Programming}.
\newblock Kluwer Academic Publishers, Dordrecht, 1998.

\bibitem{CVXPackage}
\relax{CVX Research, Inc.}
\newblock {CVX}: Matlab software for disciplined convex programming, version
  2.2.
\newblock http://cvxr.com/cvx, 2020.

\bibitem{PhamDinh1986}
T.~\relax{Pham Dinh} and E.~B. Souad.
\newblock Algorithms for solving a class of nonconvex optimization problems.
  {M}ethods of subgradients.
\newblock In J.-B. Hiriart-{U}rruty, editor, {\em Fermat Days 85: Mathematics
  for Optimization. North-Holland Mathematics Studies. Vol. 129}, pages
  249--271. Norht-Holland, Amsterdam, 1986.

\bibitem{Robinson76}
S.~M. Robinson.
\newblock Regularity and stability for convex multivalued functions.
\newblock {\em Math. Oper. Res.}, 1:130--143, 1976.

\bibitem{Rockafellar}
R.~T. Rockafellar.
\newblock {\em Convex Analysis}.
\newblock Princeton University Press, Princeton, 1970.

\bibitem{Stingl}
M.~Stingl.
\newblock {\em On the solution of nonlinear semidefinite programs by augmented
  {L}agrangian methods}.
\newblock PhD thesis, Institute of Applied Mathematics II, Friedrech-Alexander
  University of Erlangen-Nuremberg, Erlangen, Germany, 2006.

\bibitem{Strekalovsky87}
A.~S. Strekalovsky.
\newblock On the problem of the global extremum.
\newblock {\em Sov. Math. Dokl.}, 35:194--198, 1987.

\bibitem{Strekalovsky98}
A.~S. Strekalovsky.
\newblock Global optimality conditions for nonconvex optimization.
\newblock {\em J. Glob. Optim.}, 12:415--434, 1998.

\bibitem{Strekalovsky_Collect}
A.~S. Strekalovsky.
\newblock Local search for nonsmooth {D}{C} optimization with {D}{C} equality
  and inequality constraints.
\newblock In A.~M. Bagirov, M.~Gaudioso, N.~Karmitsa, M.~M. M\"{a}kel\"{a}, and
  S.~Taheri, editors, {\em Numerical Nonsmooth Optimization. State of the Art
  Algorithms}, pages 229--262. Springer, Cham, 2020.

\bibitem{Strekalovsky_Conf}
A.~S. Strekalovsky.
\newblock On a global search in {D}.{C}. optimization problems.
\newblock In M.~Ja\'{c}imovi\'{c}, M.~Khachay, V.~Malkova, and M.~Posypkin,
  editors, {\em Optimization and Applications. OPTIMA 2019. Communications in
  Computer and Information Science}, pages 222--236. Springer, Cham, 2020.

\bibitem{Strekalovsky2021}
A.~S. Strekalovsky.
\newblock On global optimality conditions for d.c. minimization problems with
  d.c. constraints.
\newblock {\em J. Appl. Numer. Optim.}, 3:175--196, 2021.

\bibitem{Thera}
M.~Thera.
\newblock Subdifferential calculus for convex operators.
\newblock {\em J. Math. Anal. Appl.}, 80:78--91, 1981.

\bibitem{Todd2001}
M.~Todd.
\newblock Semidefinite optimization.
\newblock {\em Acta Numerica}, 10:515--560, 2001.

\bibitem{TohToddTutuncu}
K.~C. Toh, M.~J. Todd, and R.~H. T\"{u}tunc\"{u}.
\newblock {S}{D}{P}{T}3 --- a {M}atlab software package for semidefinite
  programming, {V}ersion 1.3.
\newblock {\em Optim. Methods Softw.}, 11:545--581, 1999.

\bibitem{TorBagKar}
A.~H. Tor, A.~Bagirov, and B.~Karas\"ozen.
\newblock Aggregate codifferential method for nonsmooth {D}{C} optimization.
\newblock {\em J. Comput. Appl. Math.}, 259:851--867, 2014.

\bibitem{Tung}
L.~T. Tung.
\newblock Karush-{K}uhn-{T}ucker optimality conditions for nonsmooth
  multiobjective semidefinite and semi-infinite programming.
\newblock {\em J. Appl. Numer. Optim.}, 1:63--75, 2019.

\bibitem{TutuncuTohTodd}
R.~H. T\"{u}tunc\"{u}, K.~C. Toh, and M.~J. Todd.
\newblock Solving semidefinite-quadratic-linear programs using {S}{D}{P}{T}3.
\newblock {\em Math. Program.}, 95:189--217, 2003.

\bibitem{Tuy86}
H.~Tuy.
\newblock A general deterministic approach to global optimization via {D}.{C}.
  programming.
\newblock In J.-B. Hiriart-{U}rruty, editor, {\em Fermat Days 85: Mathematics
  for Optimization. North-Holland Mathematics Studies. Vol. 129}, pages
  273--303. Norht-Holland, Amsterdam, 1986.

\bibitem{Tuy_book}
H.~Tuy.
\newblock {\em Convex Analysis and Global Optimization}.
\newblock Kluwer Academic Publishers, Dordrecht, 1998.

\bibitem{Tuy2003}
H.~Tuy.
\newblock On global optimality conditions and cutting plane algorithms.
\newblock {\em J. Optim. Theory Appl.}, 118:201--216, 2003.

\bibitem{AckooijDeOliveira2019}
W.~{van Ackooij} and W.~{de Oliveira}.
\newblock Non-smooth {D}{C}-constrained optimization: constraint qualification
  and minimizing methodologies.
\newblock {\em Optim. Methods Softw.}, 34:890--920, 2019.

\bibitem{AckooijDeOliveira2019b}
W.~van Ackooij and W.~de~Oliveira.
\newblock Nonsmooth and nonconvex optimization via approximate
  difference-of-convex decompositions.
\newblock {\em J. Optim. Theory Appl.}, 182:49--80, 2019.

\bibitem{deOliveira2021}
W.~van Ackooij, S.~Demassey, P.~Javal, H.~Morais, W.~de~Oliveira, and
  B.~Swaminathan.
\newblock A bundle method for nonsmooth dc programming with application to
  chance-constrained problems.
\newblock {\em Comput. Optim. Appl.}, 78:451--490, 2021.

\bibitem{YamashitaYabe2015}
H.~Yamashita and H.~Yabe.
\newblock A primal-dual interior point method for nonlinear optimization over
  second-order cones.
\newblock {\em Optim. Meth. Softw.}, 24:407--426, 2009.

\bibitem{YamashitaYabe}
H.~Yamashita and H.~Yabe.
\newblock A survey of numerical methods for nonlinear semidefinite programming.
\newblock {\em J. Oper. Res. Soc. Japan}, 58:24--60, 2015.

\bibitem{Yuille}
A.~L. Yuille and A.~Rangarajan.
\newblock The concave-convex procedure.
\newblock {\em Neural Comput.}, 15:915--936, 2003.

\bibitem{Zhang2013}
Q.~Zhang.
\newblock A new necessary and sufficient global optimality condition for
  canonical {D}{C} problems.
\newblock {\em J. Glob. Optim.}, 55:559--577, 2013.

\bibitem{ZhengTse}
L.~Zheng and D.~Tse.
\newblock Communication on the {G}rassmann manifold: {A} geometric approach to
  the noncoherent multiple-antenna channel.
\newblock {\em IEEE Trans. Inform. Theory}, 48:359--383, 2002.

\bibitem{ZhenYang2007}
X.~Y. Zheng and X.~Yang.
\newblock Lagrange multipliers in nonsmooth semi-infinite optimization
  problems.
\newblock {\em Math. Oper. Res.}, 32:168--181, 2007.

\end{thebibliography}

\section*{Appendix. A modification of Algorithmic Pattern~\ref{alg:CCP_penalty}}

The results of our numerical experiments clearly demonstrate that the increase of the number $n_{\min}$ of iterations
during which the method does not update the penalty parameter allows the exact penalty DCA to find a deeper local
minimum or a critical point with the better value of the objective function. Pushing this idea to the extreme and
following the theoretical scheme of primal-dual penalty methods \cite{BurachikKayaPrice,Dolgopolik_MultidimPen}, one can
propose the following modification of Algorithmic Pattern~\ref{alg:CCP_penalty} corresponding to the case 
$n_{\min} = \infty$, whose general scheme is given in Algorithmic Pattern~\ref{alg:PrimalDual}.

\begin{algorithm}[ht!]	\label{alg:PrimalDual}
\caption{Primal-dual penalty DCA}

\noindent\textbf{Initialization.} {Choose an initial point $x_0 \in Q$, penalty parameter $t_0 \succ_{K^*} 0$, the
maximal norm of the penalty parameter $\tau_{\max} > 0$, $\mu > 1$, infeasibility tolerance $\varkappa \ge 0$, and set
$n := 0$ and $k := 0$.}

\noindent\textbf{Step~1.} {Compute $v_n \in \partial h_0(x_n)$ and $D H(x_n)$.}

\noindent\textbf{Step~2.} {Set the value of $(x_{n + 1}, s_{n + 1})$ to a solution of the convex problem
\begin{align*}
  &\minimize_{(x, s)} \enspace g_0(x) - \langle v_n, x \rangle + \langle t_k, s \rangle \\
  &\text{subject to} \enspace G(x) - H(x_n) - D H(x_n)(x - x_n) \preceq_K s, \quad s \succeq_K 0, \quad x \in Q.
\end{align*}
If a stopping criterion is satisfied, put $y_k := x_{n + 1}$ and go to \textbf{Step 3}. Otherwise, put $n := n + 1$ and
go to \textbf{Step 1}.
}

\noindent\textbf{Step~3.} {If $\| s_{n + 1} \| \le \varkappa$ or $\mu \| t_k \| \ge \tau_{\max}$, \textbf{Stop}.
Otherwise, define $t_{k + 1} = \mu t_k$, put $x_0 := y_k$, $n := 0$, $k := k + 1$, and go to \textbf{Step~1}.
}
\end{algorithm}

The idea of the modified method consists in applying Algorithmic Pattern~\ref{alg:CCP_penalty} with a fixed value of the
penalty parameter and updating the penalty parameter only after the method has found a point satisfying a stopping
criterion, which is similar in spirit to primal-dual augmented Lagrangian methods. The results of numerical experiments
presented in Section~\ref{sect:NumerExperiments} indicate that it is likely that Algorithmic
Pattern~\ref{alg:PrimalDual} significantly outperforms Algorithmic Pattern~\ref{alg:CCP_penalty} in terms of the quality
of computed local solutions (i.e. it is likely to be able to find critical points with better values of the objective
function than Algorithmic Pattern~\ref{alg:CCP_penalty}), but at the cost of much greater run-time.

It seems possible to extend the convergence analysis of Algorithmic Pattern~\ref{alg:CCP_penalty} presented in this
paper to the case of Algorithmic Pattern~\ref{alg:PrimalDual}. Such an extension, as well as numerical evaluation of
Algorithmic Pattern~\ref{alg:PrimalDual}, is an interesting subject for future research.

\end{document}